\documentclass[11pt,reqno,twoside]{amsart}
\usepackage{graphicx}
\usepackage{amssymb}
\usepackage{epstopdf}
\usepackage[asymmetric,top=3.5cm,bottom=4.3cm,left=3.1cm,right=3.1cm]{geometry}
\usepackage{booktabs} % for much better looking tables
\usepackage{array} % for better arrays (eg matrices) in maths
\usepackage{paralist} % very flexible & customisable lists (eg. enumerate/itemize, etc.)
\usepackage{verbatim} % adds environment for commenting out blocks of text & for better verbatim
\usepackage{subfig} % make it possible to include more than one captioned figure/table in a single float
\usepackage{tabularx}
\usepackage{amsmath,amsfonts,amsthm,mathrsfs,amssymb,cite}
\usepackage[usenames]{color}
\usepackage{bm}

\newtheorem{thm}{Theorem}[section]
\newtheorem{cor}[thm]{Corollary}
\newtheorem{lem}{Lemma}[section]
\newtheorem{prop}{Proposition}[section]
\theoremstyle{definition}
\newtheorem{defn}{Definition}[section]
\theoremstyle{remark}

\newtheorem{rem}{Remark}[section]
\numberwithin{equation}{section}

\newcommand{\calA}{\mathcal{A}}
\newcommand{\calS}{\mathcal{S}}
\newcommand{\calK}{\mathcal{K}}
\newcommand{\calM}{\mathcal{M}}
\newcommand{\calL}{\mathcal{L}}
\newcommand{\bx}{\mathbf{x}}
\newcommand{\by}{\mathbf{y}}
\newcommand{\bE}{\mathbf{E}}
\newcommand{\bH}{\mathbf{H}}
\newcommand{\bI}{\mathbf{I}}
\newcommand{\bT}{\mathbf{T}}
\newcommand{\bN}{\mathbf{N}}
\newcommand{\bK}{\mathbf{K}}
\newcommand{\bvarphi}{\bm{\varphi}}
\newcommand{\bpsi}{\bm{\psi}}
\newcommand{\bnu}{\bm{\nu}}
\newcommand{\bXi}{\bm{\Xi}}
\newcommand{\rmi}{\mathrm{i}}
\newcommand{\be}{\mathbf{e}}
\title[Surface plasmon resonance and invisibility cloaking]{Analysis of electromagnetic scattering from plasmonic inclusions beyond the quasi-static approximation and applications}

\author{Hongjie Li}
\address{Department of Mathematics, Hong Kong Baptist University, Kowloon Tong, Hong Kong SAR, China.}
\email{hongjie$_{-}$li@yeah.net}

\author{Shanqiang Li}
\address{Department of Mathematics, Harbin University of Science and Technology, Harbin, China.}
\email{lishanqiang@gmail.com}

\author{Hongyu Liu}
\address{Department of Mathematics, Hong Kong Baptist University, Kowloon Tong, Hong Kong SAR, China.}
\email{hongyu.liuip@gmail.com; hongyuliu@hkbu.edu.hk}

\author{Xianchao Wang}
\address{Department of Mathematics, Harbin Institute of Technology, Harbin, China}
\email{xcwang90@gmail.com}

\begin{document}
\maketitle

\begin{abstract}

This paper is concerned with the analysis of time-harmonic electromagnetic scattering from plasmonic inclusions in the finite frequency regime beyond the quasi-static approximation. The electric permittivity and magnetic permeability in the inclusions are allowed to be negative-valued. Using layer potential techniques for the full Maxwell system, the scattering problem is reformulated into a system of integral equations. We derive the complete eigensystem of the involved matrix-valued integral operator within spherical geometry. As applications, we construct two types of plasmonic structures such that one can induce surface plasmon resonances within finite frequencies and the other one can produce invisibility cloaking. It is particularly noted that the cloaking effect is a newly found phenomenon and is of different nature from those existing ones for plasmonic structures in the literature. The surface plasmon resonance result may find applications in electromagnetic imaging. 

\medskip

\medskip


\end{abstract}

\section{Introduction}

Plasmonic materials are artificially engineered metamaterials that exhibit negative or left-handed optical properties. V. G. Veselago was the first to theoretically investigate the electrodynamics of substances with simultaneously negative permittivity $\varepsilon$ and magnetic permeability $\mu$ \cite{Ves}. J. B. Pendry extended the work of V. G. Veselago and gave the theoretical construction of a perfect lens \cite{JP}. Negative refractive indexed materials have been fabricated in the lab \cite{Le}. Plasmonic materials have been proposed for many striking applications including biomedicine imaging \cite{ADM, APRY18, JLEE06}, near-field microscopy \cite{BL2002,BLBCL2006}, molecular recognition \cite{HF04}, heat therapeutic applications \cite{ARRu18, BGQu10, FDLi15} and invisibility cloaking \cite{Alu,GWM3,GWM5,La}. 

The mathematical study of plasmonic materials is interesting and has attracted significant attentions in recent years in the literature. The negativity of the material parameters breaks the ellipticity of the governing wave systems, and hence various resonance phenomena can be induced. At the resonant frequencies, the incident light couples with the surface plasmons to create a significant field enhancement around the plasmonic inclusions which can be much shorter in wavelength and known as the surface plasmon polaritons. The surface plasmon resonance can be used to enhance the wave imaging and it has been investigated for electromagnetic waves in \cite{ADM} for plasmonic nanoparticles, that is, for electromagnetic waves in the quasi-static regime. For a certain plasmonic structures of the core-shell form, the so-called anomalous localized resonance and invisibility cloaking effect can be induced; see \cite{Acm13,Ack13,Ack14,Klsap,AKL,LLL} for the related study in electrostatics and \cite{AKKY, AKM1, KM, DLL, LiLiu2d, LiLiu3d, LLL2} in elastostatics. 
In all of the aforementioned mathematical studies, the quasistatic approximation for the wave scattering problems is technically made. In \cite{KLO}, the authors show that without the quasistatic approximation, plasmonic resonance does not occur for the Helmholtz system if the plasmon parameter follows a specific trajectory in the complex plane towards the singular state as the loss parameter goes to zero. In a recent paper by two of the authors of the present article \cite{LiLiuhelm}, it is constructed that if the plasmon parameter follows certain designated trajectories in the complex plane toward the singular states, then plasmonic resonances can still occur for the Helmholtz system in the finite-frequency regime beyond the quasi-static approximation.

In this article, inspired by the work \cite{LiLiuhelm}, we analyze the electromagnetic scattering from plasmon inclusions at finite frequencies beyond the quasi-static approximation. By using the layer potential techniques, we reduce the electromagnetic scattering problem to a system of boundary integral equations. Then we derive the complete eigensystem of the involved matrix-valued boundary integral operator within spherical geometry. The integral operator is non-selfadjoint without the quasi-static assumption and the derivation of the spectral system is rather delicate and subtle, even within the spherical geometry. Moreover, in the current study, one needs to know the exact spectral information of the integral operator and hence the extension to general geometries is unpractical. As applications, we give two general constructions of plasmonic structures such that one can induce surface plasmon resonances and the other one can produce invisibility cloaking. We also provide concrete examples within the Drude model for the Maxwell system. Similar to \cite{ADM} in the quasi-static case, our result on surface plasmon resonance can find applications in electromagnetic wave imaging in the finite regime. It is worth of emphasizing that the cloaking effect discovered in our study is of different nature from the earlier mentioned ones induced by the anomalous localized resonance, and it critically depends on the frequencies of the electromagnetic waves. Furthermore, our numerical results show that not only the plasmonic structures, but also any inhomogeneous inclusions embedded in the structures can be cloaked.

Next, we present the mathematical setup of our study by describing the electromagnetic scattering problem and the time-harmonic Maxwell system. Let $\epsilon_m$ and $\mu_m$ be two positive constants which, respectively, characterize the electric permittivity and magnetic permeability of the uniformly homogeneous space $\mathbb{R}^3$. Let $D$ be a bounded domain in $\mathbb{R}^3$ with a smooth boundary $\partial D$ and a connected complement $\mathbb{R}^3\backslash\overline{D}$, which signifies the plamsonic inclusion embedded in the homogeneous space. The permittivity and permeability of $D$ are specified by $\epsilon_c$ and $\mu_c$. At this point, we only assume that $\epsilon_c$ and $\mu_c$ are complex numbers with nonnegative imaginary parts, that is, $\epsilon_c, \mu_c\in\mathbb{C}$ and $\Im \epsilon_c, \Im \mu_c\geq 0$. Eventually, we shall construct that $\epsilon_c$ and $\mu_c$ might be certain complex numbers with negative real parts. Let $\omega\in\mathbb{R}_+$ denote the frequency of the electromagnetic waves. Set
\begin{equation}\label{eq:a1}
 k_m=\omega\sqrt{\epsilon_m\mu_m}, \quad k_c=\omega\sqrt{\epsilon_c\mu_c}\quad \mbox{with} \quad \Re k_c>0, \  \Im k_c<0,
\end{equation}
and
\begin{equation}\label{eq:a2}
 \epsilon_D=\epsilon_m \chi(\mathbb{R}^3\backslash \overline{D})+\epsilon_c\chi(D), \quad \mu_D=\mu_m \chi(\mathbb{R}^3\backslash \overline{D})+\mu_c\chi(D),
\end{equation}
where and also in what follows $\chi$ denotes the characteristic function. Let $(\bE^i, \bH^i)$ be the incident electromagnetic field, and it satisfies the following Maxwell system in the homogeneous space,
 \begin{equation}\label{eq:a3}
  \begin{split}
     \nabla\times \bE^i-\mathrm{i}\omega\mu_m \bH^i=& 0 \qquad \mbox{in} \ \ \mathbb{R}^3,\medskip \\
     \nabla\times \bH^i+\mathrm{i}\omega\epsilon_m \bE^i =& 0\qquad \mbox{in} \ \ \mathbb{R}^3,
  \end{split}
 \end{equation}
 where $\mathrm{i}:=\sqrt{-1}$ is the imaginary unit. The electromagnetic scattering associated with the medium configuration described in \eqref{eq:a2} due to the incident field $(\bE^i, \bH^i)$ in \eqref{eq:a3} is governed by the following transmission problem of the Maxwell system,
\begin{equation}\label{eq:original_max}
 \begin{split}
      \nabla\times \bE-\mathrm{i}\omega\mu_D \bH=& 0 \qquad \mbox{in} \ \ \mathbb{R}^3\backslash\partial D, \medskip\\
     \nabla\times \bH+\mathrm{i}\omega\epsilon_D\bE =& 0 \qquad \mbox{in} \ \ \mathbb{R}^3\backslash\partial D, \medskip\\
     \bnu\times \bE|_+-\bnu\times \bE|_-  =&  \bnu\times \bH|_+-\bnu\times \bH|_-=0 \quad \mbox{on} \ \ \partial D,
 \end{split}
\end{equation}
subject to the Silver-M\"{u}ler radiation condition,
\begin{equation}\label{eq:a4}
 \lim_{|\bx|\rightarrow\infty} |\bx|\left(\sqrt{\mu_m}(\bH-\bH^i)(\bx)\times\frac{\bx}{|\bx|} - \sqrt{\epsilon_m}(\bE-\bE^i)(\bx)\right)
\end{equation}
which holds uniformly in $\bx/|\bx|\in\mathbb{S}^2$. The radiation condition \eqref{eq:a4} characterizes the outgoing nature of the scattered wave fields $\bE^s:=\bE-\bE^i$ and $\bH^s:=\bH-\bH^i$. In this paper, we are concerned with studying the quantitative behaviours of the scattered fields $\bE^s, \bH^s$ and their relationships with the proper choices of the plasmonic parameters $\epsilon_c, \mu_c$ and the incident fields $\bE^i, \bH^i$. Specifically, as discussed earlier, we are mainly interested in the two phenomena of surface plasmon resonance and invisibility cloaking. It is emphasized that we do not impose the quasi-static assumption in our study, namely, there is no such assumption that $\omega\cdot\mathrm{diam}(D)\ll 1$, and as also discussed earlier, this is one aspect that distinguishes the current study from the existing ones in the literature. 

The rest of the paper is organized as follows. In Section 2, we present the integral reformulation of the electromagnetic scattering problem using layer potential techniques. Section 3 is devoted to the spectral analysis of the integral operators derived in Section 2. In Section 4, we consider the surface plasmon resonance results and in Section 5, we consider the invisibility results. 
%The paper is concluded with some relevant discussions in Section 6. 

\section{Integral formulation of the Maxwell system}

In this section, using the layer potential techniques, we present the integral formulation of the electromagnetic scattering problem \eqref{eq:a1}--\eqref{eq:a4}. We begin with the introduction of some Sobolev spaces and potential theory for the Maxwell system for the subsequent use.

Let $D$ be a bounded domain in $\mathbb{R}^3$ with a smooth boundary $\partial D$ and a connected complement $\mathbb{R}^3\backslash\overline{D}$. We use $\nabla_{\partial D}$, $\nabla_{\partial D}\cdot$ and $\triangle_{\partial D}$ to respectively signify the surface gradient, surface divergence and Laplace-Beltrami operator on the surface $\partial D$. Denote by
\[
  L^2_T(\partial D):=\{\bvarphi\in L^2(\partial D)^3, \bnu\cdot\bvarphi=0\},
\]
where and also in what follows $\bnu$ signifies the exterior unit normal vector to the boundary of a domain concerned. Let $H^s(\partial D)$ be the usual Sobolev space of order $s\in\mathbb{R}$ on $\partial D$. Introduce the following function spaces
\[
  TH(\mathrm{div},\partial D):=\{\bvarphi\in L^2_{T}(\partial D):\nabla_{\partial D}\cdot \bvarphi\in L^2(\partial D)\},
\]
and
\[
 TH(\mathrm{curl},\partial D):=\{\bvarphi\in L^2_{T}(\partial D):\nabla_{\partial D}\cdot (\bvarphi\times\bnu)\in L^2(\partial D) \}.
\]
Define the vectorial surface curl by
\[
  \vec{\mathrm{curl}}_{\partial D}\varphi=\nabla_{\partial D}\varphi\times \bnu,
\]
for $\varphi\in L^2(\partial D)$ and one can verify  that
\begin{equation}\label{eq:lap}
  \nabla_{\partial D}\cdot\nabla_{\partial D}=\triangle_{\partial D}\quad\mbox{and}\quad
 \nabla_{\partial D}\cdot \vec{\mathrm{curl}}_{\partial D}=0.
\end{equation}
Next, we recall that the fundamental outgoing solution $G(\bx,\by,k)$ to the Helmholtz operator $\triangle + k^2$ in $\mathbb{R}^3$ is given by
\begin{equation}\label{eq:funda_sol}
 G(\bx,\by,k)=-\frac{e^{\rmi k|\bx-\by|}}{4\pi|\bx-\by|},\quad\bx, \by\in\mathbb{R}^3,\ \ \bx\neq\by. 
\end{equation}
For a density $\bvarphi\in TH(\mathrm{div},\partial D)$, we introduce the following vectorial single layer potential operator,
\begin{equation}\label{eq:vec_single}
  \calA_D^k[\bvarphi](\bx):=\int_{\partial D} G(\bx,\by,k)\bvarphi(\by)ds(\by), \quad \bx\in\mathbb{R}^3.
\end{equation}
For a scalar density $\varphi\in L^2(\partial D)$, the single layer potential operator is defined similarly by
\begin{equation}
  \calS_D^k[\varphi](\bx):=\int_{\partial D} G(\bx,\by,k) \varphi(\by)ds(\by), \quad \bx\in\mathbb{R}^3.
\end{equation}
We shall also need the following boundary integral operators,
\begin{equation}
\begin{split}
  \widetilde{\calS}_D^k[\varphi]: L^2(\partial D) & \rightarrow H^1(\partial D)  \\
      \varphi & \mapsto \widetilde{\calS}_D^k[\varphi](\bx)=\int_{\partial D} G(\bx,\by,k) \varphi(\by)ds(\by), \quad \bx\in\partial D;\\
  \left(\calK_D^{k}\right)^*[\varphi]: L^2(\partial D) & \rightarrow L^2(\partial D)  \\
      \varphi & \mapsto \left(\calK_D^{k}\right)^*[\varphi](\bx)=\int_{\partial D} \frac{\partial G(\bx,\by,k)}{\partial \bnu_{\bx}}\varphi(\by)ds(\by), \quad \bx\in\partial D;\\
  \left(\bK_D^{k}\right)^*[\bvarphi]: L^2(\partial D)^3 & \rightarrow L^2(\partial D)^3  \\
      \bvarphi & \mapsto \left(\bK_D^{k}\right)^*[\bvarphi](\bx)=\int_{\partial D} \frac{\partial G(\bx,\by,k)}{\partial \bnu_{\bx}}\bvarphi(\by)ds(\by), \quad \bx\in\partial D;\\
  \widetilde{\calA}_D^k[\bvarphi]: L^2(\partial D)^3 & \rightarrow H^1(\partial D)^3  \\
      \bvarphi & \mapsto \widetilde{\calA}_D^k[\bvarphi](\bx)=\int_{\partial D} G(\bx,\by,k) \bvarphi(\by)ds(\by), \quad \bx\in\partial D;
\end{split}
\end{equation}
and
\begin{equation}\label{eq:def_m}
\begin{split}
  \calM_D^{k}[\bvarphi]: & TH(\mathrm{div},\partial D) \rightarrow TH(\mathrm{div},\partial D)  \\
      \bvarphi & \mapsto \calM_D^{k}[\bvarphi](\bx)=\int_{\partial D} \bnu_{\bx}\times\nabla_{\bx}\times G(\bx,\by,k)\bvarphi(\by) ds(\by), \quad \bx\in\partial D;
\end{split}
\end{equation}
\begin{equation}\label{eq:def_l}
\begin{split}
  \calL_D^{k}[\bvarphi]: & TH(\mathrm{div},\partial D) \rightarrow TH(\mathrm{div},\partial D)  \\
      \bvarphi & \mapsto \calL_D^{k}[\bvarphi](\bx)=\bnu_{\bx}\times\left(k^2 \widetilde{\calA}^k_D[\bvarphi](\bx) + \nabla\widetilde{\calS}_D^k[\nabla_{\partial D}\cdot \bvarphi] \right), \quad \bx\in\partial D.
\end{split}
\end{equation}
There holds the following jump formula for $\bvarphi\in L^2_{T}(\partial D)$ (cf. \cite{Jcn}),
\begin{equation}\label{eq:jump}
  \left( \bnu\times\nabla\times\calA_D^k \right)^{\pm}=\mp\frac{\bvarphi}{2} + \calM_D^k[\bvarphi] \quad \mbox{on} \quad \partial D,
\end{equation}
where
\[
 \left( \bnu_{\bx}\times\nabla\times\calA_D^k \right)^{\pm}(\bx)=\lim_{t\rightarrow 0^+} \bnu_{\bx}\times\nabla\times\calA_D^k[\bvarphi](\bx\pm t\bnu_{\bx}), \quad \forall\bx\in\partial D.
\]
In what follows, similar to \eqref{eq:jump}, the subscripts $\pm$ signify the limits from the outside and inside of $D$, respectively.

With the above preparations, we are in a position to derive the integral formulation of the electromagnetic scattering problem \eqref{eq:a1}--\eqref{eq:a4}. Using the vectorial single layer potential operator  \eqref{eq:vec_single}, the solution to \eqref{eq:a1}--\eqref{eq:a4} can be be given by the following integral ansatz (cf. \cite{ADM}),
\begin{equation}\label{eq:layer_solution}
  \bE(\bx)=\left\{
             \begin{array}{ll}
               \bE^i(\bx)+\mu_m\nabla\times\calA_D^{k_m}[\bpsi]+\nabla\times\nabla\times\calA_D^{k_m}[\bvarphi](\bx), & \bx\in\mathbb{R}^3\backslash\overline{D},\medskip \\
               \mu_c\nabla\times\calA_D^{k_c}[\bpsi]+\nabla\times\nabla\times\calA_D^{k_c}[\bvarphi](\bx), & \bx\in D,
             \end{array}
           \right.
\end{equation}
and
\begin{equation}
  \bH(\bx)= \frac{1}{\rmi\omega\mu_D}(\nabla\times\bE)(\bx) \quad \bx\in\mathbb{R}^3\backslash\partial D.
\end{equation}
By matching the boundary condition (the third condition in \eqref{eq:original_max}) and using the jump formula \eqref{eq:jump}, the pair $(\bpsi,\bvarphi)\in TH(\mathrm{div},\partial D)\times TH(\mathrm{div},\partial D)$ is the solution to the following system of integral equations,
\begin{equation}\label{eq:operator_eq}
  (\mathbb{I} + \mathbb{K}) \bm{\Phi} = \mathbf{F},
\end{equation}
where
\[
 \begin{split}
   \mathbb{I} &= \left[
               \begin{array}{cc}
                 \frac{\mu_c+\mu_m}{2} I & 0\medskip \\
                 0 & \left( \frac{k_c^2}{2\mu_c} +\frac{k_m^2}{2\mu_m} \right)I \\
               \end{array}
             \right], \\
     \mathbb{K}&=\left[
              \begin{array}{cc}
                \mu_c \calM_D^{k_c} - \mu_m \calM_D^{k_m} & \calL_D^{k_c}-\calL_D^{k_m}\medskip \\
                \calL_D^{k_c}-\calL_D^{k_m} & \frac{k_c^2}{\mu_c} \calM_D^{k_c} - \frac{k_m^2}{\mu_m} \calM_D^{k_m}\\
              \end{array}
            \right],
 \end{split}
\]

\[
 \bm{\Phi}=\left[
             \begin{array}{c}
               \bpsi \medskip\\
               \bvarphi \\
             \end{array}
           \right]
\quad \mbox{and} \quad
 \mathbf{F}=\left[
             \begin{array}{c}
               \bnu\times \bE^i \medskip\\
               \rmi\omega  \bnu\times \bH^i \\
             \end{array}
           \right].
\]
Here and also in what follows, $I$ signifies the identity operator.

Based on the integral formulation \eqref{eq:operator_eq}, we next give the formal definitions on the surface plasmon resonance and invisibility cloaking effect associated with the scattering problem \eqref{eq:a1}--\eqref{eq:a4}. At this point, we suppose that the operator $(\mathbb{I}+\mathbb{K})$ possesses an eigen-system $\{\bXi_j\}_{j=1}^{+\infty}$ and $\{\tau_j\}_{j=1}^{\infty}$ such that 
\[
 (\mathbb{I}+\mathbb{K})[\bXi_j]=\tau_j \bXi_j. 
\]
Moreover, it is assumed that the set of functions $\{\bXi_j\}_{j=1}^{+\infty}$ is complete in the space $L^2_{T}(\partial D)^2$. Definitely, we shall construct such an eigen-system in what follows under a certain circumstance and this is one of the main technical contributions of the current study. Using such an eigen-system, the integral system \eqref{eq:operator_eq} can be solved via the spectral resolution by firstly representing the RHS term  $\mathbf{F}$ in \eqref{eq:operator_eq} as $ \mathbf{F}=\sum_{j=1}^{+\infty}f_j \bXi_j$,
and then by straightforward calculations that
\begin{equation}\label{eq:s1}
 \bm{\Phi}=\sum_{j=1}^{+\infty} \frac{f_j}{\tau_j}\bXi_j.
\end{equation}
Thus from \eqref{eq:layer_solution}, the scattered wave field in $\mathbb{R}^3\backslash\overline{D}$ to the problem \eqref{eq:a1}--\eqref{eq:a4} can be given as follows
\begin{equation}\label{eq:swf}
 \bE^s(\bx)=\sum_{j=1}^{+\infty}\frac{f_j}{\tau_j}\left(\mu_m\nabla\times\calA_D^{k_m}[\bXi_{j,1}]+\nabla\times\nabla\times\calA_D^{k_m}[\bXi_{j,2}](\bx)\right), \quad \bx\in\mathbb{R}^3\backslash\overline{D},
\end{equation}
where
\[
 \bXi_j= \left[
             \begin{array}{c}
               \bXi_{j,1} \\
               \bXi_{j,2} \\
             \end{array}
           \right].
\]
Based on the representation of the scattered wave field $ \bE^s$ in \eqref{eq:swf}, we introduce the following definitions for the phenomena of the surface plasmon resonance and invisibility cloaking.
\begin{defn}\label{def:reson}
We say that the surface plasmon resonance occurs if there exists an eigenvalue $\tau_j$ of the operator $(\mathbb{I}+\mathbb{K})$ such that $|\tau_j|\ll 1$ with $f_j\neq 0$.
\end{defn}
\begin{defn}\label{def:invi}
We say that the object $D$ is invisible to the incident wave $(\bE^i, \bH^i)$ , if all the eigenvalues $\tau_j$ of the operator $(\mathbb{I}+\mathbb{K})$ with $f_j\neq 0$ satisfy
\[
 |\tau_j|\gg1.
\]
\end{defn}

In the situation as described in Definition~\ref{def:reson}, by \eqref{eq:swf}, one readily sees that the corresponding scattered filed encounters a certain blowup. Indeed, in such a case, the significant field enhancement is mainly confined to the surface of the electromagnetic inclusion $D$ and it is referred to as the surface plasmon resonance; see our numerical illustrations in Figs.~1 and 2 in what follows. On the other hand, in the situation as described in Definition~\ref{def:invi}, one can easily verify that the corresponding scattered field should be nearly vanishing and hence invisibility cloaking effect can be observed. Indeed, in our subsequent numerical study, it is observed that not only the plasmonic structure but also certain inhomogeneous inclusions embedded in the structure are invisible under the wave impinging. Clearly, the occurrence of the surface plasmon resonance and the invisibility cloaking critically depends on the appropriate choice of the material constitution of the plasmonic inclusion, namely $(D; \epsilon_c, \mu_c)$ and the source $\mathbf{F}$, namely $(\mathbf{E}^i, \bH^i)$. In what follows, we shall first construct the eigen-system for the operator $(\mathbb{I}+\mathbb{K})$ with the spherical geometry. It is emphasized that the spectral structure of the operator $(\mathbb{I}+\mathbb{K})$ with $\omega=0$, namely in the quasi-static regime, is known \cite{ADM}. In the finite-frequency regime with $\omega\sim 1$, the operator $(\mathbb{I}+\mathbb{K})$ is no longer self-adjoint and the construction is radically more challenging. Moreover, as discussed in Section 1, for the surface plasmon resonance and the invisibility cloaking, one shall require the exact spectral information of the operator $(\mathbb{I}+\mathbb{K})$, and hence we mainly restrict to the spherical geometry. Nevertheless, as can be seen that even in such a case, the corresponding derivation is highly nontrivial. After the derivation of the exact spectral properties of the operator $(\mathbb{I}+\mathbb{K})$, we can then use the corresponding result to construct the suitable plasmonic structures that can induce surface plasmon resonance or invisibility cloaking.

\section{Spectral analysis of the integral operators}

In this section, we derive the spectral properties of the integral operators introduced earlier within the spherical geometry. To that end, we assume that $D$ is a central ball of radius $R\in\mathbb{R}_+$ in what follows. Let $\mathbb{S}$ be the unit sphere and we introduce the following vectorial spherical harmonics of order $n\in\mathbb{N}_0:=\mathbb{N}\cup\{0\}$,
\begin{equation}\label{eq:vec_sphe_harm}
  \begin{split}
    \bI_n =&\ \nabla_{\mathbb{S}}Y_{n+1}+(n+1)Y_{n+1}\bnu, \quad\,  n\geq 0, \\
    \bT_n =&\ \nabla_{\mathbb{S}}Y_{n}\times\bnu, \quad\qquad\qquad\qquad\ \,  n\geq 1,\\
    \bN_n =&\ -\nabla_{\mathbb{S}}Y_{n-1}+ nY_{n-1}\bnu, \qquad\;\;\,  n\geq 1,
  \end{split}
\end{equation}
where $\nabla_{\mathbb{S}}$ is the surface gradient on the unit sphere $\mathbb{S}$ and $Y_n$ is the spherical harmonics. It is noted that the set $(\nabla_{\mathbb{S}}Y_{n},\nabla_{\mathbb{S}}Y_{n}\times\bnu )_{n\in\mathbb{N}}$ forms an orthogonal basis of $L^2_{T}(\mathbb{S})$ (cf. \cite{Jcn}). Let $j_n(t)$ and $h_n^{(1)}(t)$ be, respectively, the spherical Bessel and Hankel functions of order $n\in \mathbb{N}_0$. The following lemma shows the spectral properties of the operator $\left(\calK_D^{k}\right)^*$ and $\widetilde{\calS}_D^k$ (cf.\cite{LiLiuhelm}).
\begin{lem}\label{thm:eigenvalue_k-star}
Suppose that $k$ and $R$ satisfy the following condition
\begin{equation}\label{eq:assumption_kR}
  j_m(kR)\neq j_{m+2}(kR) \quad \mbox{for} \quad m\geq 0.
\end{equation}
Then it holds that
\begin{equation}
  \left(\calK_D^{k}\right)^*[Y_n](x)=\lambda_{n,k,R} Y_n(\hat{x}) \quad \mbox{with} \quad n\geq 0,
\end{equation}
where
\begin{equation}\label{eq:eigenvalue_lambda_m}
  \lambda_{n,k,R}=\frac{1}{2}-\rmi k^2 R^2 j_n^{\prime}(kR) h_n^{(1)}(kR)=-\frac{1}{2}-\rmi k^2 R^2 j_n(kR) h_n^{(1)\prime}(kR).
\end{equation}
On the boundary $\partial D$, we have
\[
   \widetilde{\calS}_D^k[Y_n](x)=\chi_{n,k,R} Y_n(\hat{x}), \quad x\in \partial D,
\]
 where
\[
  \chi_{n,k,R}=-\rmi k R^2 h_n^{(1)}(kR) j_n(kR).
\]
\end{lem}

Using the vectorial spherical harmonics in \eqref{eq:vec_sphe_harm} and Lemma \ref{thm:eigenvalue_k-star}, one has by direct calculations that
\[
\left\{
  \begin{array}{ll}
    \left(\bK_D^{k}\right)^*[\nabla_{\mathbb{S}}Y_{n}+nY_{n}\bnu] = \lambda_{n-1,k,R}(\nabla_{\mathbb{S}}Y_{n}+nY_{n}\bnu) , & n\geq 1,\medskip \\
    \left(\bK_D^{k}\right)^*[-\nabla_{\mathbb{S}}Y_{n}+ (n+1)Y_{n}\bnu] = \lambda_{n+1,k,R}(-\nabla_{\mathbb{S}}Y_{n}+ (n+1)Y_{n}\bnu), & n\geq 0,
  \end{array}
\right.
\]
and
\[
\left\{
  \begin{array}{ll}
    \widetilde{\calA}_D^k[\nabla_{\mathbb{S}}Y_{n}+nY_{n}\bnu] = \chi_{n-1,k,R}(\nabla_{\mathbb{S}}Y_{n}+nY_{n}\bnu) , & n\geq 1,\medskip \\
    \widetilde{\calA}_D^k[-\nabla_{\mathbb{S}}Y_{n}+ (n+1)Y_{n}\bnu] = \chi_{n+1,k,R}(-\nabla_{\mathbb{S}}Y_{n}+ (n+1)Y_{n}\bnu), & n\geq 0.
  \end{array}
\right.
\]
By solving the above two systems, one can further show that
\begin{equation}\label{eq:spectrum_vec_K}
  \left\{
  \begin{array}{ll}
    \left(\bK_D^{k}\right)^*[\nabla_{\mathbb{S}}Y_{n}] = \pi_{1,n}\nabla_{\mathbb{S}}Y_{n} + \pi_{2,n}Y_{n}\bnu,\medskip \\
    \left(\bK_D^{k}\right)^*[Y_{n}\bnu] = \pi_{3,n}\nabla_{\mathbb{S}}Y_{n} + \pi_{4,n}Y_{n}\bnu,
  \end{array}
\right.
\end{equation}
with
\begin{equation}
\begin{split}
&  \pi_{1,n}=\frac{(n+1)\lambda_{n-1,k,R}+n\lambda_{n+1,k,R}}{2n+1},\quad \pi_{2,n}=\frac{n(n+1)}{2n+1}(\lambda_{n-1,k,R}-\lambda_{n+1,k,R}),\medskip\\
& \pi_{3,n}=\frac{\lambda_{n-1,k,R}-\lambda_{n+1,k,R}}{2n+1}, \hspace*{1.9cm} \pi_{4,n}=\frac{(n+1)\lambda_{n+1,k,R}+n\lambda_{n-1,k,R}}{2n+1};
\end{split}
\end{equation}
and
\begin{equation}\label{eq:spectrum_vec_A}
  \left\{
  \begin{array}{ll}
    \widetilde{\calA}_D^k[\nabla_{\mathbb{S}}Y_{n}] = \sigma_{1,n}\nabla_{\mathbb{S}}Y_{n} + \sigma_{2,n}Y_{n}\bnu,\medskip \\
    \widetilde{\calA}_D^k[Y_{n}\bnu] = \sigma_{3,n}\nabla_{\mathbb{S}}Y_{n} + \sigma_{4,n}Y_{n}\bnu,
  \end{array}
\right.
\end{equation}
with
\[
\begin{split}
&  \sigma_{1,n}=\frac{(n+1)\chi_{n-1,k,R}+n\chi_{n+1,k,R}}{2n+1},\quad \sigma_{2,n}=\frac{n(n+1)}{2n+1}(\chi_{n-1,k,R}-\chi_{n+1,k,R}),\medskip\\
 & \sigma_{3,n}=\frac{\chi_{n-1,k,R}-\chi_{n+1,k,R}}{2n+1}, \hspace*{1.9cm} \sigma_{4,n}=\frac{(n+1)\chi_{n+1,k,R}+n\chi_{n-1,k,R}}{2n+1}.
\end{split}
\]
With these preparations, we next derive the spectral properties of the operators $\calM_D^k$ and $\calL_D^{k}$. It is noted that $\calM_D^k$ and $\calL_D^{k}$ are entries of the matrix operator $\mathbb{K}$, and their spectral properties form a critical ingredient for our subsequent derivation of the spectral system for the operator $\mathbb{I}+\mathbb{K}$. 

\begin{prop}\label{prop:spec_m}
The orthogonal base functions of $L^2_{T}(\mathbb{S})$, $\nabla_{\mathbb{S}}Y_{n}$ and $\nabla_{\mathbb{S}}Y_{n}\times\bnu $ with $n\geq 1$, are eigenfunctions for the operator $\calM_D^{k}$ defined in \eqref{eq:def_m}, and one has
 \begin{equation}\label{eq:ll1}
 \calM_D^{k}[\nabla_{\mathbb{S}}Y_{n}\times\bnu]=m_{1,n}^k \nabla_{\mathbb{S}}Y_{n}\times\bnu\ \ \mbox{and}\ \  \calM_D^{k}[\nabla_{\mathbb{S}}Y_{n}]=m_{2,n}^k\nabla_{\mathbb{S}}Y_{n},
\end{equation}
with
\begin{equation}\label{eq:ll2}
 m_{1,n}^k=\left( \pi_{1,n} + \frac{1}{R}(\sigma_{1,n}-n(n+1)\sigma_{3,n})  \right)\quad\mbox{and}\quad m_{2,n}^k=\left( \lambda_{n,k,R} + \frac{\chi_{n,k,R}}{R} \right)
\end{equation}
where $\pi_{1,n}$,  $\sigma_{1,n}$ and $\sigma_{3,n}$, and, $\lambda_{n,k,R}$ and $\chi_{n,k,R}$  are given in \eqref{eq:spectrum_vec_K}, \eqref{eq:spectrum_vec_A} and Lemma \ref{thm:eigenvalue_k-star}, respectively.
\end{prop}
\begin{proof}
With the help of vector calculus, the identity $\partial_{\bx_1}G(\bx,\by,k)=-\partial_{\by_1}G(\bx,\by,k)$ and the definition of $\calM_D^k$ given in \eqref{eq:def_m}, one can find that
\begin{equation}\label{eq:m1}
\begin{split}
  \calM_D^{k}[\bvarphi] & =\int_{\partial D} \bnu_{\bx}\times\nabla_{\bx}\times G(\bx,\by,k)\bvarphi(\by) ds(\by) \\
                        & = -\bnu_{\bx}\times \int_{\partial D}\nabla_{\by}G(\bx,\by,k)\times \bvarphi(\by) ds(\by) \\
                        & = -\bnu_{\bx}\times \int_{\partial D}((\nabla_{\by}G(\bx,\by,k)\cdot\bnu_\by)\bnu_\by +\nabla_{\partial D_{\by}}G(\bx,\by,k) )\times \bvarphi(\by) ds(\by) \\
                        & = -\bnu_{\bx}\times \int_{\partial D}((\nabla_{\bx}G(\bx,\by,k)\cdot\bnu_\bx)\bnu_\by +\nabla_{\partial D_{\by}}G(\bx,\by,k) )\times \bvarphi(\by) ds(\by),
\end{split}
\end{equation}
where the last identity follows from 
\[
 \nabla_{\by}G(\bx,\by,k)\cdot\bnu_\by =\nabla_{\bx}G(\bx,\by,k)\cdot\bnu_\bx \quad \mbox{for} \quad \bx, \by\in \partial D.
\]
For $\bvarphi=\nabla_{\mathbb{S}}Y_{n}\times\bnu $, one has
\begin{equation}\label{eq:p1}
 \begin{split}
     & -\bnu_{\bx}\times \int_{\partial D}(\nabla_{\bx}G(\bx,\by,k)\cdot\bnu_\bx)\bnu_\by\times\bvarphi ds(\by) \\
   = & -\bnu_{\bx}\times\left(\bK_D^{k}\right)^*[\nabla_{\mathbb{S}}Y_{n}] = \pi_{1,n}\nabla_{\mathbb{S}}Y_{n} \times  \bnu,
 \end{split}
\end{equation}
which follows from \eqref{eq:spectrum_vec_K}, and the following fact
\[
 \begin{split}
     & \int_{\partial D} \nabla_{\partial D_{\by}}G(\bx,\by,k) \times \bvarphi(\by) ds(\by) \\
   = & -\int_{\partial D} (\nabla_{\partial D_{\by}}G(\bx,\by,k)\cdot \nabla_{\mathbb{S}}Y_{n})\bnu_{\by}ds(\by)\\
   = & -\sum_{j=1}^3\be_j \int_{\partial D} (\nabla_{\partial D_{\by}}G(\bx,\by,k)\cdot \nabla_{\mathbb{S}}Y_{n})(\bnu_{\by}\cdot \be_j)ds(\by)\\
   = &  \sum_{j=1}^3\be_j \int_{\partial D} G(\bx,\by,k) \nabla_{\partial D}\cdot(\nabla_{\mathbb{S}}Y_{n}(\bnu_{\by}\cdot \be_j))ds(\by) \\
   = &  \sum_{j=1}^3\be_j \int_{\partial D} G(\bx,\by,k) (\nabla_{\partial D}\cdot\nabla_{\mathbb{S}}Y_{n}(\bnu_{\by}\cdot \be_j) + \nabla_{\mathbb{S}}Y_{n}\cdot \nabla_{\partial D}(\bnu_{\by}\cdot \be_j) )ds(\by) \\
   = & \frac{1}{R}\left(-n(n+1)\widetilde{\calA}_D^k[Y_n\bnu] + \widetilde{\calA}_D^k[\nabla_{\mathbb{S}}Y_{n}]\right),
 \end{split}
\]
where $\be_j$, $j=1,2,3$, are the unit coordinate vectors in $\mathbb{R}^3$ and the last identity follows from 
\[
  \sum_{j=1}^3\be_j  \left(\nabla_{\mathbb{S}}Y_{n}\cdot \nabla_{\partial D}(\bnu_{\by}\cdot \be_j)\right)=\frac{1}{R} \nabla_{\mathbb{S}}Y_{n}.
\]
From \eqref{eq:spectrum_vec_A}, one has that
\begin{equation}\label{eq:p2}
\begin{split}
 &-\bnu_{\bx}\times \int_{\partial D} \nabla_{\partial D_{\by}}G(\bx,\by,k) \times (\nabla_{\mathbb{S}}Y_{n}\times\bnu) ds(\by)\\
 =&\frac{1}{R}(\sigma_{1,n}-n(n+1)\sigma_{3,n}) \nabla_{\mathbb{S}}Y_{n} \times  \bnu.
 \end{split}
\end{equation}
Then by combining \eqref{eq:m1}, \eqref{eq:p1} and \eqref{eq:p2}, one can directly verify that
\[
 \calM_D^{k}[\nabla_{\mathbb{S}}Y_{n}\times\bnu]=\left( \pi_{1,n} + \frac{1}{R}(\sigma_{1,n}-n(n+1)\sigma_{3,n})  \right)\nabla_{\mathbb{S}}Y_{n}\times\bnu.
\]

For $\bvarphi=\nabla_{\mathbb{S}}Y_{n}=\bnu\times(\nabla_{\mathbb{S}}Y_{n}\times\bnu)$,
by \eqref{eq:vec_sphe_harm} and Theorem \ref{thm:eigenvalue_k-star}, one has that
\begin{equation}\label{eq:p3}
 \begin{split}
     & -\bnu_{\bx}\times \int_{\partial D}(\nabla_{\bx}G(\bx,\by,k)\cdot\bnu_\bx)\bnu_\by\times\bvarphi ds(\by) \\
   = & \bnu_{\bx}\times\left(\bK_D^{k}\right)^*[\nabla_{\mathbb{S}}Y_{n}\times\bnu] = \lambda_{n,k,R}\nabla_{\mathbb{S}}Y_{n}.
 \end{split}
\end{equation}
Furthermore, one can deduce that
\[%\label{eq:p4}
 \begin{split}
     & \int_{\partial D} \nabla_{\partial D_{\by}}G(\bx,\by,k) \times (\bnu_{\by}\times(\nabla_{\mathbb{S}}Y_{n}\times\bnu_{\by})) ds(\by) \\
   = &  \int_{\partial D} (\nabla_{\partial D_{\by}}G(\bx,\by,k)\cdot (\nabla_{\mathbb{S}}Y_{n}\times\bnu_{\by}))\bnu_{\by}ds(\by)\\
   = &  \sum_{j=1}^3 \be_j\int_{\partial D} (\nabla_{\partial D_{\by}}G(\bx,\by,k)\cdot (\nabla_{\mathbb{S}}Y_{n}\times\bnu_{\by}))(\bnu_{\by}\cdot \be_j)ds(\by)
   \end{split}
   \]
   \begin{equation}\label{eq:p4}
 \begin{split}
   = & -\sum_{j=1}^3 \be_j\int_{\partial D} G(\bx,\by,k) \nabla_{\partial D}\cdot((\nabla_{\mathbb{S}}Y_{n}\times\bnu_{\by})(\bnu_{\by}\cdot \be_j))ds(\by) \\
   = & -\sum_{j=1}^3 \be_j\int_{\partial D} G(\bx,\by,k) ((\nabla_{\partial D}\cdot(\nabla_{\mathbb{S}}Y_{n}\times\bnu_{\by}))(\bnu_{\by}\cdot \be_j)\\
   & + (\nabla_{\mathbb{S}}Y_{n}\times\bnu_{\by})\cdot \nabla_{\partial D}(\bnu_{\by}\cdot \be_j))ds(\by) \\
   = & -\sum_{j=1}^3 \be_j\int_{\partial D} G(\bx,\by,k) (  (\nabla_{\mathbb{S}}Y_{n}\times\bnu_{\by})\cdot \nabla_{\partial D}(\bnu_{\by}\cdot \be_j))ds(\by) \\
   = & -\frac{1}{R}\left( \widetilde{\calA}_D^k[\nabla_{\mathbb{S}}Y_{n}\times\bnu]\right).
 \end{split}
\end{equation}
In the derivation of \eqref{eq:p4}, we have used the fact that
\[
 \nabla_{\partial D}\cdot(\nabla_{\mathbb{S}}Y_{n}\times\bnu_{\by})=0,
\]
which follows from the identity \eqref{eq:lap}. Thus from Lemma \ref{thm:eigenvalue_k-star}, one has that
\begin{equation}\label{eq:p5}
 -\bnu_{\bx}\times \int_{\partial D} \nabla_{\partial D_{\by}}G(\bx,\by,k) \times (\nabla_{\mathbb{S}}Y_{n}\times\bnu) ds(\by)=\frac{\chi_{n,k,R}}{R} \nabla_{\mathbb{S}}Y_{n}.
\end{equation}
Finally, with the help of \eqref{eq:m1}, \eqref{eq:p3} and \eqref{eq:p5}, one can verify that
\[
 \calM_D^{k}[\nabla_{\mathbb{S}}Y_{n}]=\left(\lambda_{n,k,R} + \frac{\chi_{n,k,R}}{R} \right)\nabla_{\mathbb{S}}Y_{n}.
\]

The proof is complete.
\end{proof}

\begin{prop}\label{prop:spec_l}
For the operator $\calL_D^{k}$ defined in \eqref{eq:def_l}, and  $\nabla_{\mathbb{S}}Y_{n}$ and $\nabla_{\mathbb{S}}Y_{n}\times\bnu $ with $n\geq 1$, there hold the following identities,
  \[
 \calL_D^{k}[\nabla_{\mathbb{S}}Y_{n}\times\bnu] = l_{1,n}^k \nabla_{\mathbb{S}}Y_{n}\ \ \mbox{and}\ \  \calL_D^{k}[\nabla_{\mathbb{S}}Y_{n}] = l_{2,n}^k \nabla_{\mathbb{S}}Y_{n}\times\bnu,
\]
with
\[
  l_{1,n}^k=k^2 \chi_{n,k,R}\quad\mbox{and}\quad  l_{2,n}^k=\left( \frac{n(n+1)}{R^2}\chi_{n,k,R} -k^2 \sigma_{1,n} \right), 
\]
where $\sigma_{1,n}$ and $\chi_{n,k,R}$  are given in \eqref{eq:spectrum_vec_A} and Theorem \ref{thm:eigenvalue_k-star}, respectively.
\end{prop}
\begin{proof}
  Recall that
\begin{equation}\label{eq:l}
 \calL_D^{k}[\bvarphi]=\bnu_{\bx}\times\left(k^2 \widetilde{\calA}^k_D[\bvarphi](\bx) + \nabla\widetilde{\calS}_D^k[\nabla_{\partial D}\cdot \bvarphi] \right).
\end{equation}
For $\bvarphi=\nabla_{\mathbb{S}}Y_{n}\times\bnu $, one can show that
\[
 \bnu_{\bx}\times k^2 \widetilde{\calA}^k_D[\nabla_{\mathbb{S}}Y_{n}\times\bnu] = k^2 \chi_{n,k,R} \nabla_{\mathbb{S}}Y_{n}.
\]
The identity in \eqref{eq:lap} gives 
\[
 \nabla_{\partial D}\cdot (\nabla_{\mathbb{S}}Y_{n}\times\bnu)=0,
\]
and thus from the last two identities, one has 
\[
 \calL_D^{k}[\nabla_{\mathbb{S}}Y_{n}\times\bnu] = k^2 \chi_{n,k,R} \nabla_{\mathbb{S}}Y_{n}.
\]
For $\bvarphi=\nabla_{\mathbb{S}}Y_{n}$, by \eqref{eq:spectrum_vec_A}, one can show
\[
 \bnu_{\bx}\times k^2 \widetilde{\calA}^k_D[\nabla_{\mathbb{S}}Y_{n}] = -k^2 \sigma_{1,n} \nabla_{\mathbb{S}}Y_{n}\times\bnu. 
\]
As for another term in \eqref{eq:l}, one has that
\[
 \begin{split}
     & \bnu_{\bx}\times \nabla\calS_D^k[\nabla_{\partial D}\cdot \nabla_{\mathbb{S}}Y_{n}] = \bnu_{\bx}\times \nabla\calS_D^k[\triangle_{\mathbb{S}}Y_{n}/R] \\
   = & \frac{-n(n+1)}{R}\bnu_{\bx}\times \nabla\calS_D^k[Y_{n}] = \frac{n(n+1)}{R^2}\chi_{n,k,R} \nabla_{\mathbb{S}}Y_{n}\times\bnu,
 \end{split}
\]
which follows from the identity \eqref{eq:lap}, Lemma \ref{thm:eigenvalue_k-star}, and the fact (cf. \cite{Jcn})
\[
 \triangle_{\mathbb{S}}Y_{n}=-n(n+1) Y_{n}.
\]
Therefore, there holds
\[
 \calL_D^{k}[\nabla_{\mathbb{S}}Y_{n}] = \left( \frac{n(n+1)}{R^2}\chi_{n,k,R} -k^2 \sigma_{1,n} \right)\nabla_{\mathbb{S}}Y_{n}\times\bnu.
\]

The proof is complete.
\end{proof}

After achieving the spectral systems for the operators $\calM_D^{k}$ and $\calL_D^{k}$, respectively, in Propositions \ref{prop:spec_m} and \ref{prop:spec_l}, we proceed to derive the spectral system of the operator $\mathbb{I}+\mathbb{K}$ defined in \eqref{eq:operator_eq}. We have

\begin{thm}\label{thm:spectrum_final}
  The eigenvalues and their corresponding eigenfunctions for the operator $\mathbb{I}+\mathbb{K}$ given in \eqref{eq:operator_eq} are given as follows with $n\in\mathbb{N}$,
\begin{equation}\label{eq:spectrum}
 \begin{split}
     & (\mathbb{I}+\mathbb{K}) [\bXi_{1,n}] = \tau_{1,n}\bXi_{1,n}, \quad (\mathbb{I}+\mathbb{K}) [\bXi_{2,n}] = \tau_{2,n}\bXi_{2,n}, \\
     & (\mathbb{I}+\mathbb{K}) [\bXi_{3,n}] = \tau_{3,n}\bXi_{3,n}, \quad (\mathbb{I}+\mathbb{K}) [\bXi_{4,n}] = \tau_{4,n}\bXi_{4,n},
 \end{split}
\end{equation}
where
\begin{equation}\label{eq:spectrum1}
 \begin{split}
     & \tau_{1,n}=\alpha_1(l_{1,n}^{k_c} - l_{1,n}^{k_m}) + \frac{k_c^2}{\mu_c} m_{2,n}^{k_c} - \frac{k_m^2}{\mu_m} m_{2,n}^{k_m} + \frac{k_c^2}{2\mu_c} +\frac{k_m^2}{2\mu_m}, \\
     & \tau_{2,n}=\alpha_2(l_{1,n}^{k_c} - l_{1,n}^{k_m}) + \frac{k_c^2}{\mu_c} m_{2,n}^{k_c} - \frac{k_m^2}{\mu_m} m_{2,n}^{k_m} + \frac{k_c^2}{2\mu_c} +\frac{k_m^2}{2\mu_m}, \\
     & \tau_{3,n}=\alpha_3(l_{2,n}^{k_c} - l_{2,n}^{k_m}) + \frac{k_c^2}{\mu_c} m_{1,n}^{k_c} - \frac{k_m^2}{\mu_m} m_{1,n}^{k_m} + \frac{k_c^2}{2\mu_c} +\frac{k_m^2}{2\mu_m}, \\
     & \tau_{4,n}=\alpha_4(l_{2,n}^{k_c} - l_{2,n}^{k_m}) + \frac{k_c^2}{\mu_c} m_{1,n}^{k_c} - \frac{k_m^2}{\mu_m} m_{1,n}^{k_m} + \frac{k_c^2}{2\mu_c} +\frac{k_m^2}{2\mu_m},
 \end{split}
\end{equation}
and
\[
 \begin{split}
     & \bXi_{1,n}= \left[
                 \begin{array}{c}
                   \alpha_1 \nabla_{\mathbb{S}}Y_{n}\times\bnu \\
                    \nabla_{\mathbb{S}}Y_{n} \\
                 \end{array}
               \right], \quad
 \bXi_{2,n}= \left[
                 \begin{array}{c}
                   \alpha_2 \nabla_{\mathbb{S}}Y_{n}\times\bnu \\
                    \nabla_{\mathbb{S}}Y_{n} \\
                 \end{array}
               \right],  \\
     & \bXi_{3,n}= \left[
                 \begin{array}{c}
                   \alpha_3 \nabla_{\mathbb{S}}Y_{n} \\
                    \nabla_{\mathbb{S}}Y_{n}\times\bnu \\
                 \end{array}
               \right], \hspace*{.85cm}
 \bXi_{4,n}= \left[
                 \begin{array}{c}
                   \alpha_4 \nabla_{\mathbb{S}}Y_{n} \\
                    \nabla_{\mathbb{S}}Y_{n}\times\bnu \\
                 \end{array}
               \right],
 \end{split}
\]
with
\[
 \begin{split}
     & \alpha_1=\frac{k_m^2 (2m_{2,n}^{k_m}-1) \mu_c +\mu_m(\mu_c(\mu_c + 2 m_{1,n}^{k_c}\mu_c +\mu_m-2m_{1,m}^{k_m}\mu_m ) - k_c^2 (2m_{2,n}^{k_c}+1) )  + \sqrt{\beta_1} }{4(l_{1,n}^{k_c} - l_{1,n}^{k_m})\mu_c\mu_m}, \\
     & \alpha_2=\frac{k_m^2 (2m_{2,n}^{k_m}-1) \mu_c +\mu_m(\mu_c(\mu_c + 2 m_{1,n}^{k_c}\mu_c +\mu_m-2m_{1,m}^{k_m}\mu_m ) - k_c^2 (2m_{2,n}^{k_c}+1) )  - \sqrt{\beta_1} }{4(l_{1,n}^{k_c} - l_{1,n}^{k_m})\mu_c\mu_m}, \\
     & \alpha_3=\frac{k_m^2 (2m_{1,n}^{k_m}-1) \mu_c +\mu_m(\mu_c(\mu_c + 2 m_{2,n}^{k_c}\mu_c +\mu_m-2m_{2,m}^{k_m}\mu_m ) - k_c^2 (2m_{1,n}^{k_c}+1) )  + \sqrt{\beta_2} }{4(l_{2,n}^{k_c} - l_{2,n}^{k_m})\mu_c\mu_m}, \\
     & \alpha_4=\frac{k_m^2 (2m_{1,n}^{k_m}-1) \mu_c +\mu_m(\mu_c(\mu_c + 2 m_{2,n}^{k_c}\mu_c +\mu_m-2m_{2,m}^{k_m}\mu_m ) - k_c^2 (2m_{1,n}^{k_c}+1) )  - \sqrt{\beta_2} }{4(l_{2,n}^{k_c} - l_{2,n}^{k_m})\mu_c\mu_m},
 \end{split}
\]
and
\[
 \begin{split}
     & \beta_1=16(l_{1,n}^{k_c} - l_{1,n}^{k_m})(l_{2,n}^{k_c} - l_{2,n}^{k_m})\mu_c^2\mu_m^2 +  \\
     &  \qquad \left(k_m^2 (2m_{2,n}^{k_m}-1)\mu_c \right.  \left. +\mu_m(\mu_c( \mu_c + 2m_{1,n}^{k_c}\mu_c - 2m_{1,n}^{k_m}\mu_m) -k_c^2 (2m_{2,n}^{k_c} +1) )\right)^2, \\
     & \beta_2=16(l_{1,n}^{k_c} - l_{1,n}^{k_m})(l_{2,n}^{k_c} - l_{2,n}^{k_m})\mu_c^2\mu_m^2 +  \\
     &  \qquad \left(k_m^2 (2m_{1,n}^{k_m}-1)\mu_c \right.  \left. +\mu_m(\mu_c( \mu_c + 2m_{2,n}^{k_c}\mu_c - 2m_{2,n}^{k_m}\mu_m) -k_c^2 (2m_{1,n}^{k_c} +1) )\right)^2.
 \end{split}
\]
Here $m_{1,n}^k, m_{2,n}^k$ and $l_{1,n}^k, l_{2,n}^k$ are, respectively, given in Propositions \ref{prop:spec_m} and \ref{prop:spec_l}. Moreover, the eigenfunctions $\big\{\{\bXi_{i,n}\}_{i=1}^4\big\}_{n\in\mathbb{N}}$ form a complete basis of $L^2_{T}(\partial D)^2$.
\end{thm}
\begin{proof}
\eqref{eq:spectrum} can be shown by direct calculations with the help of Propositions \ref{prop:spec_m} and \ref{prop:spec_l}. The completeness of the set of eigenfunctions $\big\{\{\bXi_{i,n}\}_{i=1}^4\big\}_{n\in\mathbb{N}}$ on $L^2_{T}(\partial D)^2$ follows from the fact that the family $(\nabla_{\mathbb{S}}Y_{n},\nabla_{\mathbb{S}}Y_{n}\times\bnu )$ with $n\geq 1$ forms an orthogonal basis of $L^2_{T}(\mathbb{S})$.
\end{proof}

\section{Surface plasmon resonance}

As applications of the spectral results in the previous section, in particular Theorem~\ref{thm:spectrum_final}, we next construct plasmonic structures of the form \eqref{eq:a1}--\eqref{eq:a2} that can induce surface plasmon resonance and cloaking effect. In this section, we mainly consider the surface plasmon resonance.

By Definition \ref{def:reson} and the corresponding discussion at the end of Section 2, it is sufficient for us to design the electric permittivity $\epsilon_c$ and the magnetic permeability $\mu_c$ in the domain $D$ such that one of the eigenvalues of the associated operator $\mathbb{I}+\mathbb{K}$ is small enough. However, from the expressions of $\{\tau_{i,n}\}_{i=1}^4$, $n\in\mathbb{N}$, in \eqref{eq:spectrum1}, these eigenvalues are too complicated to allow for the explicit design. With the aid of computer searching, we first show the existence of such plasmonic structures. Suppose $D$ is the unit ball and set the parameters in \eqref{eq:a1} and \eqref{eq:a2} to be
\begin{equation}\label{eq:config1}
 \mu_m=\epsilon_m=1,\ \  \omega=5, \quad \mu_c=1, \quad \mbox{and} \quad \epsilon_c=-1.04018+0.00004\rmi. 
\end{equation}
One can verify from the expression of $\tau_{1,n}$ in \eqref{eq:spectrum1} that
\[
  |\tau_{1,40}|\ll1. 
 \]
Let the incident wave be chosen to be a plane wave of the form
\begin{equation}\label{eq:config2}
\bE^i=4\left(e^{\rmi\omega (x/\sqrt{2}+y/\sqrt{2})}, -e^{\rmi\omega (x/\sqrt{2}+y/\sqrt{2})}, 0\right),
\end{equation}
which guarantees that the Fourier coefficient in \eqref{eq:s1} corresponding to $\tau_{1,40}$ is not vanishing. Hence, the conditions in Definition~\ref{def:reson} are fulfilled and surface resonance occurs. Indeed, we plot the resonant electric field in Fig.~1 corresponding to the electromagnetic configuration in \eqref{eq:config1} and \eqref{eq:config2}. Here and also in what follows, we make use of the finite-element method in numerically solving the involved Maxwell systems in the simulations. It can be readily seen that there is significant field enhancement near the boundary of the plasmonic inclusion, namely surface plasmon resonance occurs. 

\begin{figure}\label{fig:resonance}
  \centering
 {\includegraphics[width=4.4cm]{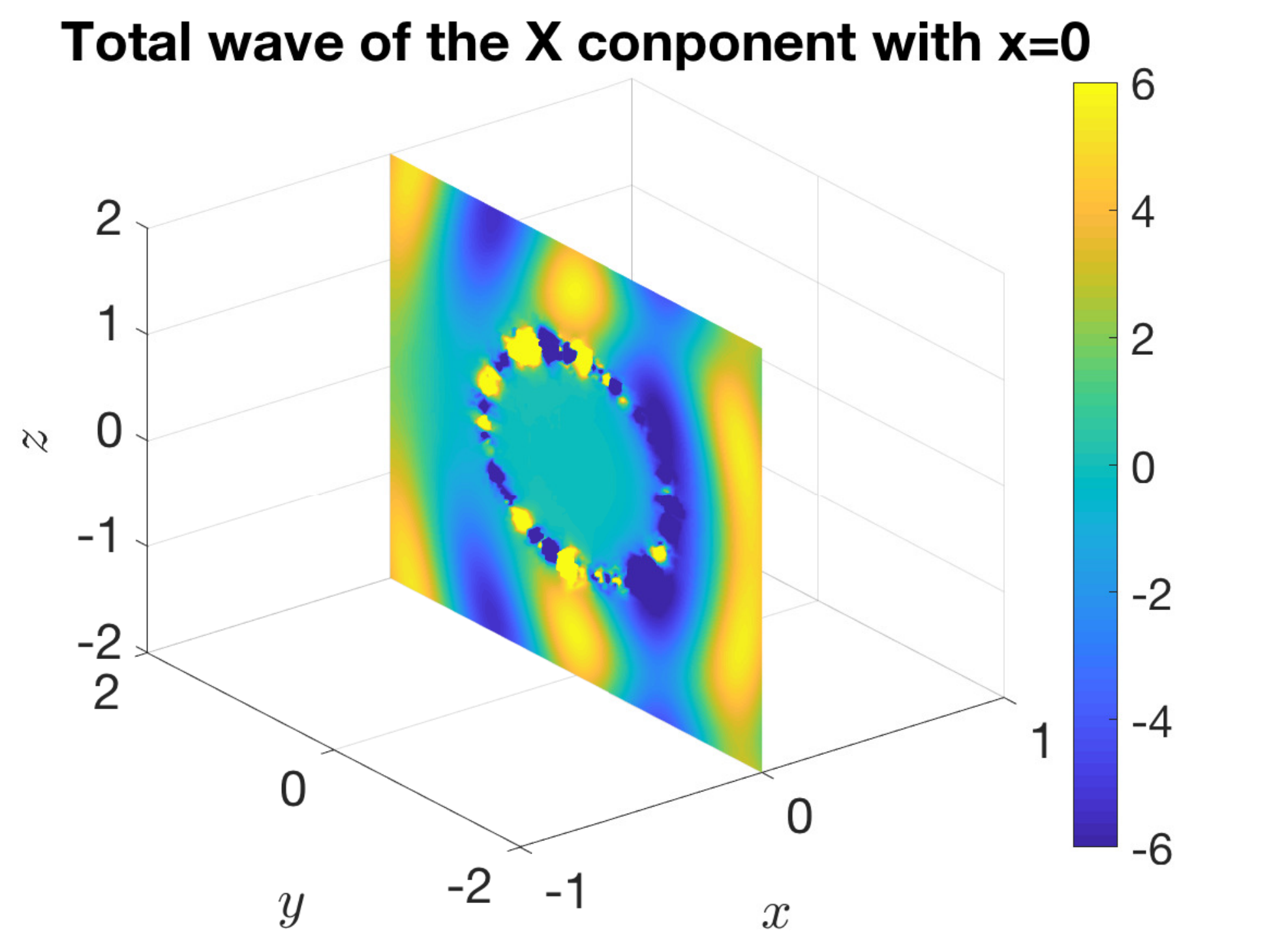}}
 {\includegraphics[width=4.4cm]{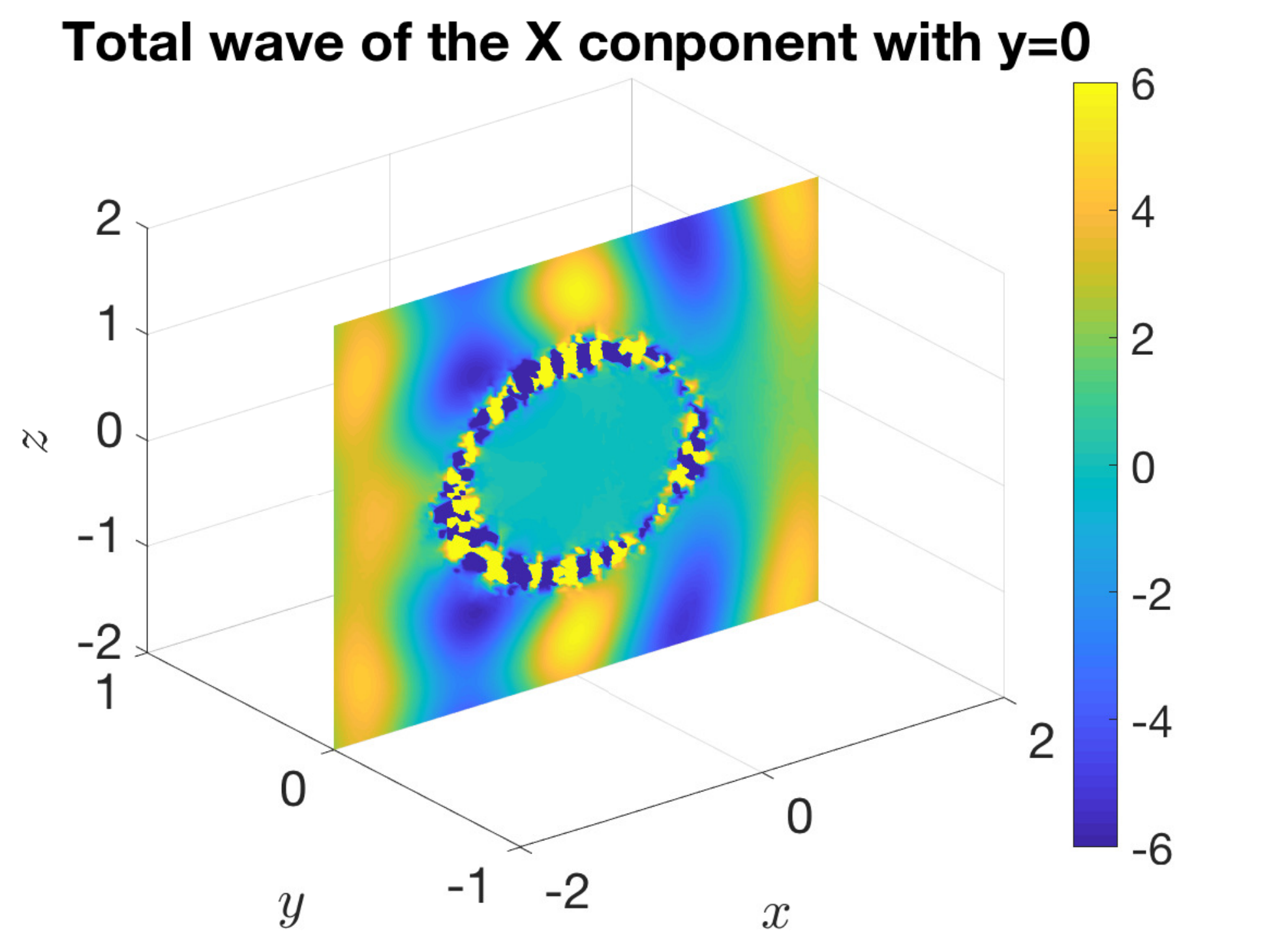}}
 {\includegraphics[width=4.4cm]{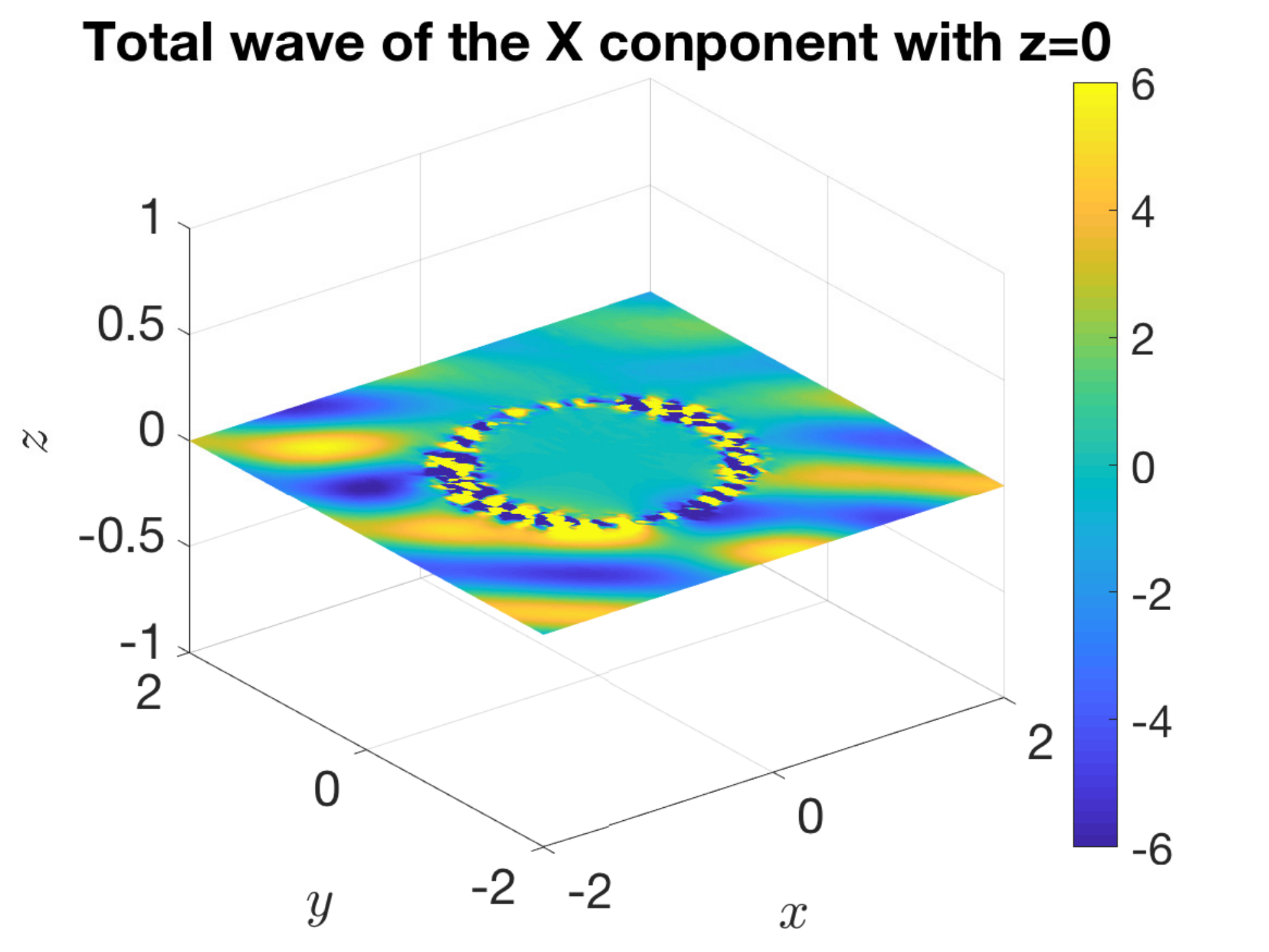}\\}
{\includegraphics[width=4.4cm]{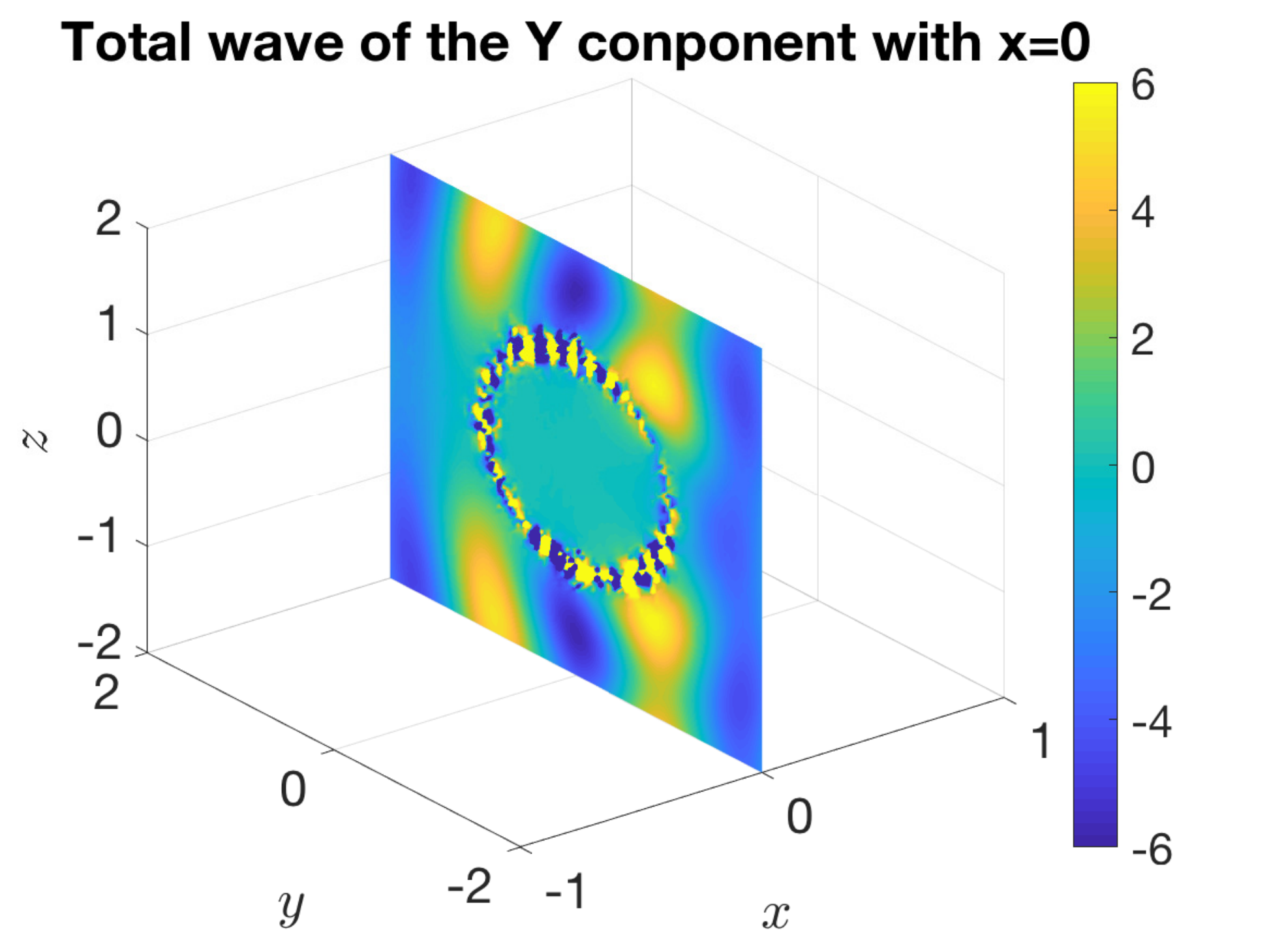}}
 {\includegraphics[width=4.4cm]{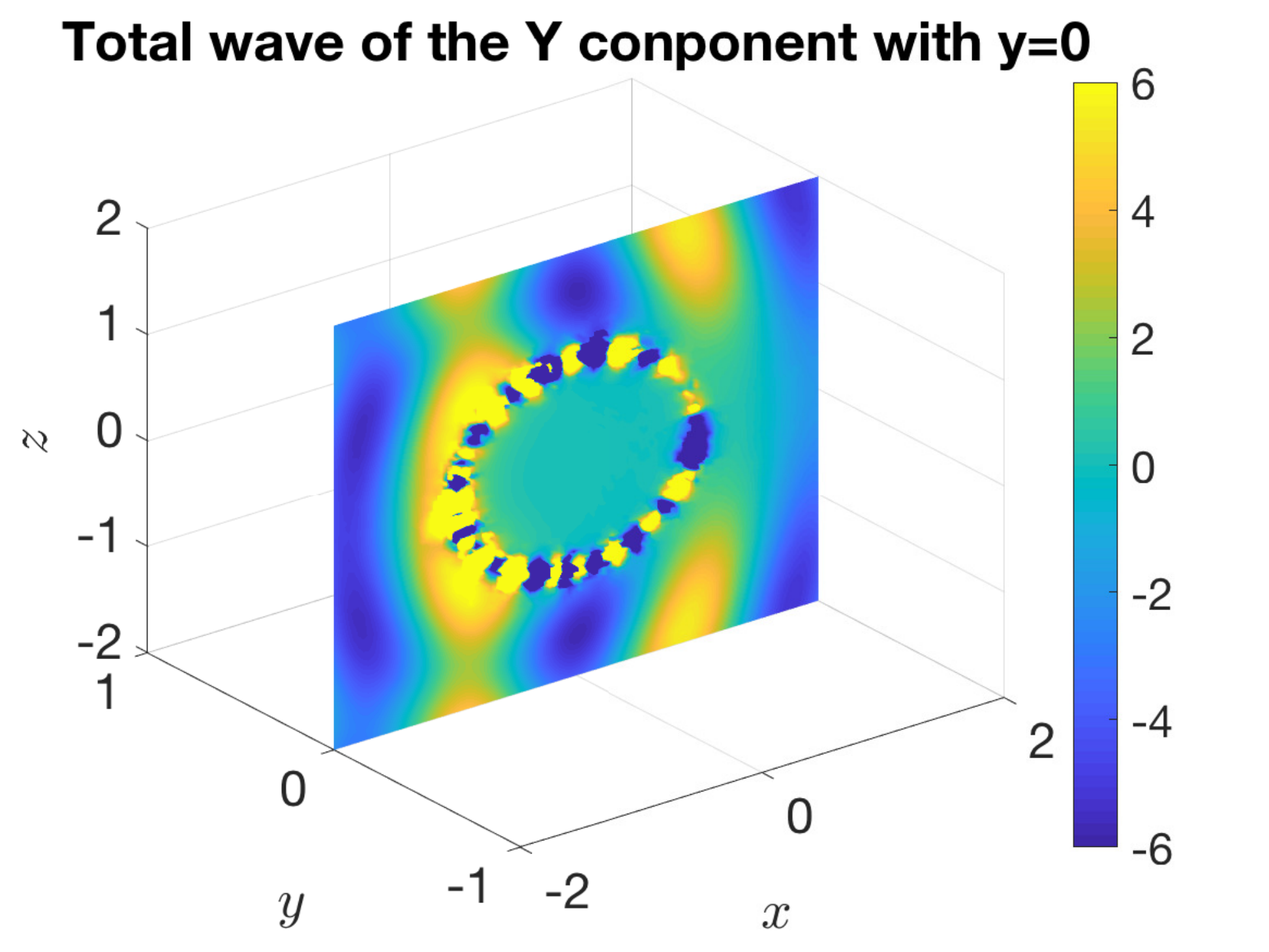}}
 {\includegraphics[width=4.4cm]{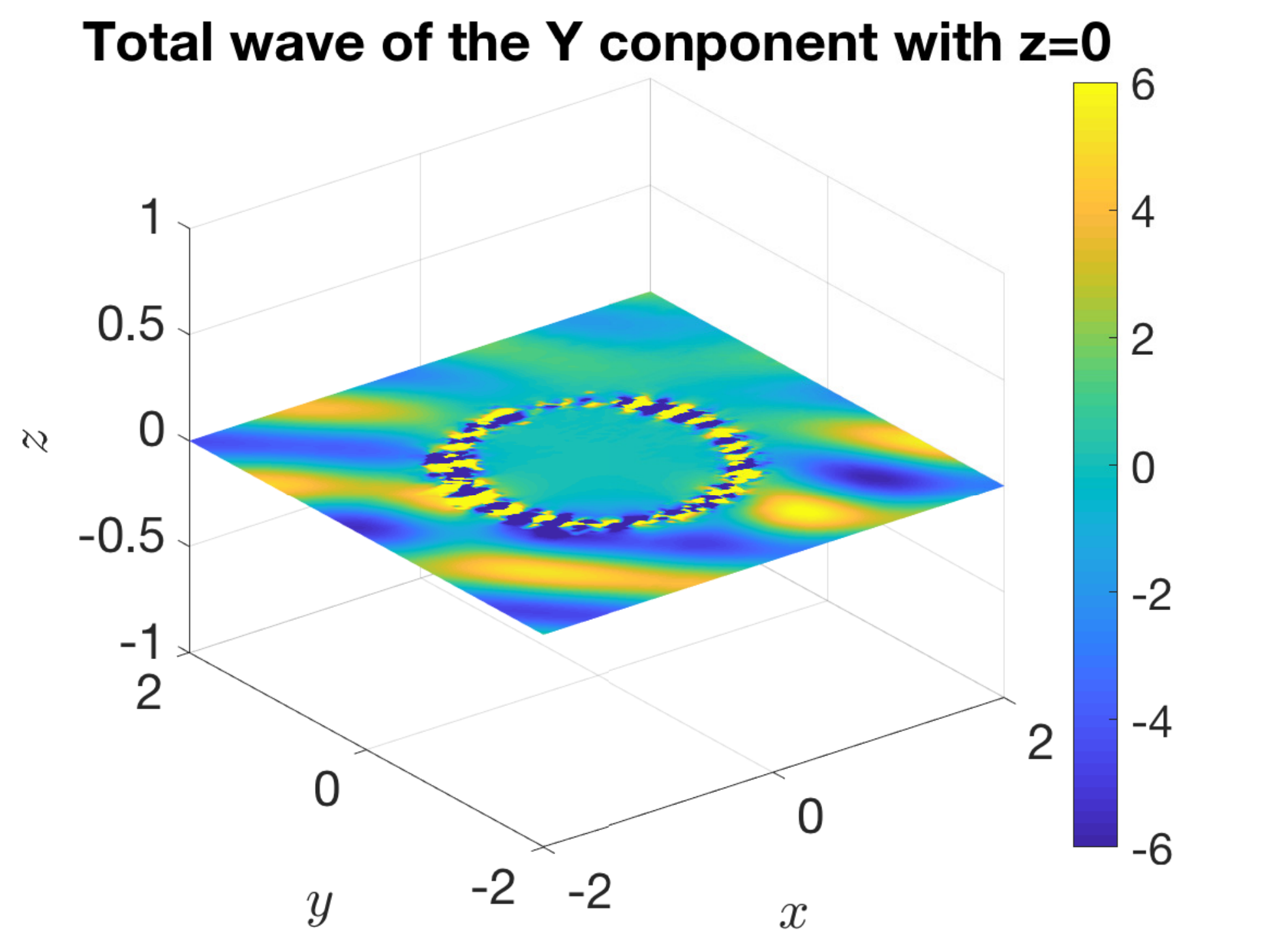}\\}
 {\includegraphics[width=4.4cm]{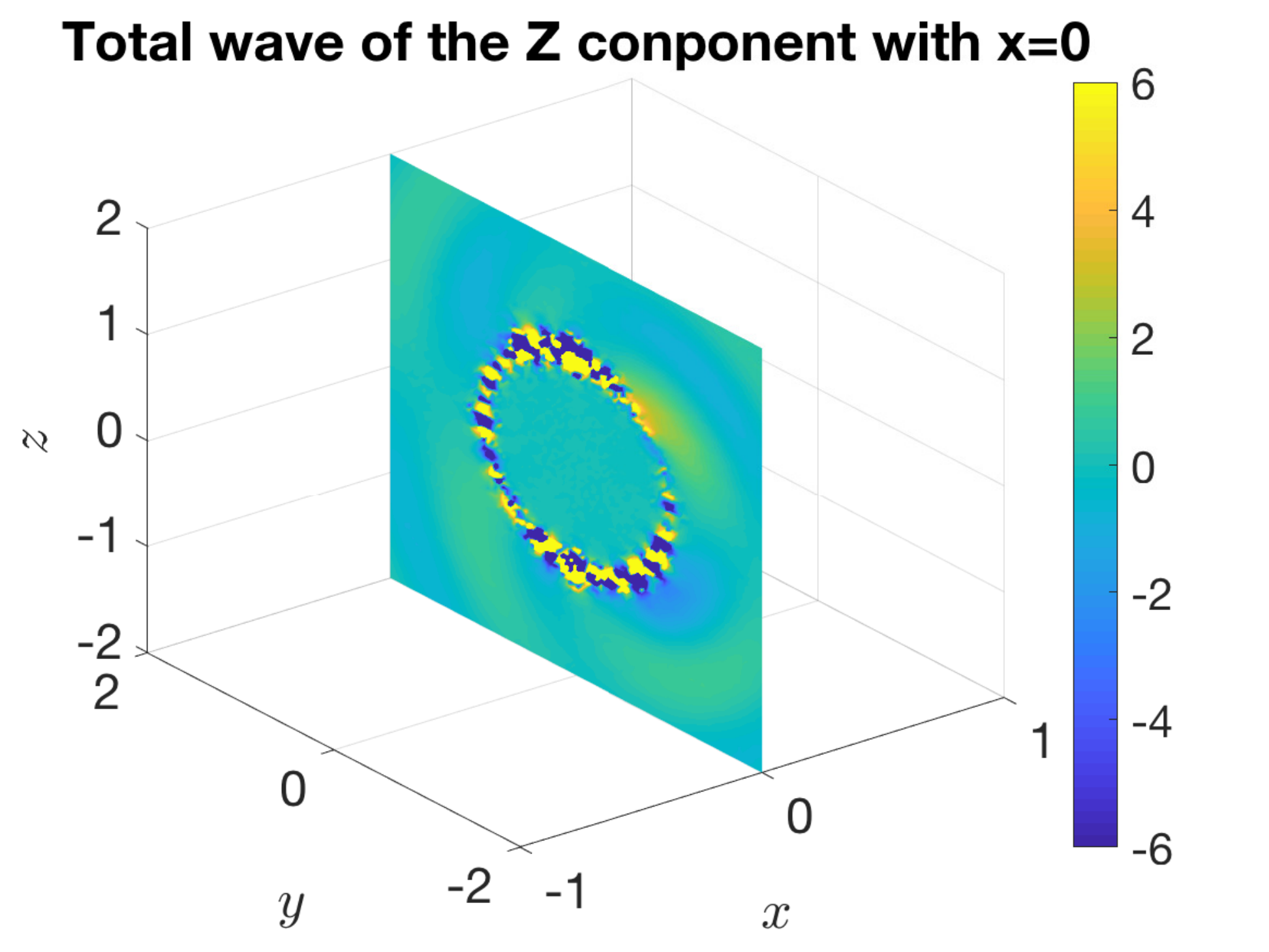}}
 {\includegraphics[width=4.4cm]{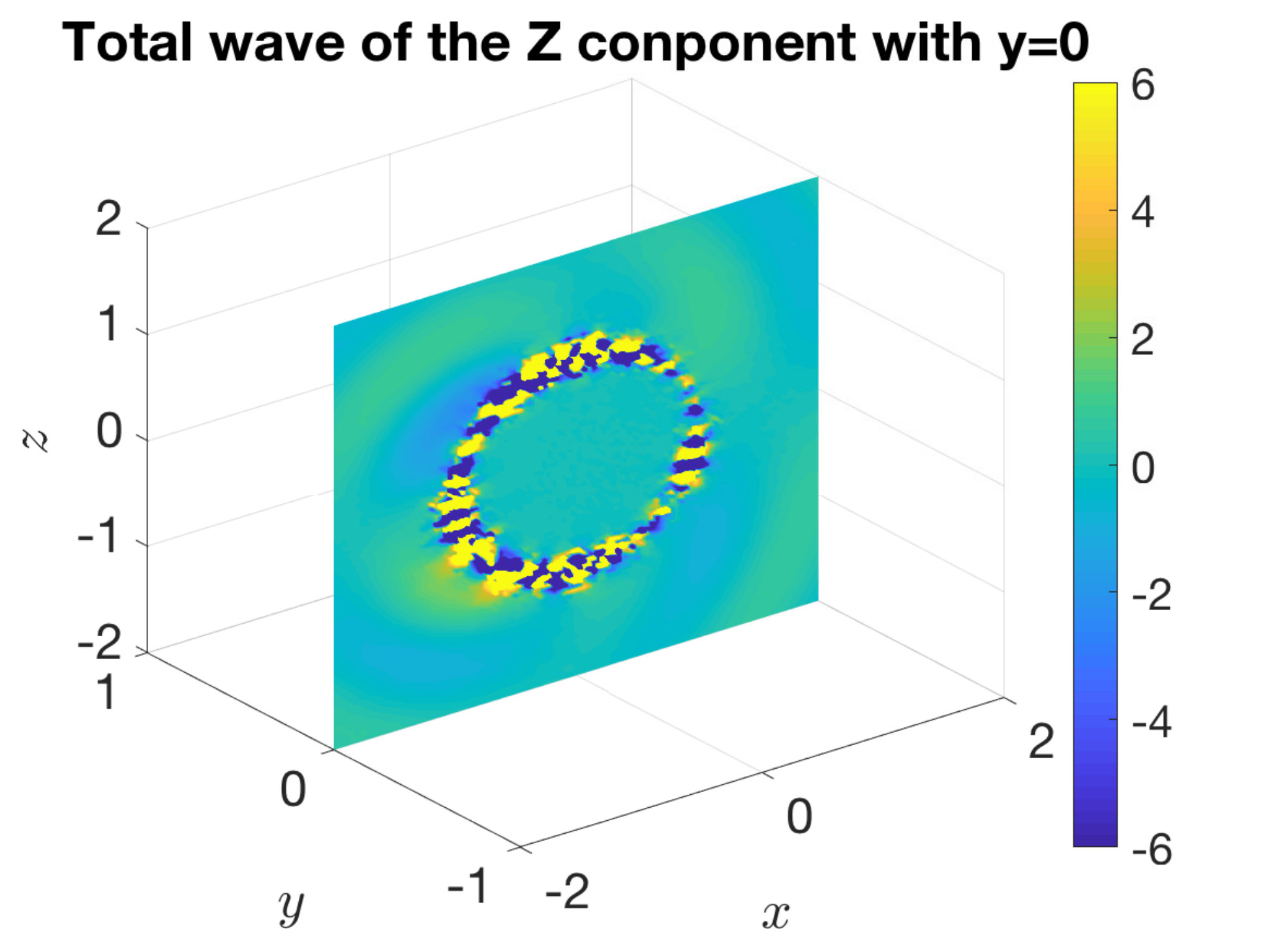}}
 {\includegraphics[width=4.4cm]{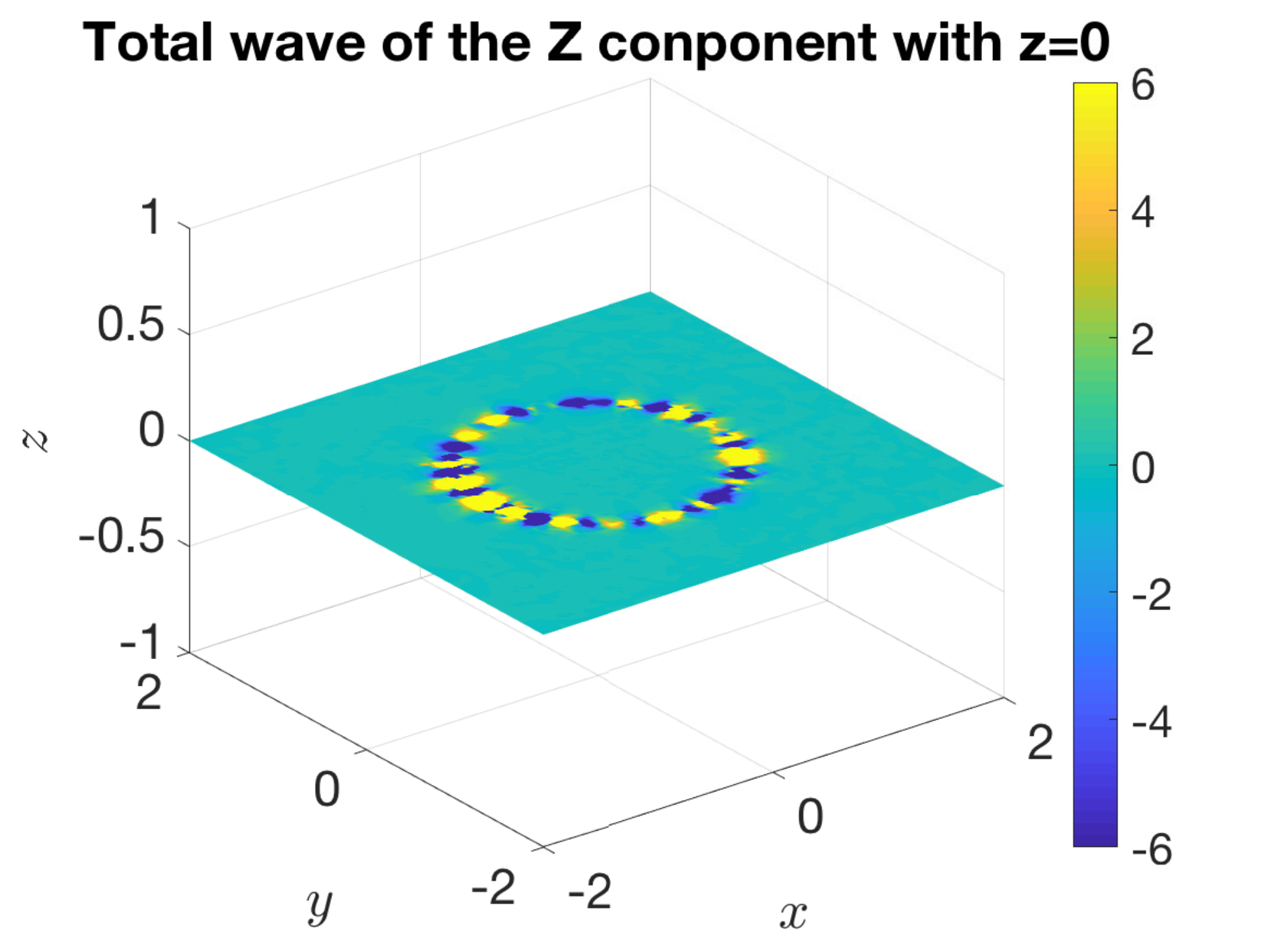}\\}
 \caption{Plotting of the first, second and third components of the resonant electric field corresponding to the electromagnetic configuration described in \eqref{eq:config1} and \eqref{eq:config2}.}
\end{figure}

Next, by imposing a certain restriction on the incident fields, we can construct a general class of plasmonic structures that can induce surface plasmon resonance. To that end, we first note that for $n\gg 1$, the spherical Bessel and Hankel functions $j_n(t)$ and $h_n^{(1)}(t)$ have the following asymptotic properties
\begin{equation}\label{eq:asymptotic_j}
  j_n(t)=\frac{t^n}{1\cdot 3 \cdots (2n+1)}\left(1+\mathcal{O}\left(\frac{1}{n}\right) \right),
\end{equation}
and
\begin{equation}\label{eq:asymptotic_h}
  h_n^{(1)}(t)=\frac{1\cdot 3 \cdots (2n-1)}{\rmi t^{n+1}}\left(1+\mathcal{O}\left(\frac{1}{n}\right) \right).
\end{equation}
With the help of the asymptotic properties for $j_n(t)$ and $h_n^{(1)}(t)$ in \eqref{eq:asymptotic_j} and \eqref{eq:asymptotic_h}, one can show the following estimates for the parameters, $\lambda_{n,k,R}$, $\chi_{n,k,R}$, $m_{1,n}^k$, $m_{2,n}^k$, $l_{1,n}^k$ and $l_{2,n}^k$, given in Proposition \ref{thm:eigenvalue_k-star}, Lemmas \ref{prop:spec_m} and \ref{prop:spec_l} with $n\gg 1$,
\[
\begin{split}
& \lambda_{n,k,R}=\frac{1}{2(2n+1)}\left(1+\mathcal{O}\left(\frac{1}{n}\right) \right), \quad  \chi_{n,k,R}=\frac{-R}{2n+1}\left(1+\mathcal{O}\left(\frac{1}{n}\right) \right),\\
& m_{1,n}^k=\frac{1}{4n+2}\left(1+\mathcal{O}\left(\frac{1}{n}\right) \right), \hspace*{1.3cm} m_{2,n}^k=-\frac{1}{4n+2}\left(1+\mathcal{O}\left(\frac{1}{n}\right) \right),
\end{split}
\]
and
\[
 l_{1,n}^k=\frac{-Rk^2}{2n+1}\left(1+\mathcal{O}\left(\frac{1}{n}\right) \right), \quad l_{2,n}^k=-\frac{n(n+1)}{(2n+1)R} + \frac{Rk^2(4n^2+4n+3)}{8n^3+12n^2-2n-3}\left(1+\mathcal{O}\left(\frac{1}{n}\right) \right).
\]
Using the above estimates, one can further derive the following asymptotic expressions for the eigenvalues $\tau_{i,n}$ in \eqref{eq:spectrum1} for $n\gg 1$,
\begin{equation}\label{eq:aym_eigen}
 \tau_{i,n}=\tilde{\tau}_{i,n} \left(1+\mathcal{O}\left(\frac{1}{n}\right) \right), \quad i=1,2,3,4,
\end{equation}
where
\[
 \begin{split}
   \tilde{\tau}_{1,n}= & \frac{1}{2(2n+1)}\bigg( (n+1)\mu_c + n\mu_m + ((n+1)\epsilon_m + n\epsilon_c)\omega^2 - \\
     & \sqrt{((n+1)\mu_c + n\mu_m - ((n+1)\epsilon_m + n\epsilon_c)\omega^2)^2 -\frac{4(\epsilon_c\mu_c-\epsilon_m\mu_m)^2(4n^2+4n+3)\omega^4R^2}{4n^2+4n-3} } \bigg),\\
   \tilde{\tau}_{2,n}= & \frac{1}{2(2n+1)}\bigg( (n+1)\mu_c + n\mu_m + ((n+1)\epsilon_m + n\epsilon_c)\omega^2 + \\
     & \sqrt{((n+1)\mu_c + n\mu_m - ((n+1)\epsilon_m + n\epsilon_c)\omega^2)^2 -\frac{4(\epsilon_c\mu_c+\epsilon_m\mu_m)^2(4n^2+4n+3)\omega^4R^2}{4n^2+4n-3} } \bigg),\\
   \tilde{\tau}_{3,n}= & \frac{1}{2(2n+1)}\bigg( (n+1)\mu_m + n\mu_c + ((n+1)\epsilon_c + n\epsilon_m)\omega^2 - \\
     & \sqrt{((n+1)\mu_m + n\mu_c - ((n+1)\epsilon_c + n\epsilon_m)\omega^2)^2 -\frac{4(\epsilon_c\mu_c-\epsilon_m\mu_m)^2(4n^2+4n+3)\omega^4R^2}{4n^2+4n-3} } \bigg),
 \end{split}
\]
and
\[
 \begin{split}
   &\tilde{\tau}_{4,n}= \frac{1}{2(2n+1)}\bigg( (n+1)\mu_m + n\mu_c + ((n+1)\epsilon_c + n\epsilon_m)\omega^2 + \\
     & \sqrt{((n+1)\mu_m + n\mu_c - ((n+1)\epsilon_c + n\epsilon_m)\omega^2)^2 -\frac{4(\epsilon_c\mu_c-\epsilon_m\mu_m)^2(4n^2+4n+3)\omega^4R^2}{4n^2+4n-3} } \bigg).
 \end{split}
\]
Since the eigenfunctions $\{\bXi_{i,n}\}_{i=1}^4$ with $n\in\mathbb{N}$ given in the Theorem \ref{thm:spectrum_final}  are complete on $L^2_T(\mathbb{S})^2$, we can write the source term $\mathbf{F}$ in \eqref{eq:operator_eq} in the following form
\begin{equation}\label{eq:source_form}
\mathbf{F}=\sum_{i=1}^4\sum_{n=1}^{+\infty}f_{i,n}\bXi_{i,n}.
\end{equation}

We can show the following resonance result.

\begin{thm}\label{thm:reson_n}
Suppose the Newtonial potential $\mathbf{F}$ in \eqref{eq:operator_eq} has the representation in \eqref{eq:source_form}. Let $n_0\in\mathbb{N}$ be sufficiently large such that the spherical Bessel and Hankel functions $j_{n_0}(t)$ and $h_{n_0}^{(1)}(t)$ enjoy the asymptotic properties given in \eqref{eq:asymptotic_j} and \eqref{eq:asymptotic_h}. Let the parameter $\epsilon_c$ and $\mu_c$ inside the domain $D$ be chosen such that the following conditions are fulfilled:
\begin{equation}\label{eq:cf1}
\epsilon_c=-\epsilon_m  \quad \mbox{and} \quad  \Re\left( \mu_c+\mu_m \right) \geq 0,
\end{equation}
or
\begin{equation}\label{eq:cf2}
\mu_c=-\mu_m  \quad \mbox{and} \quad \Re\left( (\epsilon_c + \epsilon_m)\omega^2 \right) \geq 0. 
\end{equation}
If the Fourier coefficients $f_{1,n_0}\neq0$ or  $f_{3,n_0}\neq 0$, then surface plasmon resonance occurs for both the medium configurations \eqref{eq:cf1} and \eqref{eq:cf2}. Similarly, let the parameter $\epsilon_c$ and $\mu_c$ inside the domain $D$ be chosen such that the following conditions are fulfilled:
\begin{equation}\label{eq:cf3}
\epsilon_m=-\epsilon_c  , \quad \mbox{and} \quad  \Re\left( \mu_c+\mu_m \right) \leq 0,
\end{equation}
or
\begin{equation}\label{eq:cf4}
\mu_m=-\mu_c  \quad \mbox{and} \quad \Re\left( (\epsilon_c + \epsilon_m)\omega^2 \right) \leq 0. 
\end{equation}
If the Fourier coefficients $f_{2,n_0}\neq0$ or  $f_{4,n_0}\neq 0$, then surface plasmon resonance occurs for both the medium configurations \eqref{eq:cf3} and \eqref{eq:cf4}. 
\end{thm}
\begin{proof}
From the expression of $\tau_{1,n_0}$ in \eqref{eq:aym_eigen}, one can show by direct calculations that
\[
 \begin{split}
   &\tilde{\tau}_{1,n_0}=  \frac{1}{2(2{n_0}+1)}\bigg( (n_{0}+1)\mu_c + n_0\mu_m + ((n_0+1)\epsilon_m + n_0\epsilon_c)\omega^2 - \\
     & \sqrt{((n_0+1)\mu_c + n_0\mu_m - ((n_0+1)\epsilon_m + n_0\epsilon_c)\omega^2)^2 -\frac{4(\epsilon_c\mu_c-\epsilon_m\mu_m)^2(4n_0^2+4n_0+3)\omega^4R^2}{4n_0^2+4n_0-3} } \bigg)\\
    &\qquad =\frac{1}{4}\big((\mu_c+\mu_m) + (\epsilon_c + \epsilon_m)\omega^2 - \sqrt{((\mu_c+\mu_m) -(\epsilon_c + \epsilon_m)\omega^2)^2} \big) +\mathcal{O}\left(\frac{1}{n_0}\right).
 \end{split}
\]
If 
\[
\epsilon_m=-\epsilon_c \quad \mbox{and} \quad  \Re\left( \mu_c+\mu_m \right) \geq 0,
\]
one can find that 
\[
  \tilde{\tau}_{1,n_0}= \mathcal{O}\left(\frac{1}{n_0}\right).
\]
Clearly, one has resonance for this case since $n_0\gg 1$. For the other case with
\[
\mu_m=-\mu_c  \quad \mbox{and} \quad \Re\left( (\epsilon_c + \epsilon_m)\omega^2 \right) \geq 0 ,
\]
one can show by following a similar argument that
\[
   \tilde{\tau}_{1,n_0}= \mathcal{O}\left(\frac{1}{n_0}\right),
\]
and hence resonance occurs. The occurrence of resonance for the medium configurations \eqref{eq:cf3} and \eqref{eq:cf4} can be shown in a similar manner with the help of the asymptotic expressions in \eqref{eq:aym_eigen}. 

The proof is complete. 
\end{proof}

\begin{cor}
 By Theorem \ref{thm:reson_n}, it is readily seen that if the parameters $\epsilon_c$ and $\mu_c$ inside the domain $D$ are chosen as 
 \[
 \epsilon_c=-\epsilon_m \quad \mbox{and} \quad \mu_c=-\mu_m,
 \]
 and the Fourier coefficients $f_{i,n}$ for $\mathbf{F}$ in \eqref{eq:source_form} is non-vanishing for some $i\in\{1,2,3,4\}$ and a certain $n=n_0\gg 1$, then plasmon resonance occurs. 
 \end{cor}

Since we are considering the electromagnetic scattering in the finite-frequency regime, it is more practical for the plasmon parameters to be constructed according to the Drude model (cf. \cite{ADM}), which states that
\begin{equation}\label{eq:drude}
\begin{split}
 \epsilon_c&=\epsilon_0\left( 1-\frac{\omega_p^2}{\omega(\omega+\rmi\tau)} \right),\\
 \mu_c&=\mu_0\left( 1-\mathcal{F}\frac{\omega^2}{\omega^2-\omega_0^2+\rmi\tau\omega} \right),
 \end{split}
\end{equation}
where $\omega_p$ is the plasmon frequency of the bulk material, $\tau>0$ is the object's damping coefficient, $\mathcal{F}$ is a filling factor and $\omega_0$ is a localized plasmon resonant frequency. In the rest of the section, we show that by following a similar strategy as before via the use of the spectral result in Theorem~\ref{thm:spectrum_final}, one can construct the desired plasmonic structures according to the Drude model. Indeed,  if we take
\[
 \epsilon_0=\mu_0=1,\quad \omega=5,\quad \tau=0.0001, \quad \omega_0=2,\quad \omega_p^2=51.0045 \ \ \ \mbox{and}\quad \mathcal{F}=0,
\]
then by straightforward calculations one can obtain that 
\[
 \epsilon_c=-1.04018+0.00004\rmi \quad \mbox{and} \quad \mu_c=1,
\]
which is exactly the one constructed earlier in \eqref{eq:config1}. 
Let us consider another case with the filling factor $\mathcal{F}\neq 0$, and set
\begin{equation}\label{eq:config6}
 \epsilon_0=\mu_0=1,\quad \omega=5,\quad \tau=0.00001, \quad\omega_0=2, \quad \omega_p^2=51.518 \quad \mbox{and}\quad \mathcal{F}=0.1. 
\end{equation}
Associated with \eqref{eq:config6}, one has by direct calculations that 
\begin{equation}\label{eq:config7}
 \epsilon_c=-1.06072 + 4.1\times10^{-6}\rmi \quad \mbox{and} \quad \mu_c=0.880952+2.8\times10^{-7}\rmi.
\end{equation}
If the incident electric field is taken to be \eqref{eq:config2} again, one can show that surface plasmon resonance occurs. We plot the corresponding resonant electric field in Fig.~2 and one can readily see the surface plasmon resonance phenomenon.  

\begin{figure}\label{fig:reson_drude}
  \centering
 {\includegraphics[width=4.4cm]{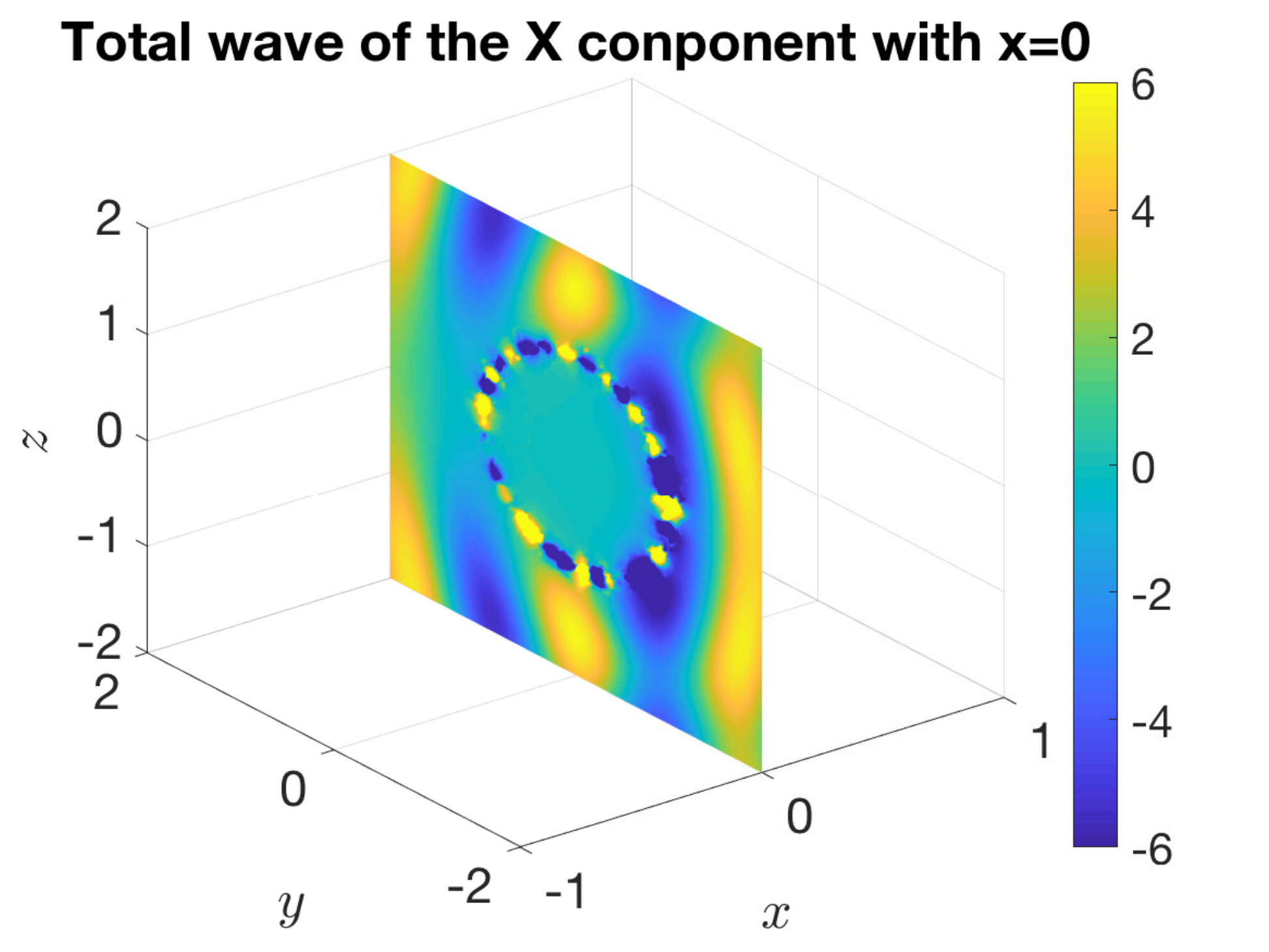}}
 {\includegraphics[width=4.4cm]{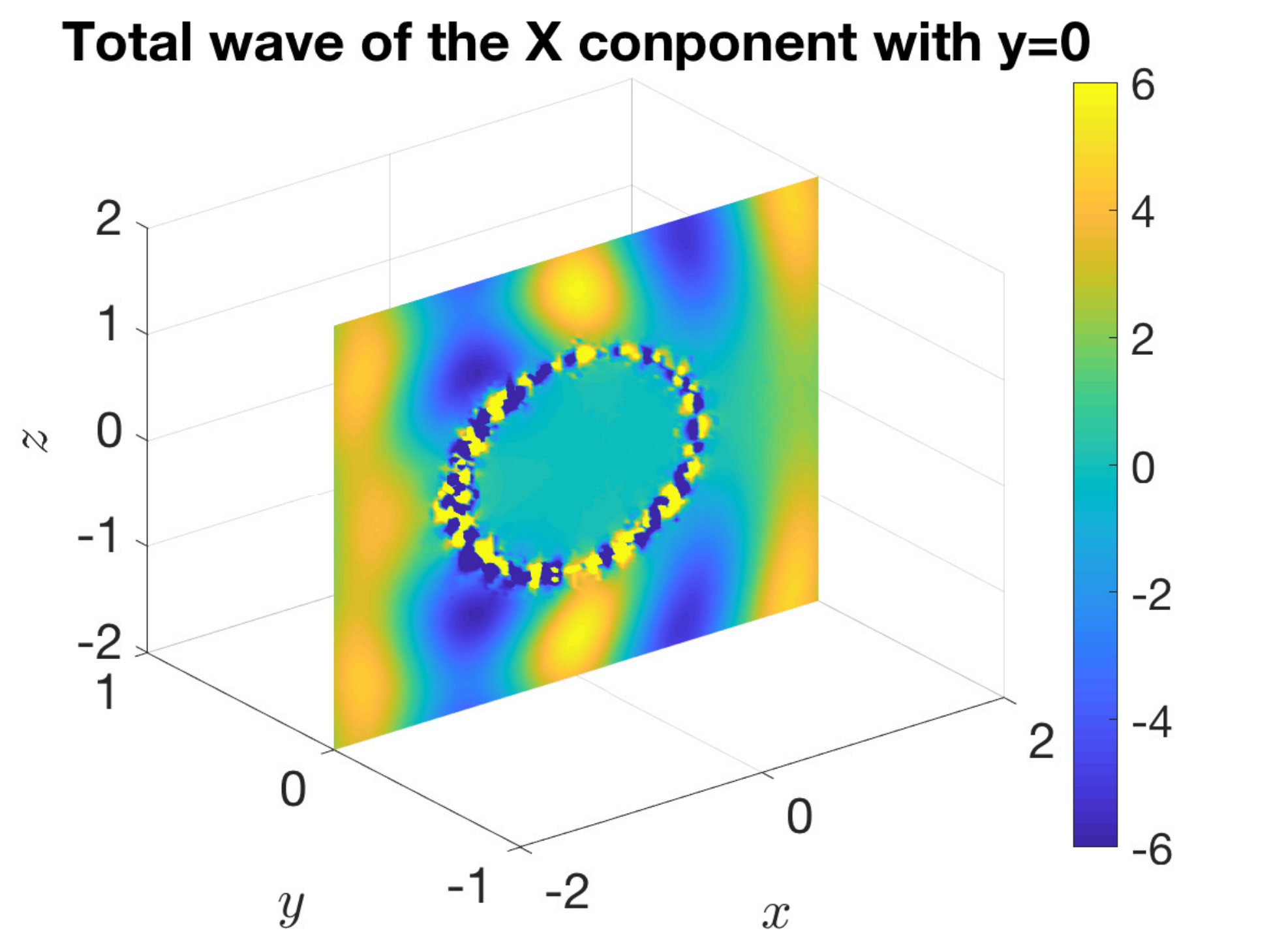}}
 {\includegraphics[width=4.4cm]{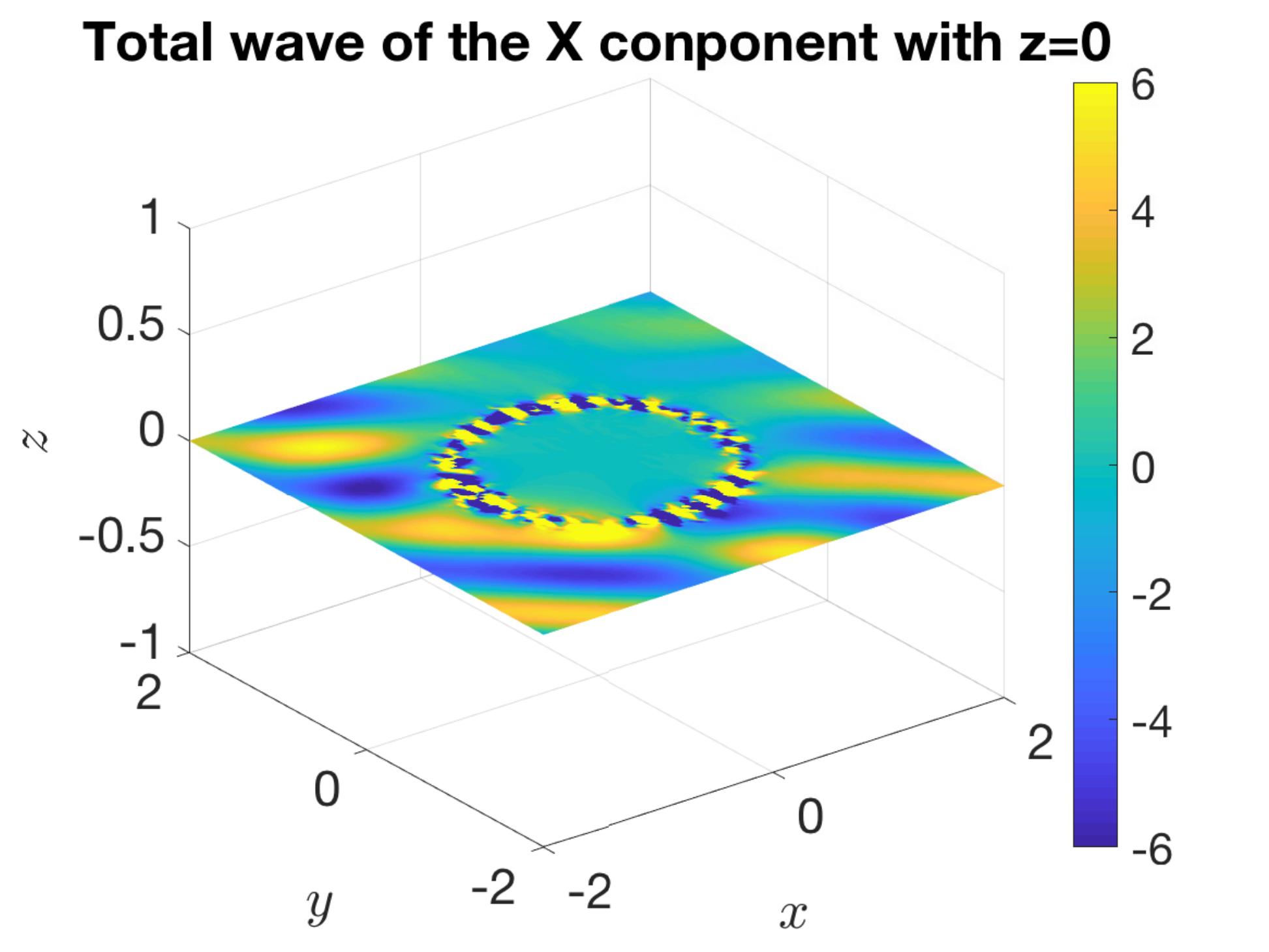}\\}
{\includegraphics[width=4.4cm]{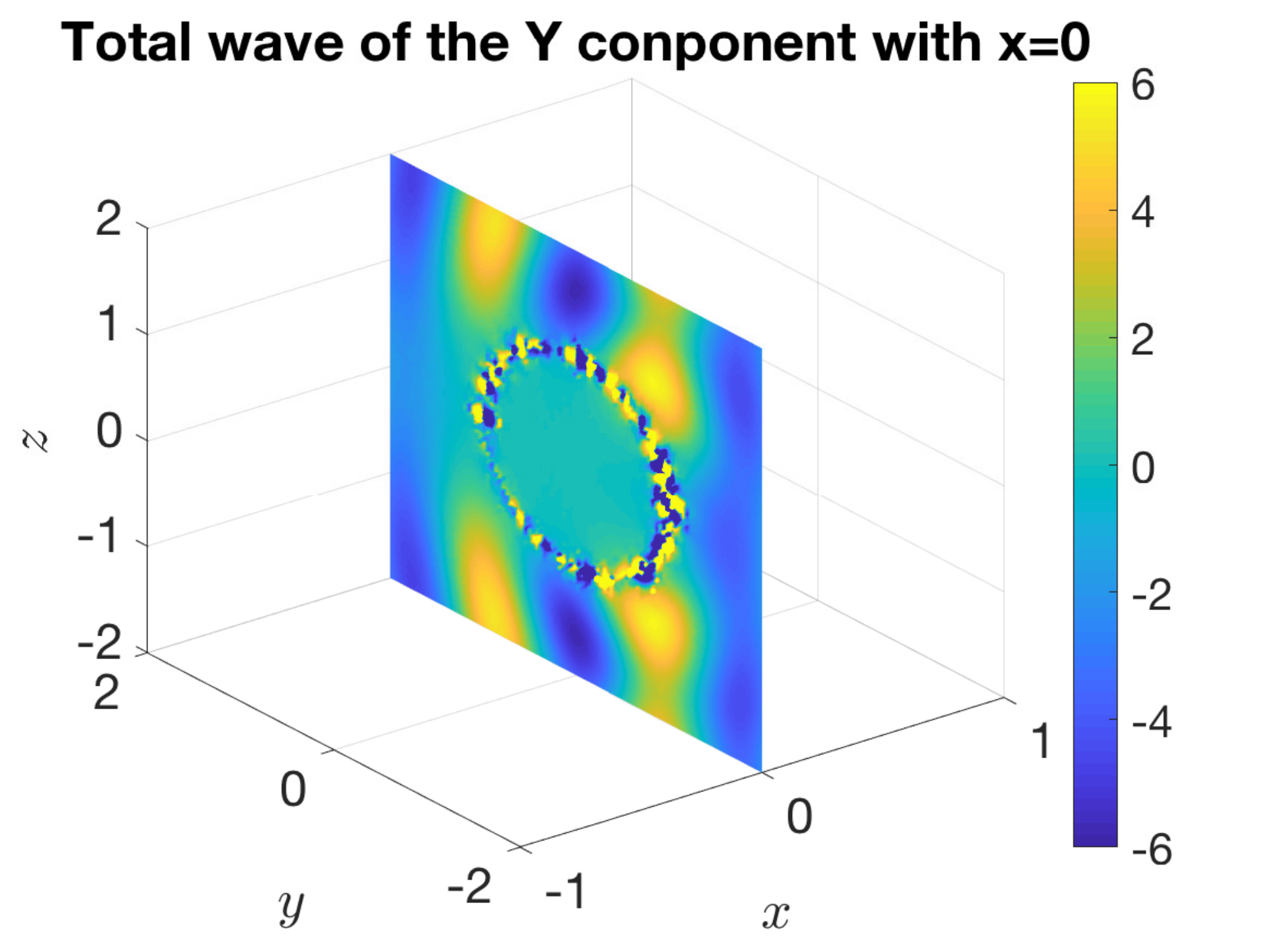}}
 {\includegraphics[width=4.4cm]{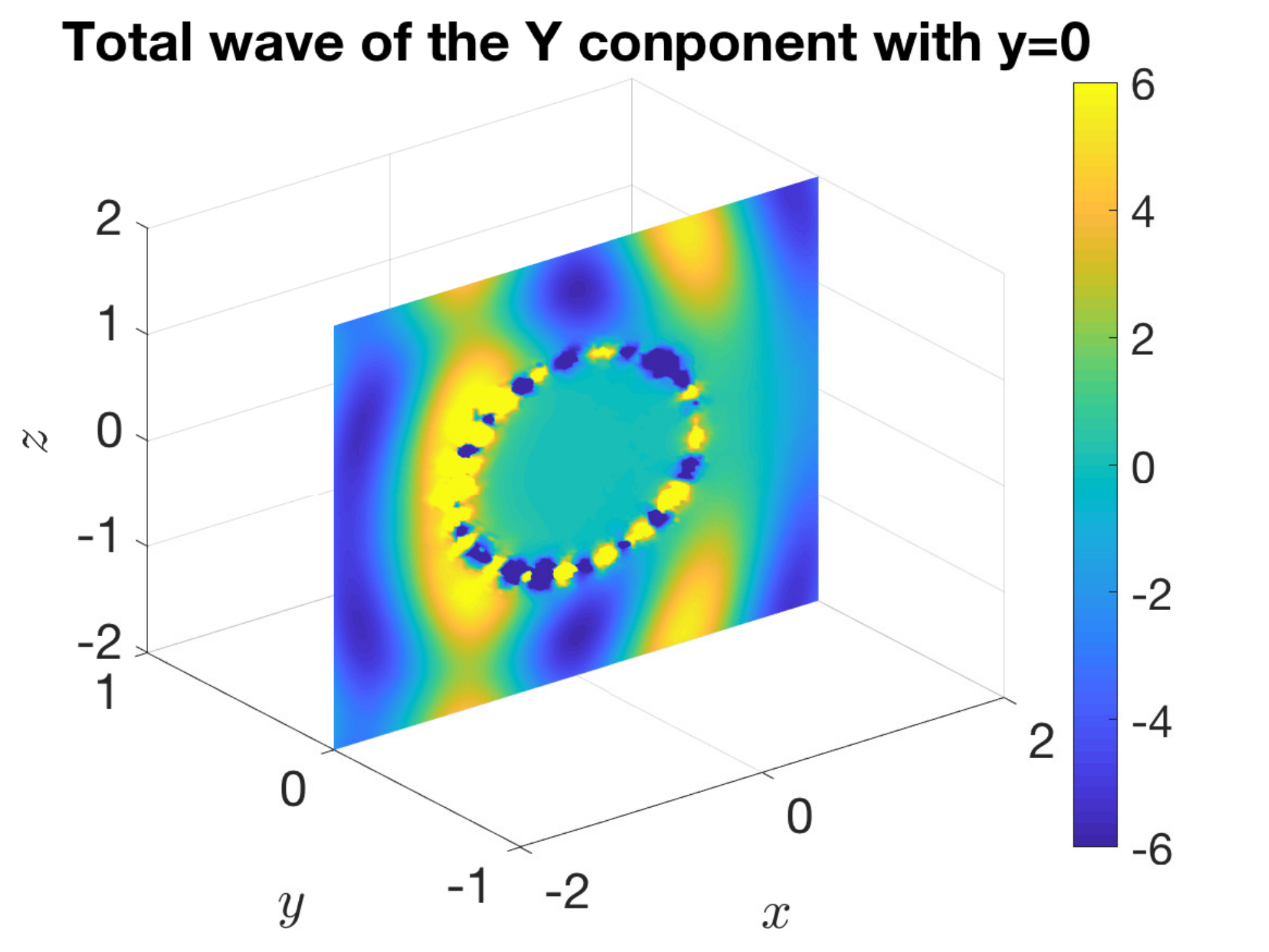}}
 {\includegraphics[width=4.4cm]{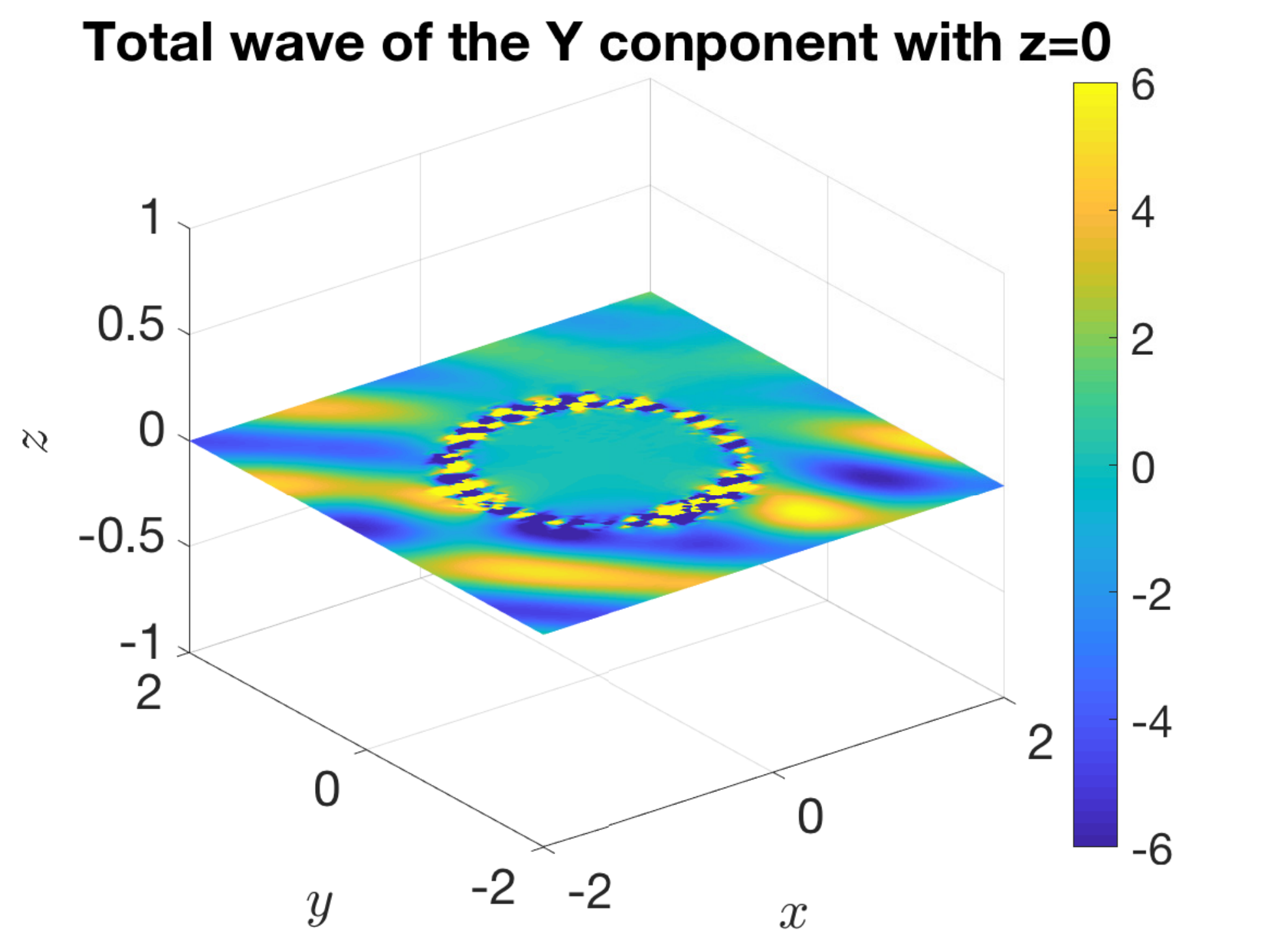}\\}
 {\includegraphics[width=4.4cm]{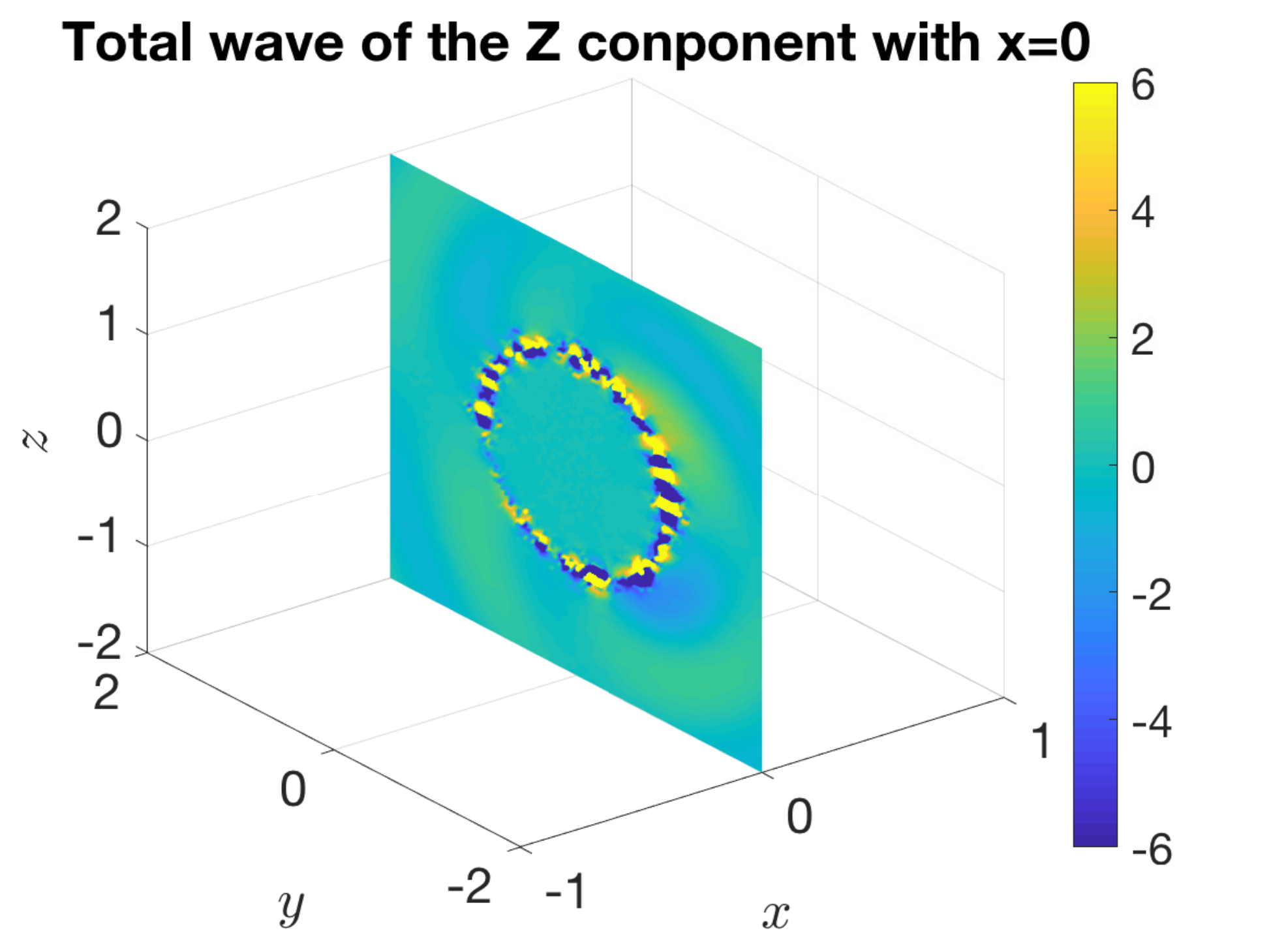}}
 {\includegraphics[width=4.4cm]{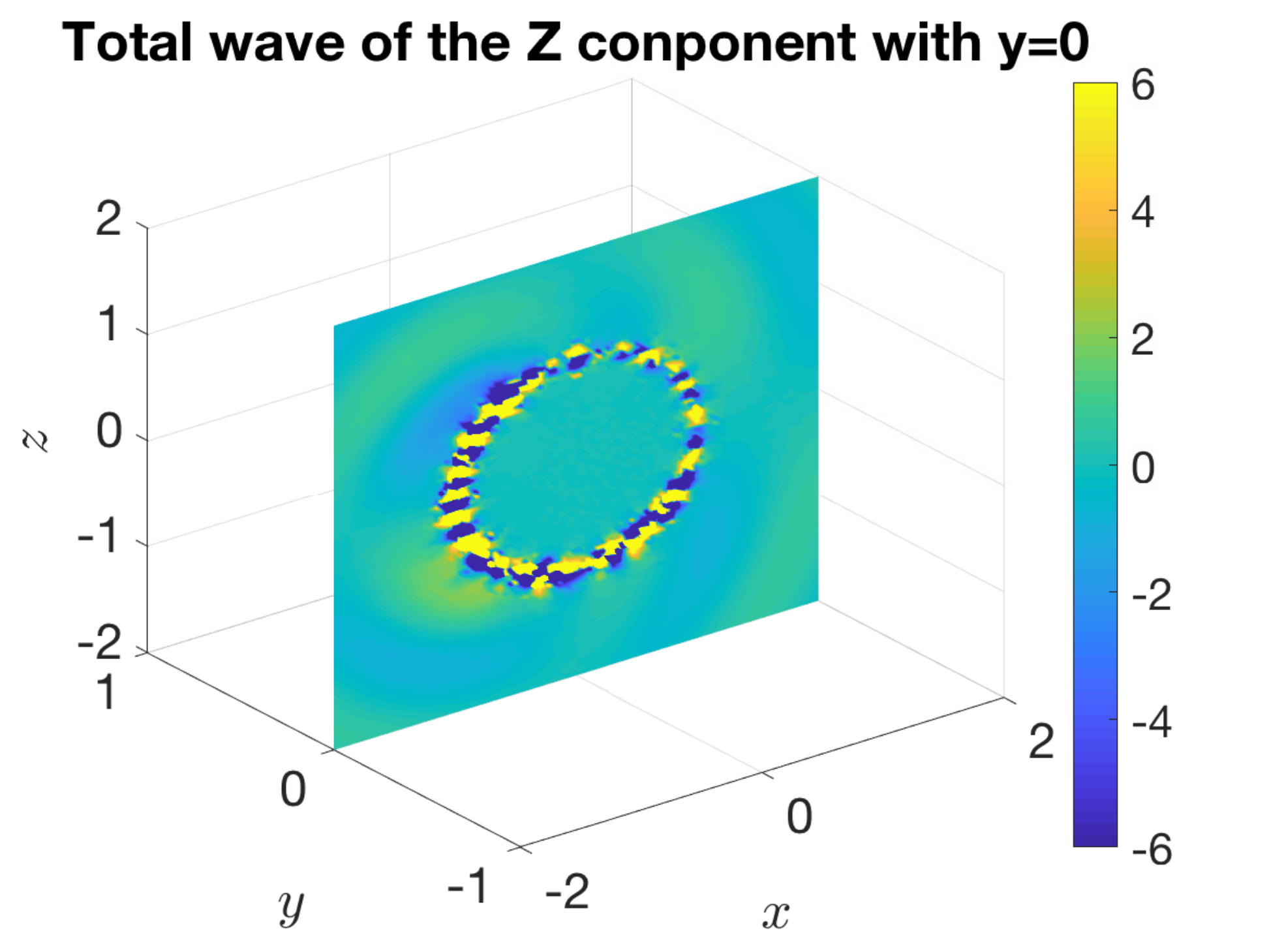}}
 {\includegraphics[width=4.4cm]{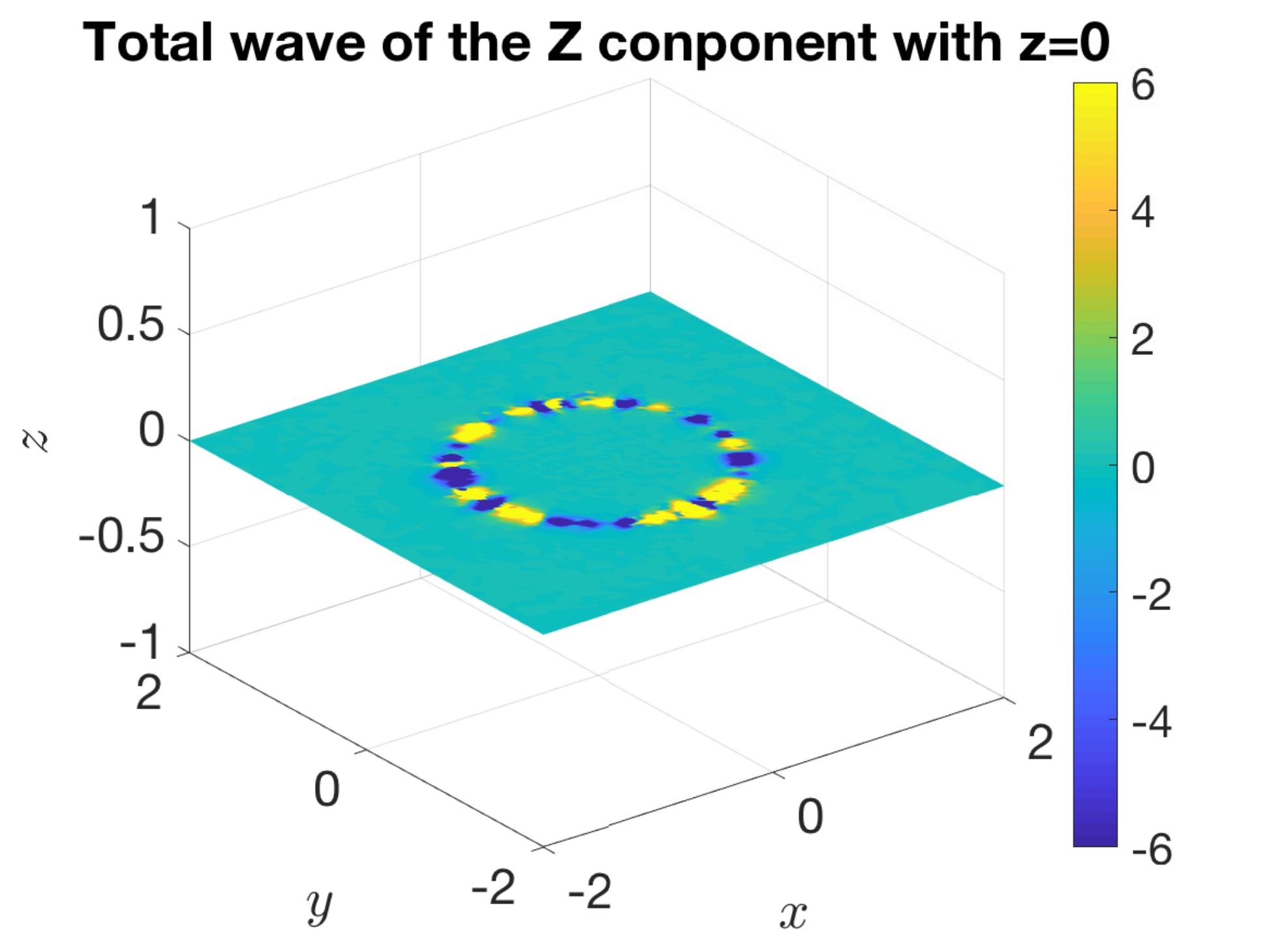}\\}
  \caption{Plotting of the first, second and third components of the resonant electric field corresponding to the electromagnetic configuration described in \eqref{eq:config7} and \eqref{eq:config2}.}
\end{figure}

\section{Invisibility cloaking effect}

In this section, we construct plasmonic structures that can induce the cloaking effect associated to certain incident waves. By Definition~\ref{def:invi} as well as the discussion made at the end of Section 2, one needs to design the electric permittivity $\epsilon_c$ and the magnetic permeability $\mu_c$ in the domain $D$ such that all those eigenvalues of the associated operator $\mathbb{I}+\mathbb{K}$ that correspond to the non-vanishing Fourier coefficients of the source $\mathbf{F}$, through the relation \eqref{eq:s1}, should be large enough. That means the plasmonic structures and the incident waves are related to each other and this makes the corresponding construction rather delicate and tricky. 

Similar to our study for the surface plasmon resonance in the previous section, we first present some specific constructions as well as the corresponding numerical simulations for illustrating the cloaking effects. Let
\begin{equation}\label{eq:ca1}
 R=1, \quad \omega=5 ,\quad  \mu_m=\epsilon_m=1, \quad \mu_c=1 \quad \mbox{and} \quad \epsilon_c=-6.55806+0.000001\rmi,
\end{equation}
and associated with such a configuration, one has from \eqref{eq:spectrum1} that
\begin{equation}\label{eq:ca2}
  |\tau_{3,1}|\gg1.
 \end{equation}
We take the incident electric field to be
\begin{equation}\label{eq:cloaking_source}
\bE^i =  \left[
\begin{array}{c}
\bE_x^i\\
\bE_y^i\\
\bE_z^i\\
\end{array}
\right]
  =\left[
                 \begin{array}{c}
                  \frac{100y}{\omega\sqrt{x^2+y^2+z^2}}
                   \left(\frac{\sin\left(\omega\sqrt{x^2+y^2+z^2}\right)}{\omega^2(x^2+y^2+z^2)} - \frac{\cos\left(\omega\sqrt{x^2+y^2+z^2}\right)}{\omega\sqrt{x^2+y^2+z^2}} \right)  \\
                  \frac{-100x}{\omega\sqrt{x^2+y^2+z^2}} \left(\frac{\sin\left(\omega\sqrt{x^2+y^2+z^2}\right)}{\omega^2(x^2+y^2+z^2)} - \frac{\cos\left(\omega\sqrt{x^2+y^2+z^2}\right)}{\omega\sqrt{x^2+y^2+z^2}} \right) \\
                    0
                 \end{array}
               \right],
\end{equation}
whose trace on $\partial D$ corresponds to an eigenfunction of the eigenvalue $\tau_{3,1}$ in \eqref{eq:ca2}. Hence, the conditions in Definition~\ref{def:invi} are fulfilled, and one should have cloaking effect associated with the configuration described in \eqref{eq:ca1}--\eqref{eq:cloaking_source}. In Fig.~3, we present the slice plotting of the incident, total and scattered electric fields associated with the configuration described in \eqref{eq:ca1}--\eqref{eq:cloaking_source} for illustration of the cloaking effect. It is noted that the $z$-component of the wave field is zero and hence it is not plotted. One can readily see that the scattered field outside the plasmonic inclusion is nearly vanishing and hence the plasmonic inclusion is nearly invisible under such a wave impinging. Furthermore, we next consider the cloaking effect if there is an inhomogeneous inclusion located inside the plasmonic structure. Indeed, we let $B_{1/2}$ be a central ball of radius $1/2$ which embraces an inhomogeneous medium with $\epsilon=5$ and $\mu=1$. In Fig.~4, we present the slice plotting of the incident, total and scattered electric fields associated with the aforesaid configuration. Since the corresponding $z$-components are zero and we only plot the $x$- and $y$-components of the wave fields. It is interesting to note that not only the plasmonic structure but also the embedded inclusion are nearly invisible under such a wave impinging. The latter cloaking effect could also be rigorously verified, but that would involve much lengthy and complicated analysis and we leave it for the future investigation. Next, we show that if one imposes a certain restrictive condition on the impinging wave fields, one can construct a large class of plasmonic structures that can induce cloaking effect. 
 
 \begin{figure}\label{fig:cloaking_xy}
  \centering
  {\includegraphics[width=4.4cm]{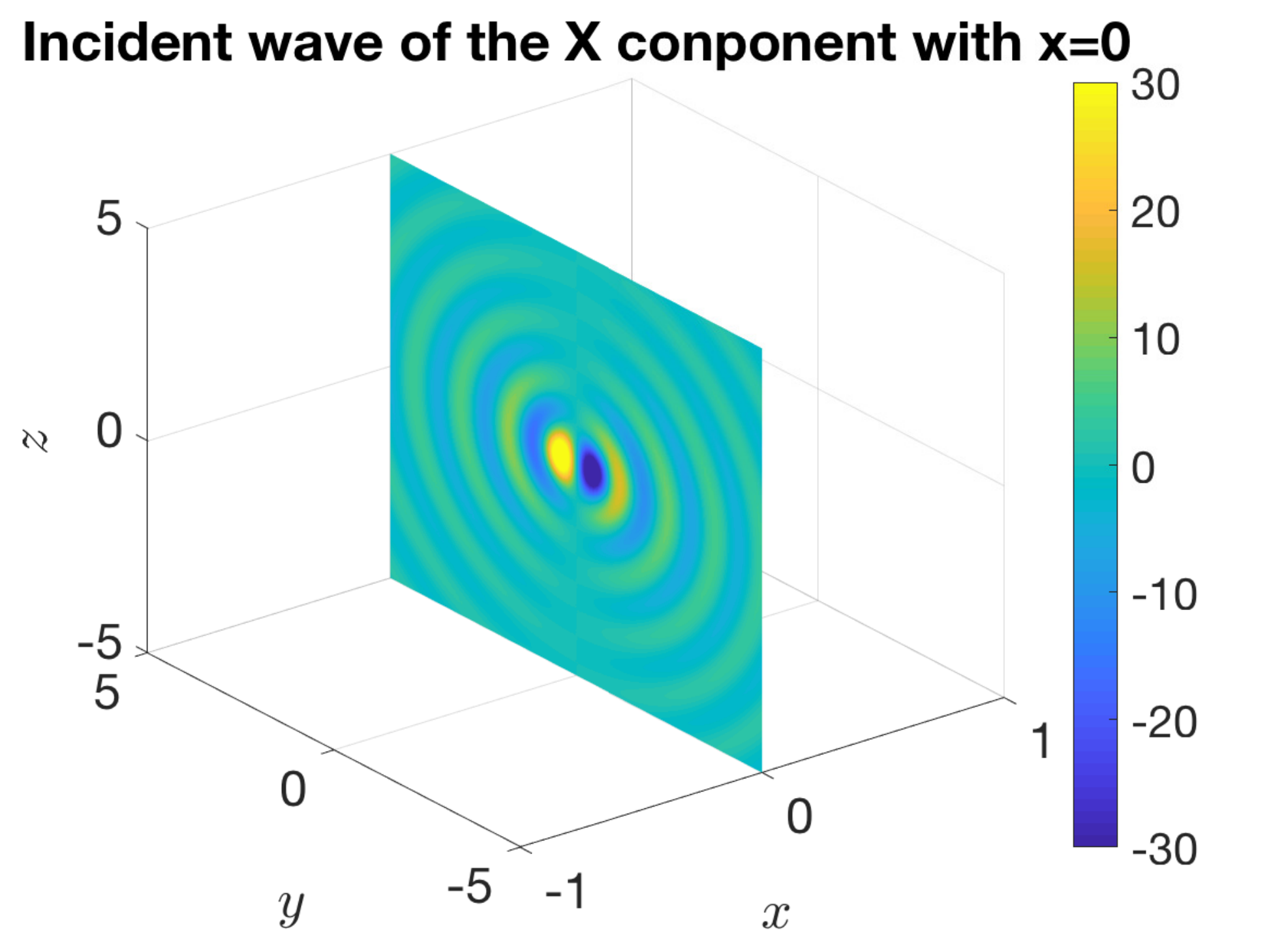}}
 {\includegraphics[width=4.4cm]{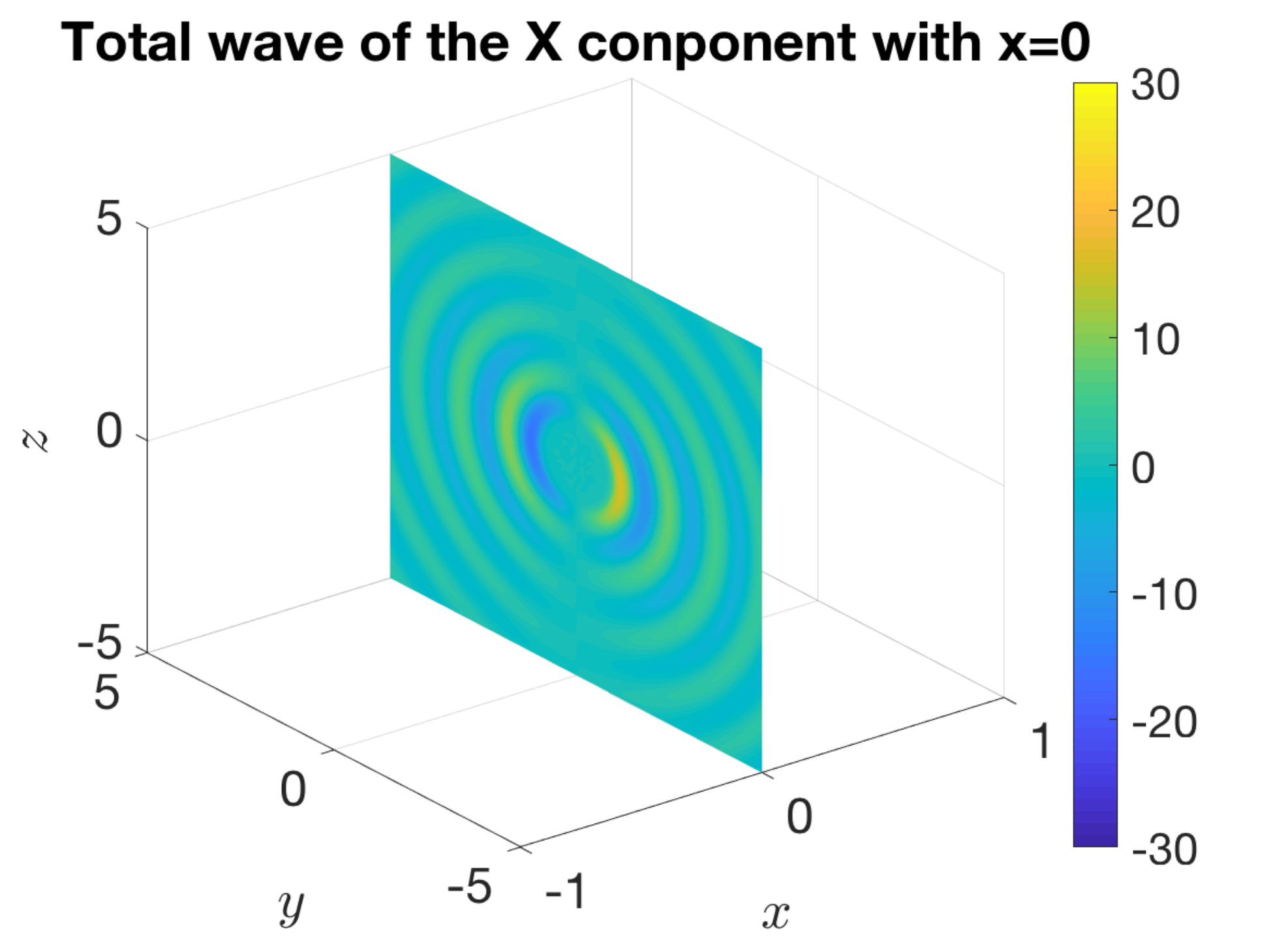}}
 {\includegraphics[width=4.4cm]{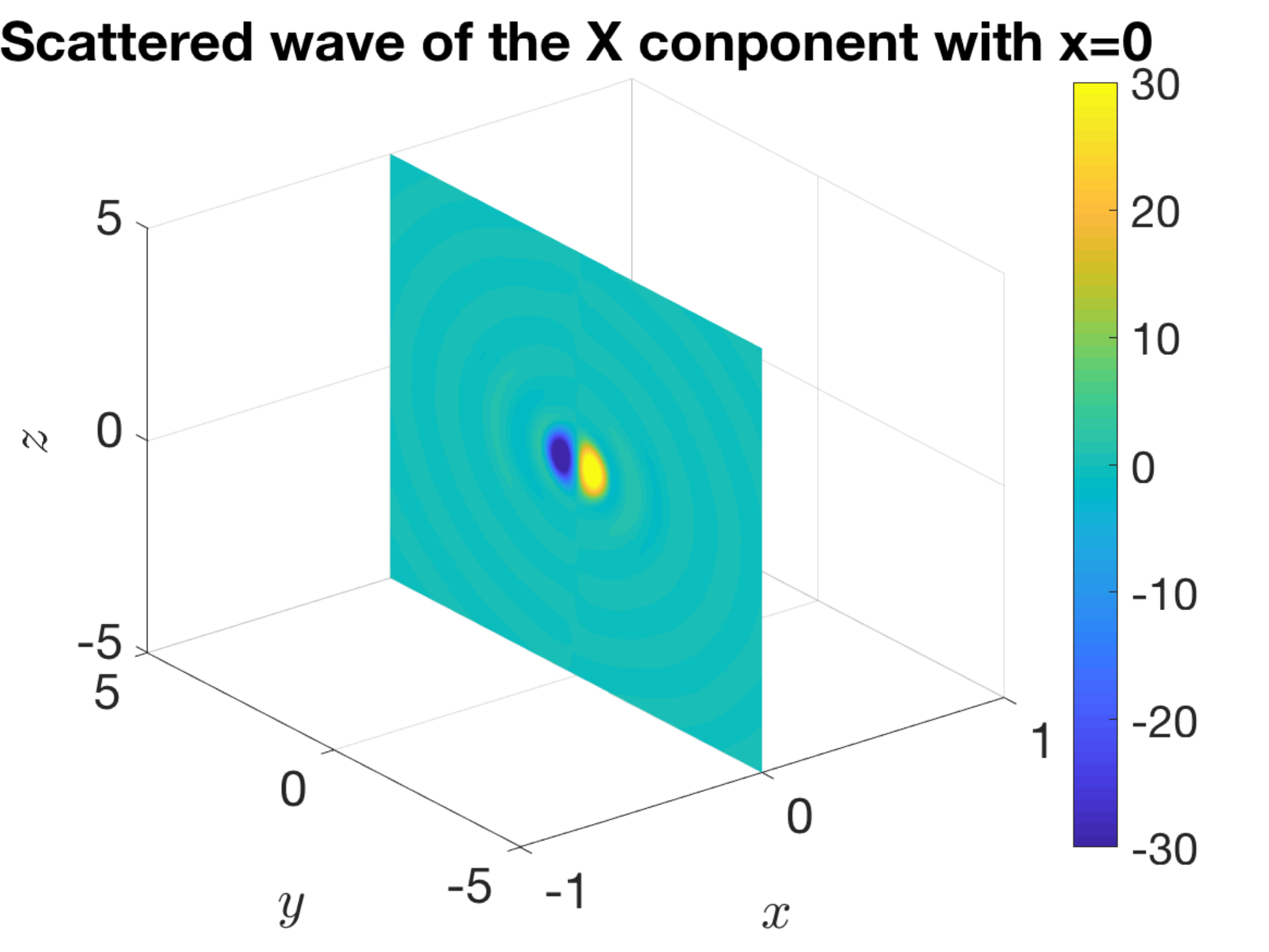}\\}
{\includegraphics[width=4.4cm]{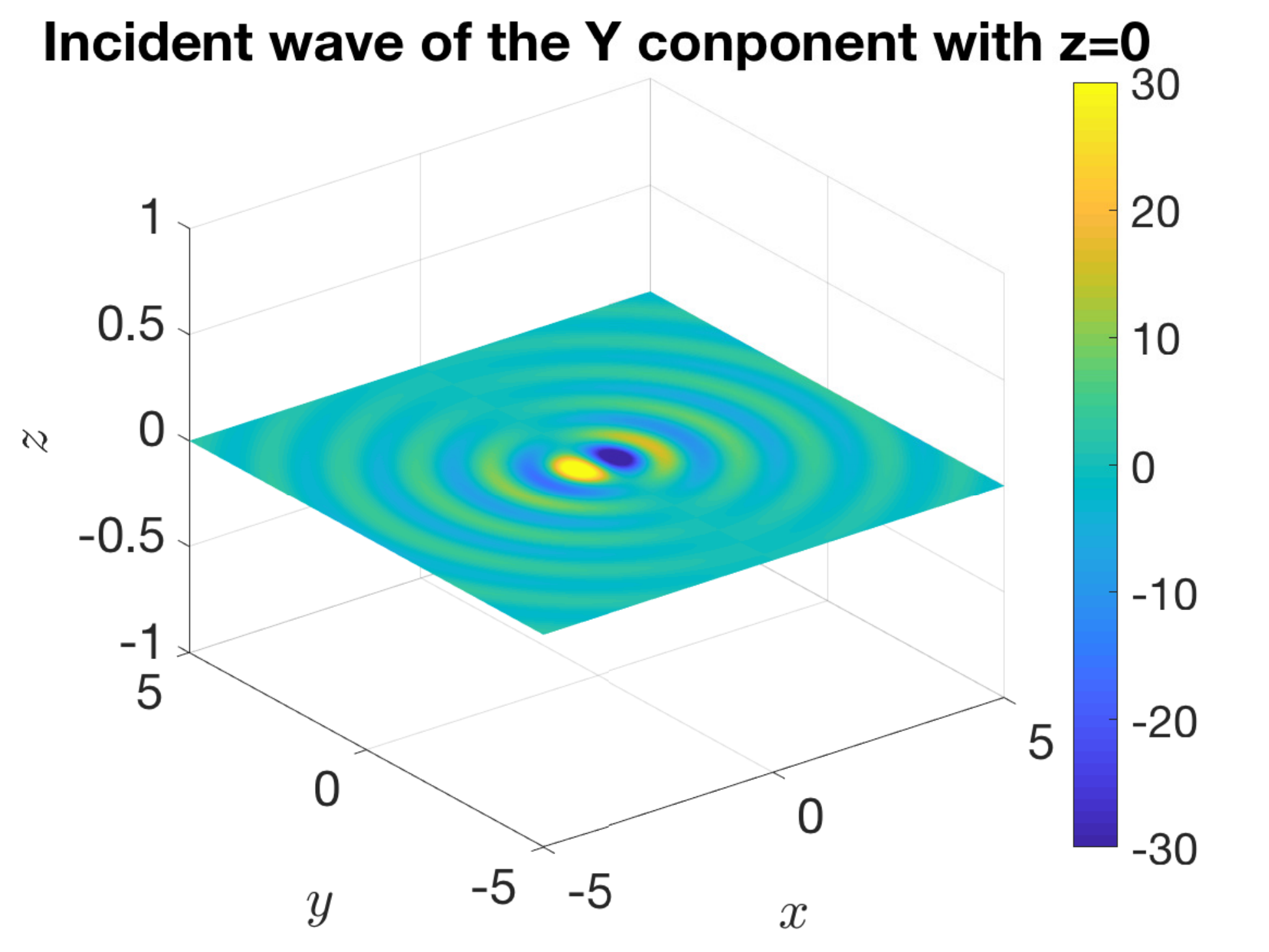}}
 {\includegraphics[width=4.4cm]{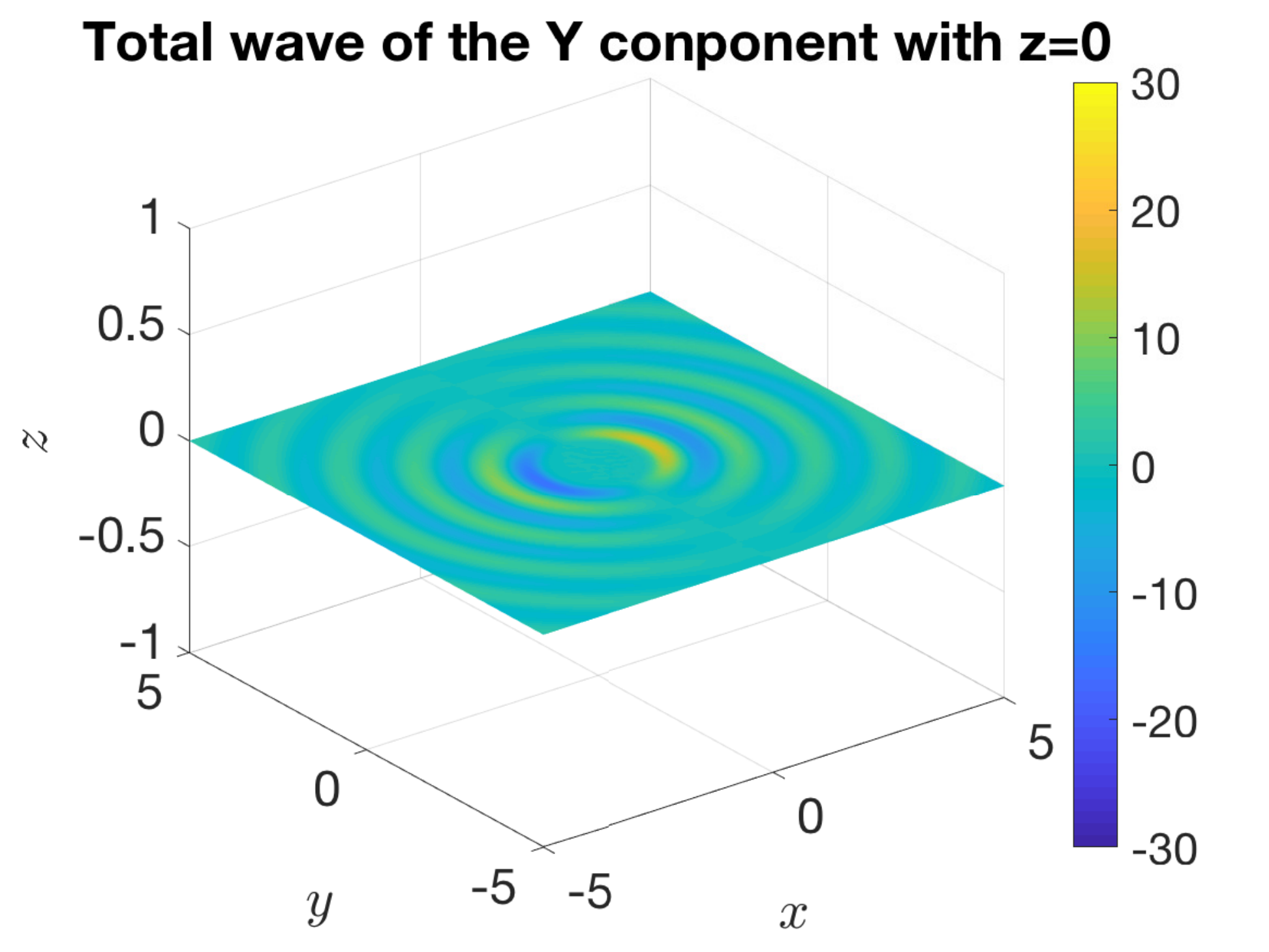}}
 {\includegraphics[width=4.4cm]{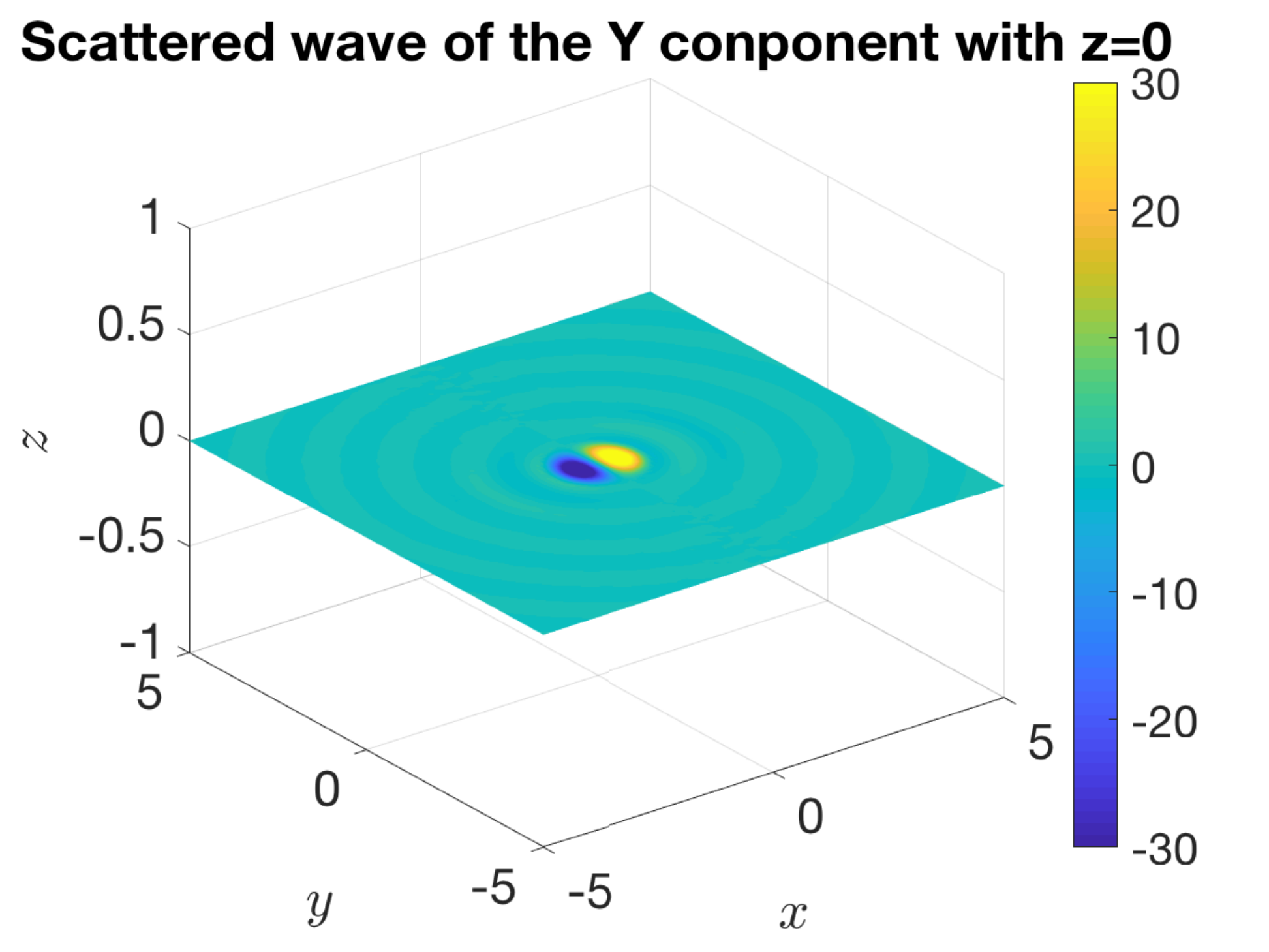}\\}
  \caption{Slice plotting of the $x$- and $y$-components of the incident, total and scattered electric fields associated to the electromagnetic configuration in \eqref{eq:ca1}--\eqref{eq:cloaking_source}. }
\end{figure}

\begin{figure}\label{fig:cloakingo_x}
  \centering
  {\includegraphics[width=4.4cm]{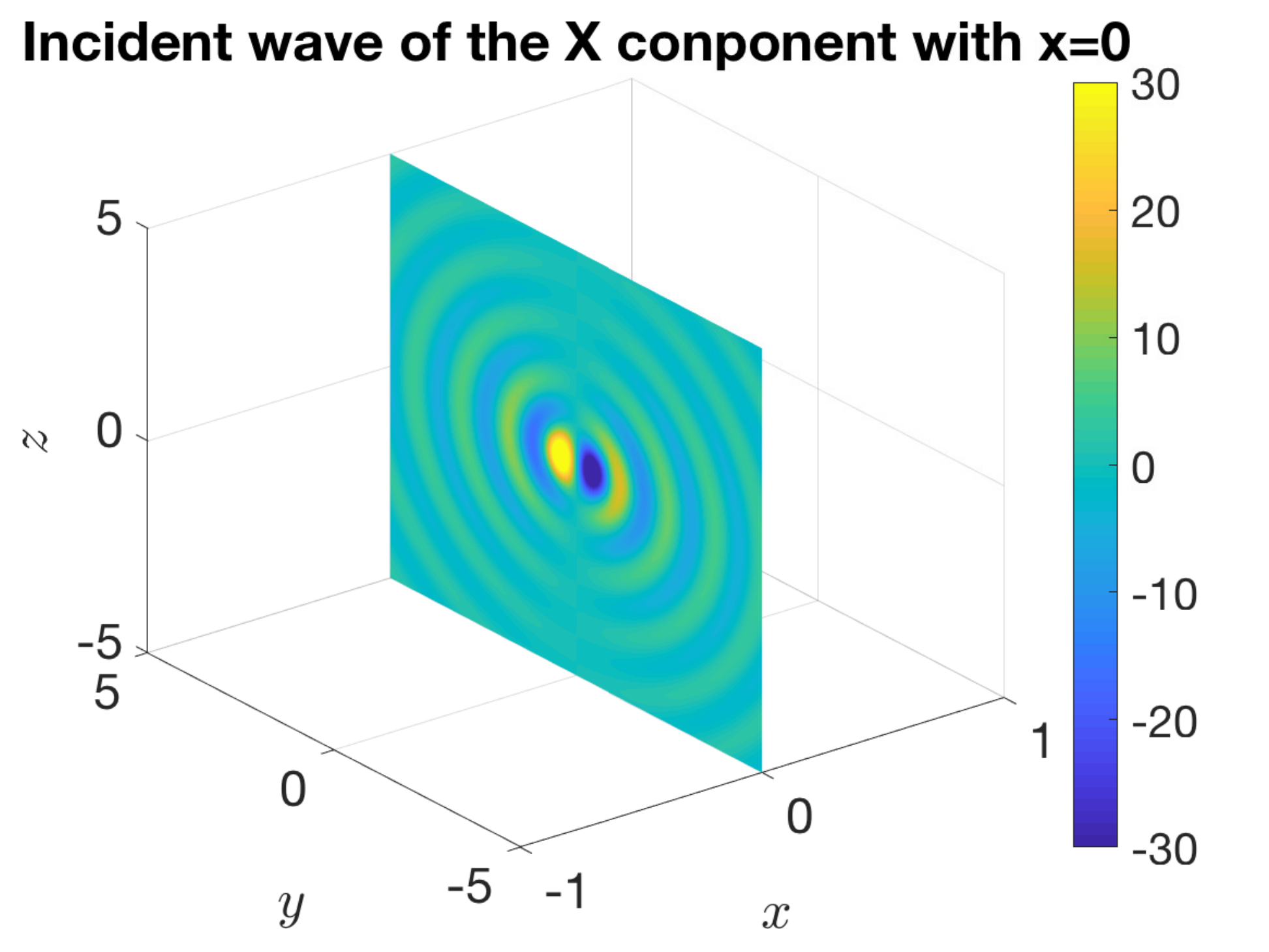}}
 {\includegraphics[width=4.4cm]{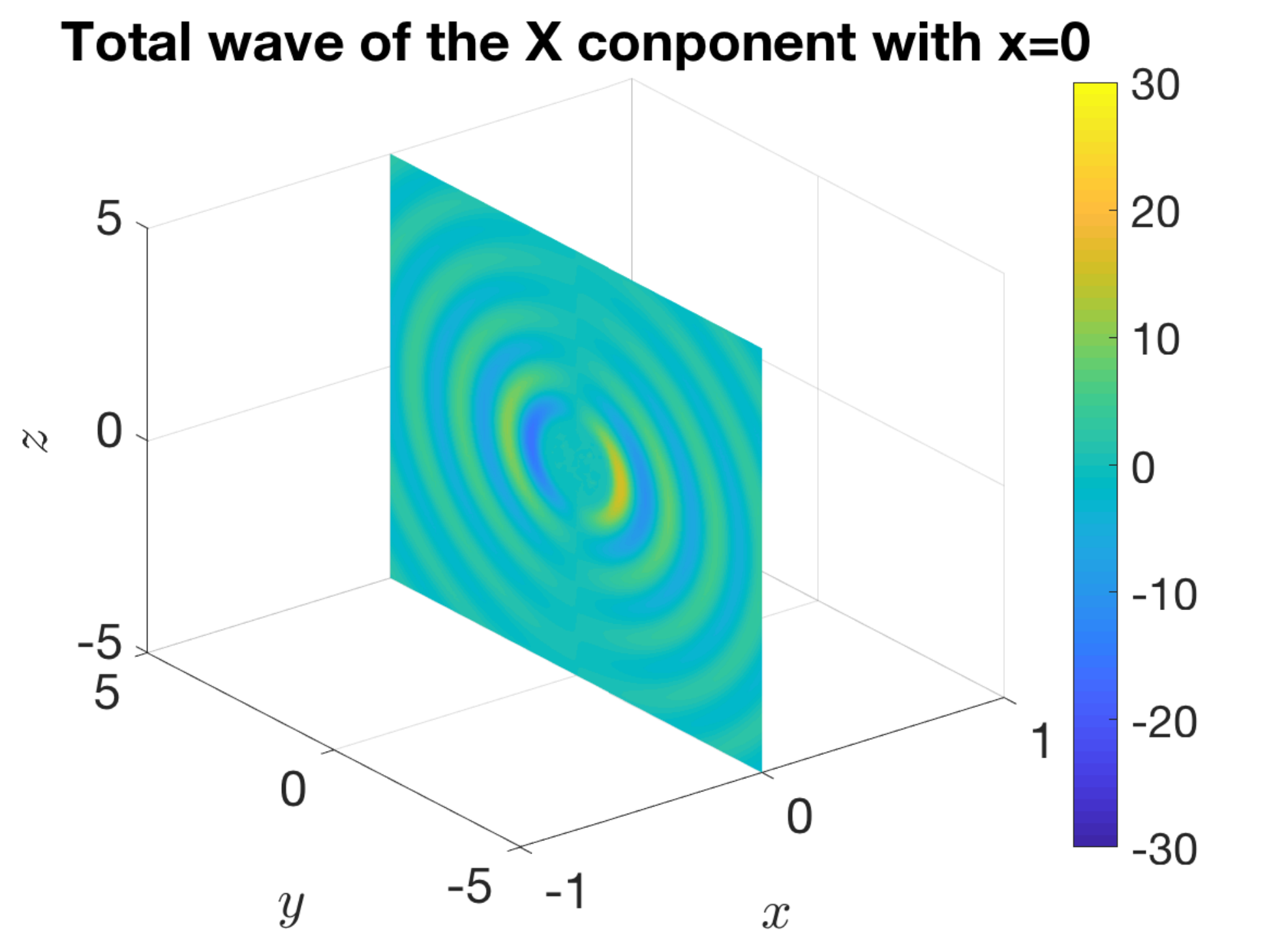}}
 {\includegraphics[width=4.4cm]{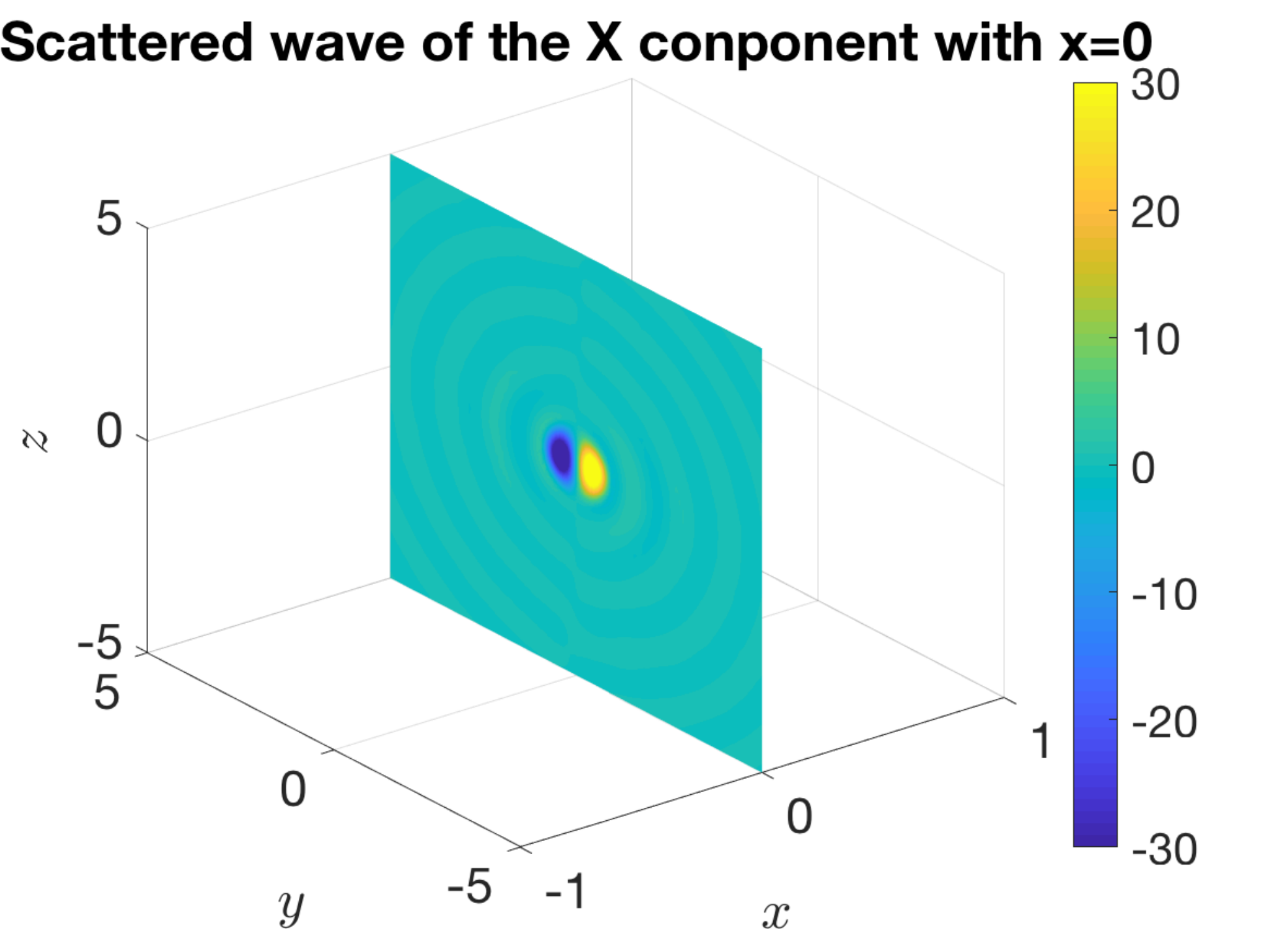}\\}
 {\includegraphics[width=4.4cm]{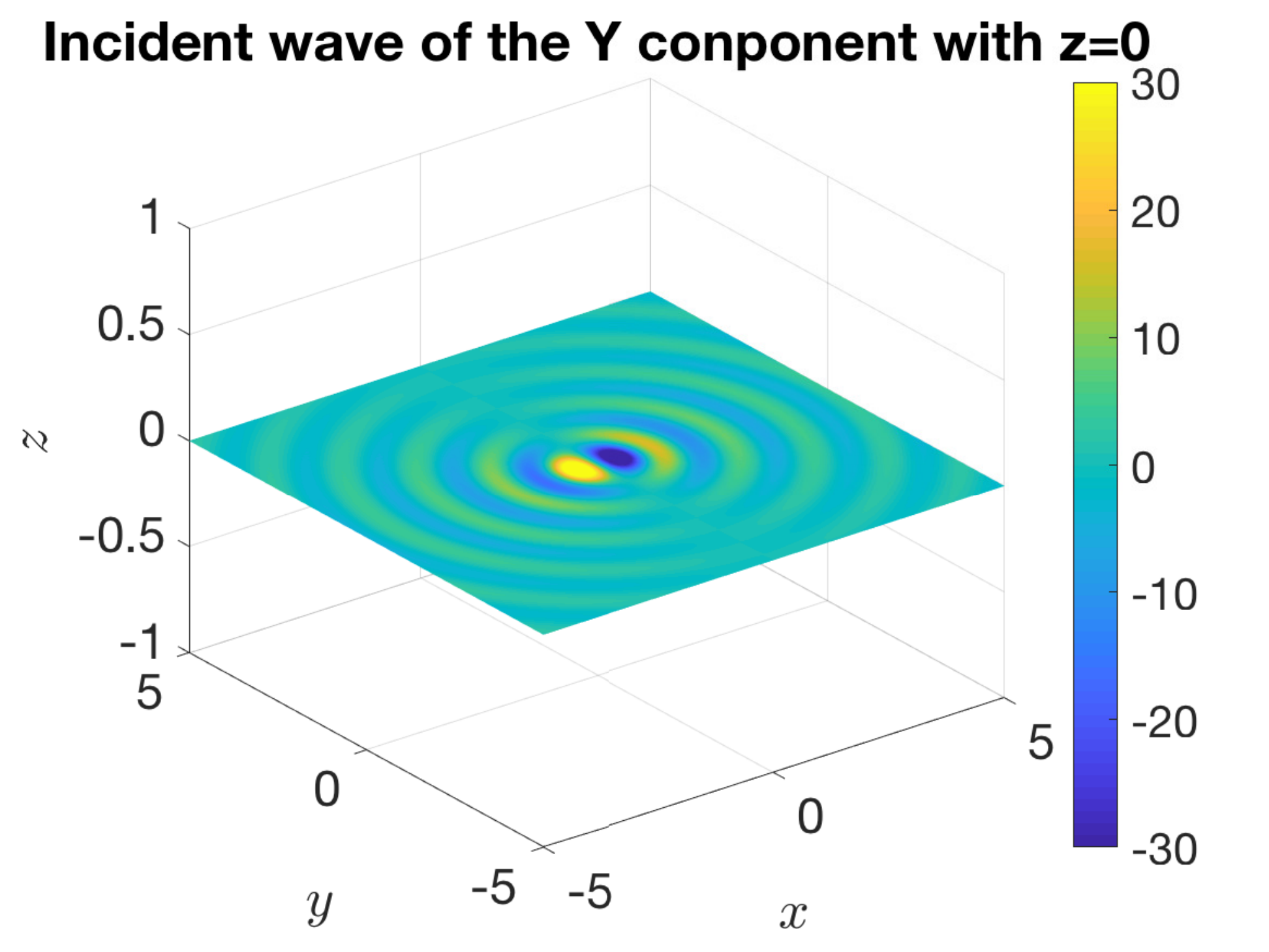}}
 {\includegraphics[width=4.4cm]{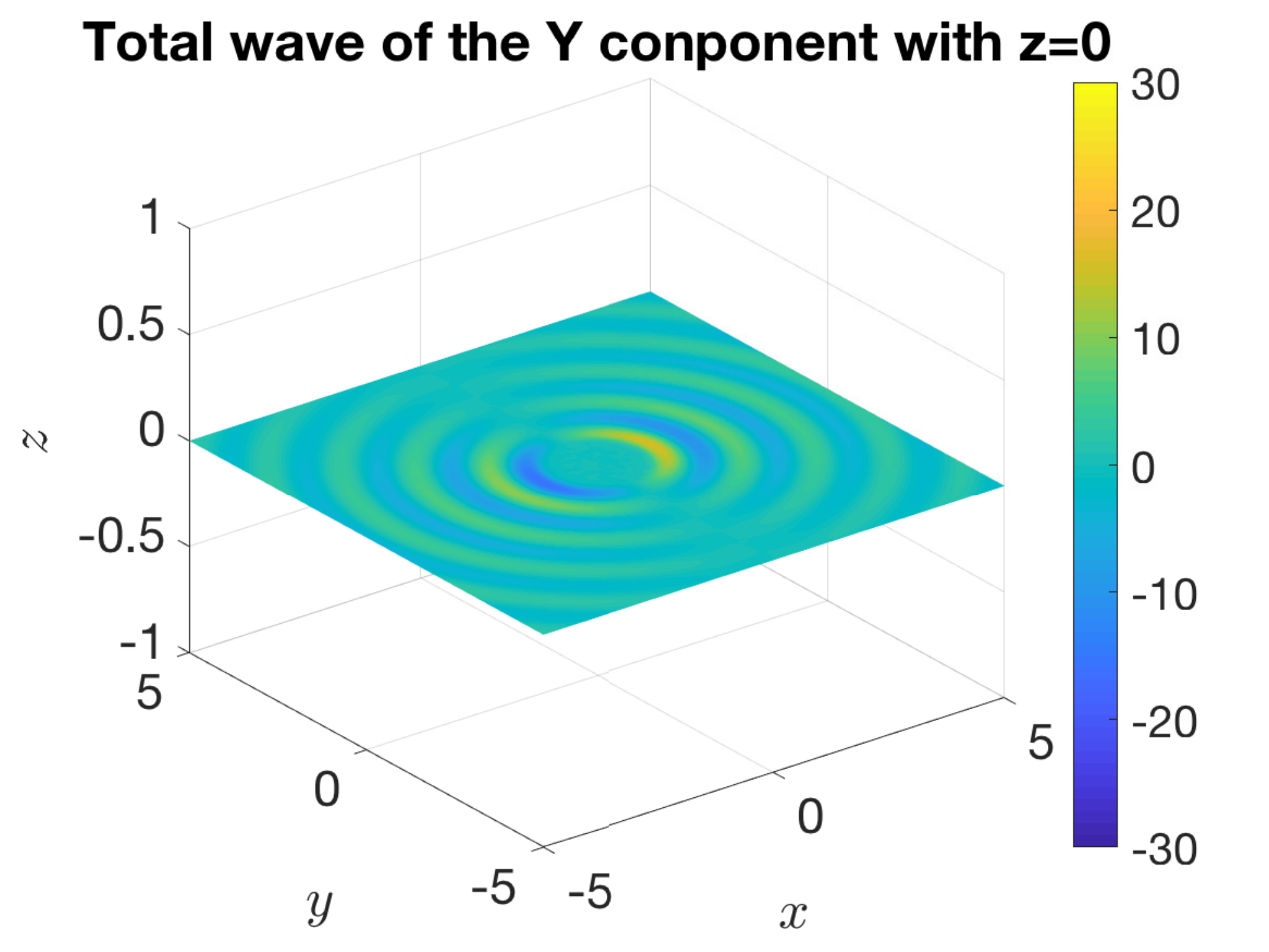}}
 {\includegraphics[width=4.4cm]{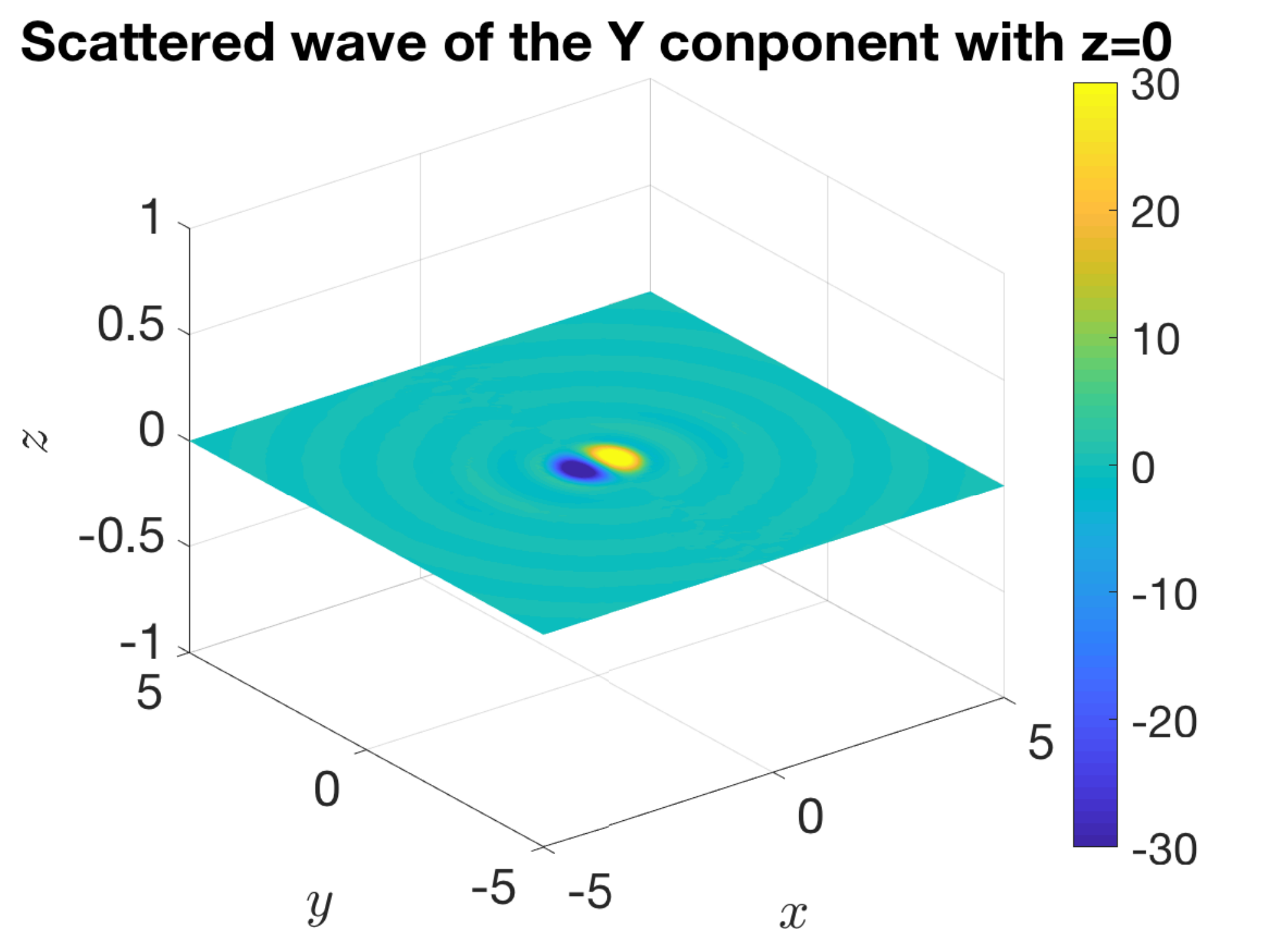}\\}
  \caption{Slice plotting of the $x$- and $y$-components of the incident, total and scattered electric fields associated with the electromagnetic configuration in \eqref{eq:ca1}--\eqref{eq:cloaking_source} that embraces an inclusion $(B_{1/2}; \epsilon=5, \mu=1)$ inside the plasmonic structure. }
\end{figure}

\begin{thm}\label{rem:invisible_1}
Suppose that $\mathbf{F}$ given in \eqref{eq:operator_eq} due to the incident wave has the following representation
\begin{equation}\label{eq:source_invisible1}
 \mathbf{F}=\sum_{i=1,3}\sum_{n=N}^{+\infty}f_{i,n}\bXi_{i,n},
\end{equation}
where $N\in\mathbb{N}$ is large enough such that when $n\geq N$ the spherical Bessel and Hankel functions $j_n(t)$ and $h_n^{(1)}(t)$ enjoy the asymptotic properties given in \eqref{eq:asymptotic_j} and \eqref{eq:asymptotic_h}. Let the plasmonic parameters $\epsilon_c$ and $\mu_c$ inside the domain $D$ and the wave frequency $\omega\in\mathbb{R}_+$ be chosen such that
\begin{equation}\label{eq:re01}
  \Re\left( (\mu_c+\mu_m) -(\epsilon_c + \epsilon_m)\omega^2 \right) \geq 0 \quad \mbox{and} \quad |(\epsilon_c + \epsilon_m)\omega^2|\gg 1,
\end{equation}
or 
\begin{equation}\label{eq:re02}
 \Re\left( (\mu_c+\mu_m) -(\epsilon_c + \epsilon_m)\omega^2 \right) < 0 \quad \mbox{and} \quad |\mu_c+\mu_m|\gg 1,
\end{equation}
then the corresponding scattered wave field is nearly vanishing outside $D$. That is, the plasmonic inclusion $D$ is nearly invisible. 
\end{thm}
\begin{proof}
By following a similar argument to the proof of Theorem \ref{thm:reson_n}, one can show that when $n\gg1$ and 
\[
  \Re\left( (\mu_c+\mu_m) -(\epsilon_c + \epsilon_m)\omega^2 \right) \geq 0,
\]
there holds
\begin{equation}\label{eq:ff1}
  \tilde{\tau}_{1,n}=\frac{1}{2}(\epsilon_c + \epsilon_m)\omega^2 + \mathcal{O}\left(\frac{1}{n}\right).
\end{equation}
Then by using \eqref{eq:ff1} and the second condition in \eqref{eq:re01}, one has 
\begin{equation}\label{eq:ff2}
 | \tilde{\tau}_{1,n}|\gg 1,
\end{equation}
which holds for $n\in\mathbb{N}$ sufficiently large. Clearly, the conditions in Definition~\ref{def:invi} are fulfilled and hence one has the cloaking effect. The other case in \eqref{eq:re02} can be proved in a similar manner. 

The proof is complete. 
\end{proof}

\begin{thm}\label{rem:invisible_2}
Suppose that $\mathbf{F}$ given in \eqref{eq:operator_eq} due to the incident wave has the following representation
\begin{equation}\label{eq:source_invisible2}
 \mathbf{F}=\sum_{i=2,4}\sum_{n=N}^{+\infty}f_{i,n}\bXi_{i,n},
\end{equation}
where $N\in\mathbb{N}$ is large enough such that when $n\geq N$ the spherical Bessel and Hankel functions $j_n(t)$ and $h_n^{(1)}(t)$ enjoy the asymptotic properties given in \eqref{eq:asymptotic_j} and \eqref{eq:asymptotic_h}. Let the plasmonic parameters $\epsilon_c$ and $\mu_c$ inside the domain $D$ and the wave frequency $\omega\in\mathbb{R}_+$ be chosen such that
\begin{equation}\label{eq:re03}
  \Re\left( (\mu_c+\mu_m) -(\epsilon_c + \epsilon_m)\omega^2 \right) \geq 0, \quad \mbox{and} \quad  |\mu_c+\mu_m|\gg 1,
\end{equation}
or 
\begin{equation}\label{eq:re04}
 \Re\left( (\mu_c+\mu_m) -(\epsilon_c + \epsilon_m)\omega^2 \right) < 0, \quad \mbox{and} \quad |(\epsilon_c + \epsilon_m)\omega^2|\gg 1,
\end{equation}
then the corresponding scattered wave field is nearly vanishing outside $D$. That is, the plasmonic inclusion $D$ is nearly invisible. 
\end{thm}
\begin{proof}
The proof is analogous to that of the Theorem \ref{rem:invisible_1}.
\end{proof}

\begin{rem}\label{rem:cc1}
It is remarked that the conditions \eqref{eq:re01} or \eqref{eq:re02} in Theorem~\ref{rem:invisible_1} can be easily fulfilled, say e.g., one can simply choose the parameter $\epsilon_c$ such that 
\[
 \epsilon_c < -\epsilon_m,
\]
and the source $\mathbf{F}$ is of the form \eqref{eq:source_invisible1} with $\omega\gg1$, then \eqref{eq:re01} is satisfied. The same remark holds for the other conditions in Theorems~\ref{rem:invisible_1} and \ref{rem:invisible_2}.
\end{rem}

\begin{rem}\label{rem:cc2}
By Theorem \ref{rem:invisible_1} and \ref{rem:invisible_2}, one can conclude that if the parameters $\epsilon_c$ and $\mu_c$ inside the domain $D$, and the frequency $\omega$ satisfy the following condition
\begin{equation}\label{eq:cc1}
 |\mu_c+\mu_m|\gg 1 \quad \mbox{and}\quad  |(\epsilon_c + \epsilon_m)\omega^2|\gg 1,
\end{equation}
then the inclusion $(D; \epsilon_c,\mu_c)$ is nearly invisible to a general source term $\mathbf{F}$ of the form
\[
  \mathbf{F}=\sum_{i=1}^4\sum_{n=N}^{+\infty}f_{i,n}\bXi_{i,n},
\]
where $N\in\mathbb{N}$ is large enough such that the spherical Bessel and Hankel functions $j_n(t)$ and $h_n^{(1)}(t)$ enjoy the asymptotic properties given in \eqref{eq:asymptotic_j} and \eqref{eq:asymptotic_h}. However, it is also noted that the first condition in \eqref{eq:cc1} may not be practical and hence this remark may be mainly of theoretical interest. 
\end{rem}

By Remarks~\ref{rem:cc1} and \ref{rem:cc2}, one can readily see that the high-frequency of the electromagnetic wave plays a critical role for such a cloaking phenomenon and it may not be so realistic to consider such a phenomenon in the quasi-static regime. Finally, we present plasmonic constructions following the Drude model \eqref{eq:drude} that can induce the cloaking effect. 

 If we take 
\[
 \epsilon_0=\mu_0=1,\quad \omega=5,\quad \tau=6.615\times10^{-7}, \quad \omega_0=2,\quad \omega_p^2=188.952 \quad \mbox{and}\quad \mathcal{F}=0,
\]
then by \eqref{eq:drude} one can obtain that 
\[
 \epsilon_c=-6.55806+0.000001\rmi \quad \mbox{and} \quad \mu_c=1,
\]
which is exactly the one constructed before in \eqref{eq:ca1} and its cloaking effect has been shown in Figs.~3 and 4. Next, we show a case with a nonzero filling factor by taking 
\begin{equation}\label{eq:gg1}
 \epsilon_0=\mu_0=1,\quad \omega=5,\quad \tau=0.00001, \quad\omega_0=2,\quad \omega_p^2=186.769\quad \mbox{and}\quad \mathcal{F}=0.02,
\end{equation}
then by \eqref{eq:drude} one can obtain that 
\begin{equation}\label{eq:gg2}
 \epsilon_c=-6.47076+0.00001494\rmi \quad \mbox{and} \quad \mu_c=0.97619 + 5.66893\times10^{-8}\rmi.
\end{equation}
Let the incident wave $\bE^i$ be given in \eqref{eq:cloaking_source}. The slice plottings of the $x$- and $y$-components of the incident, total and scattered electric fields associated with the afore-described configuration are, respectively, presented in Figs.~5 and 6. It is noted that the corresponding $z$-components are all zero. Apparently, there are cloaking effects observed. 

\begin{figure}\label{fig:reson_drude_x}
  \centering
  {\includegraphics[width=4.4cm]{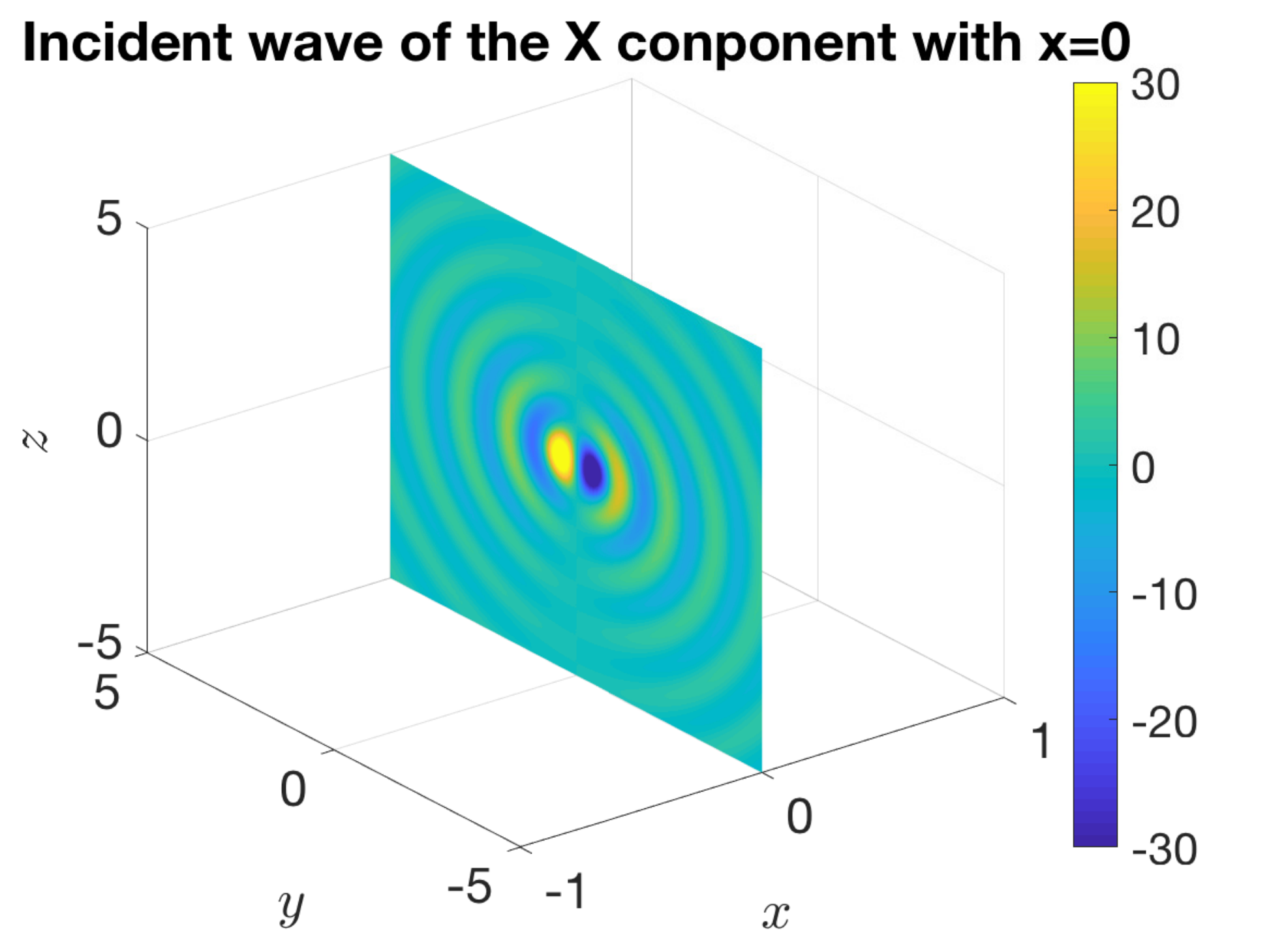}}
 {\includegraphics[width=4.4cm]{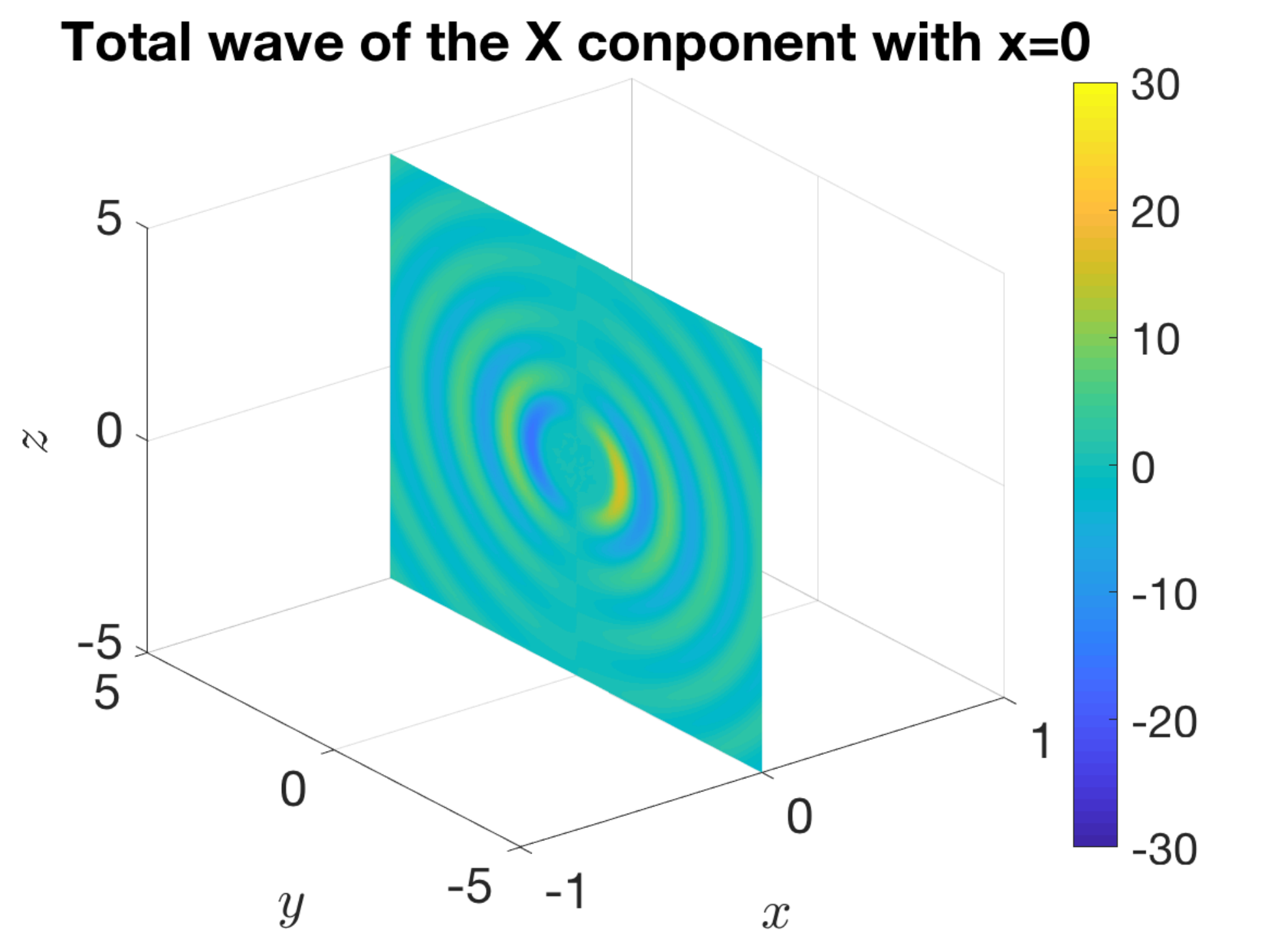}}
 {\includegraphics[width=4.4cm]{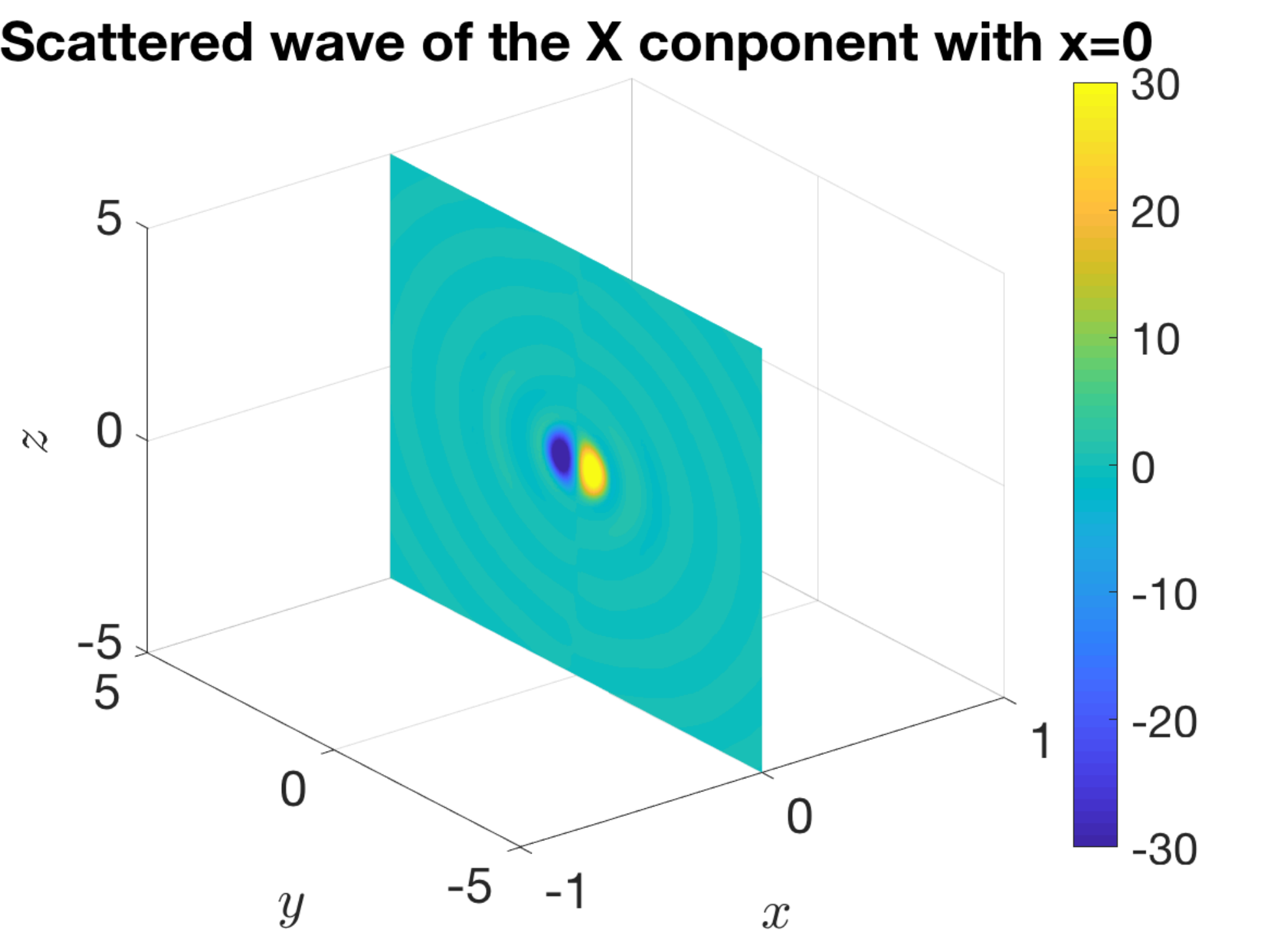}\\}
 {\includegraphics[width=4.4cm]{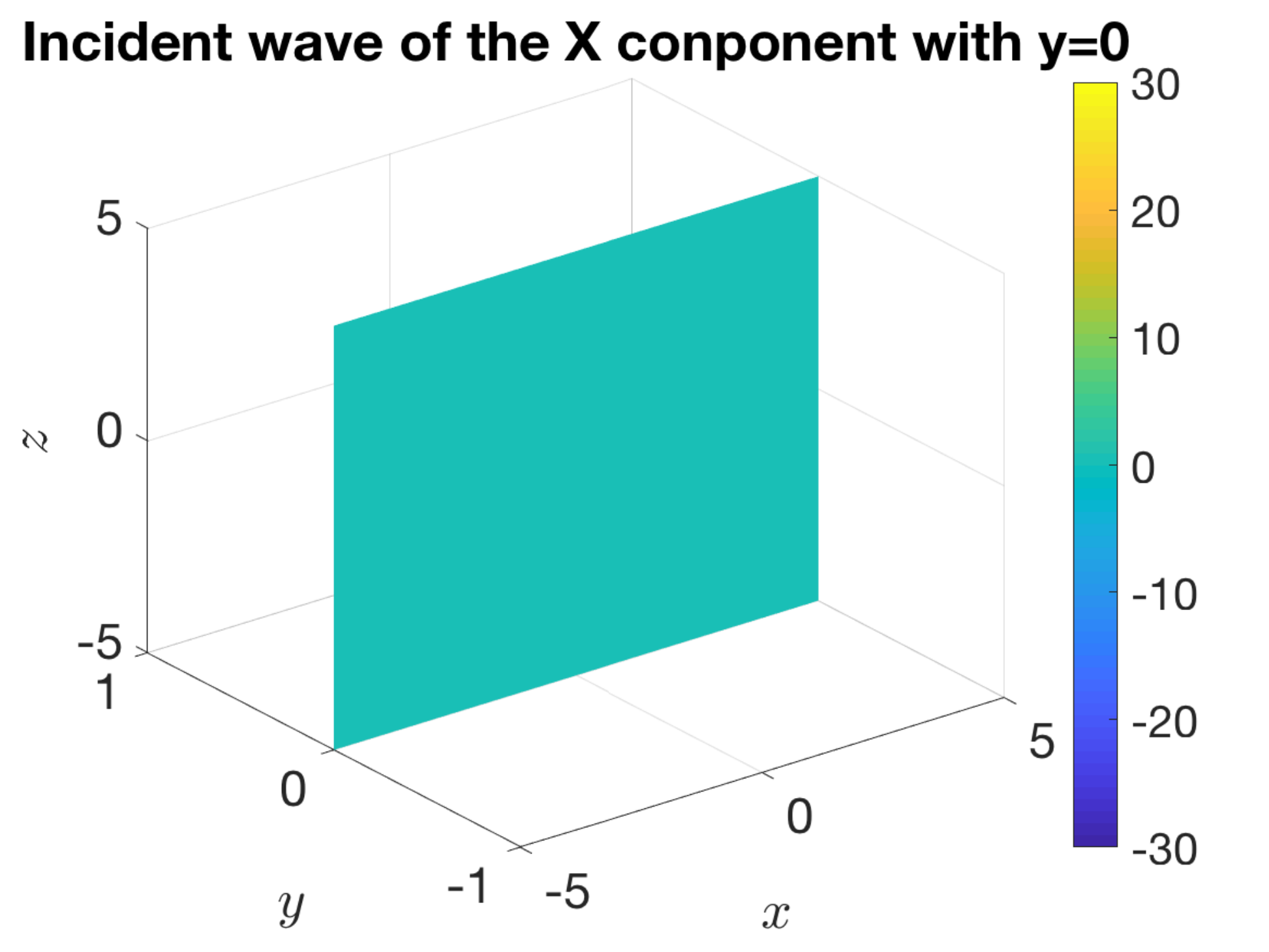}}
 {\includegraphics[width=4.4cm]{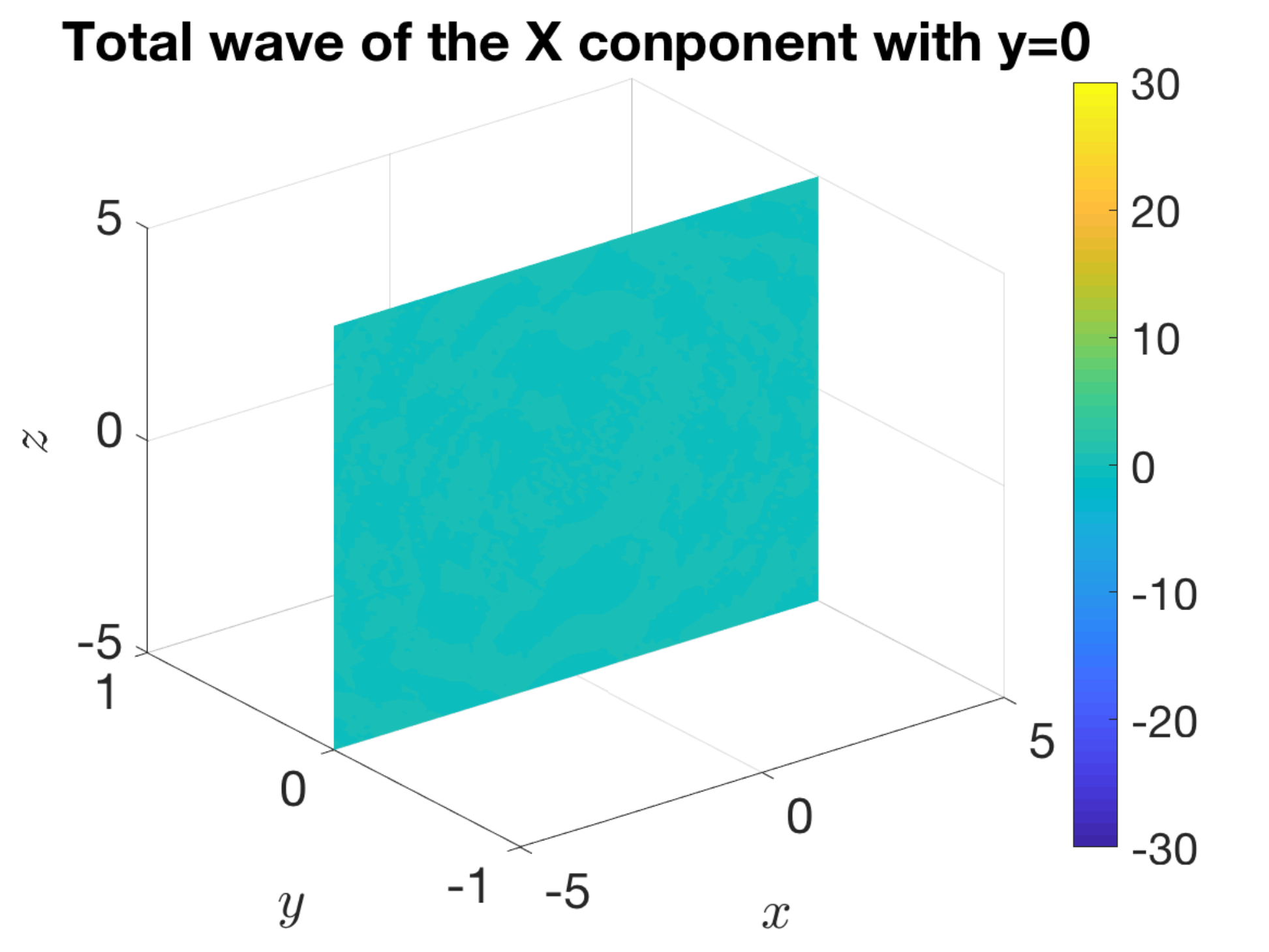}}
 {\includegraphics[width=4.4cm]{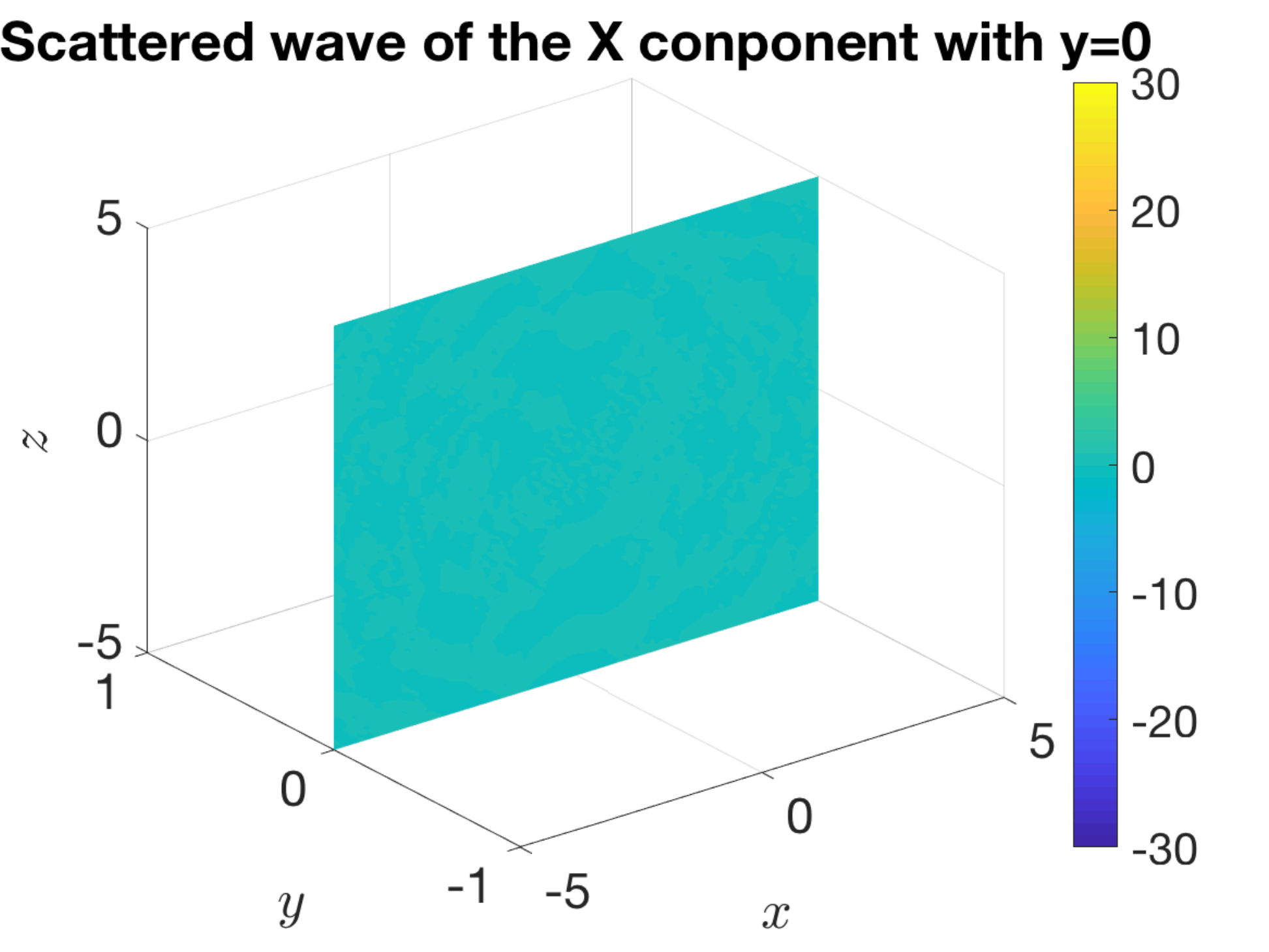}\\}
 {\includegraphics[width=4.4cm]{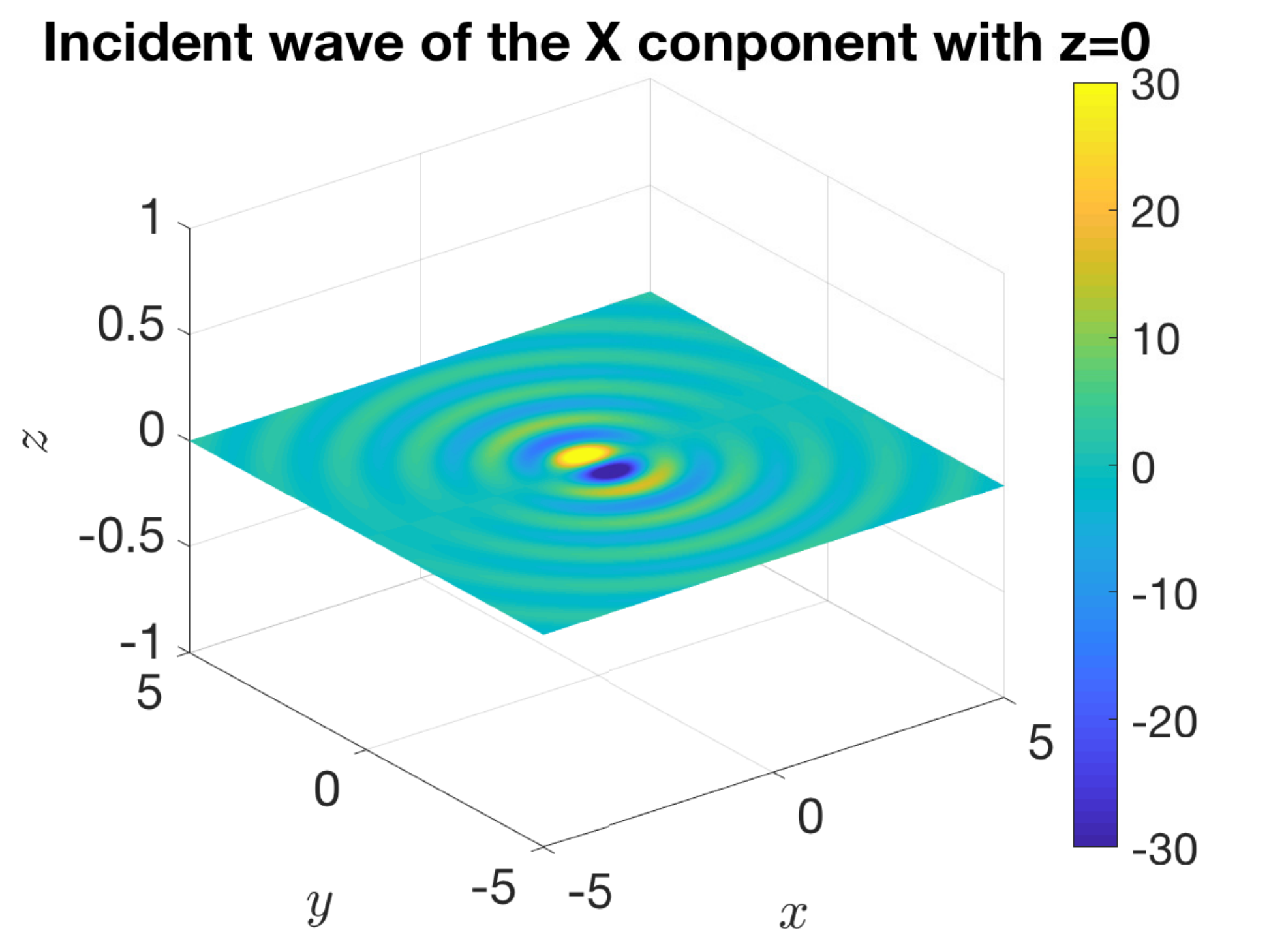}}
 {\includegraphics[width=4.4cm]{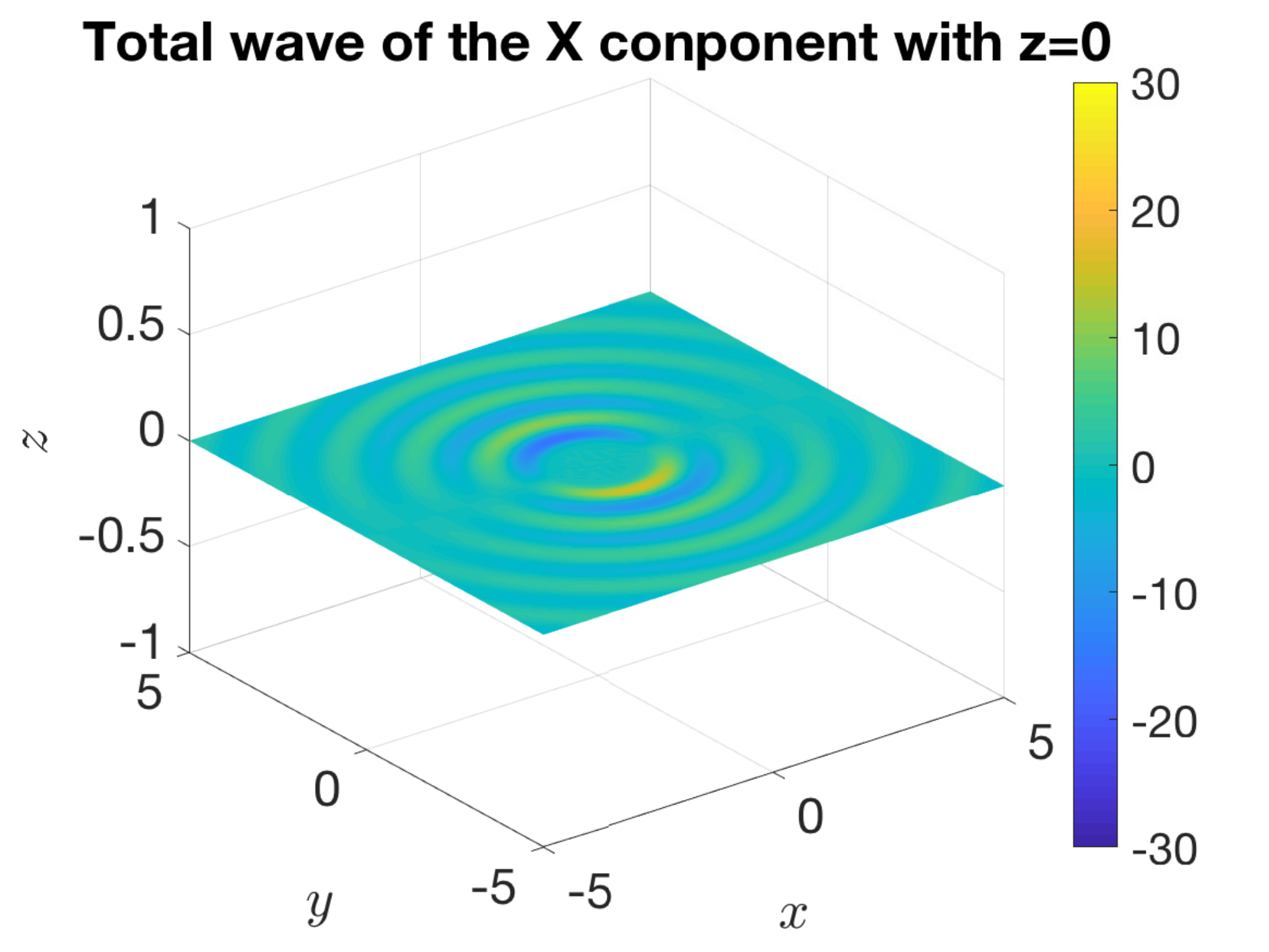}}
 {\includegraphics[width=4.4cm]{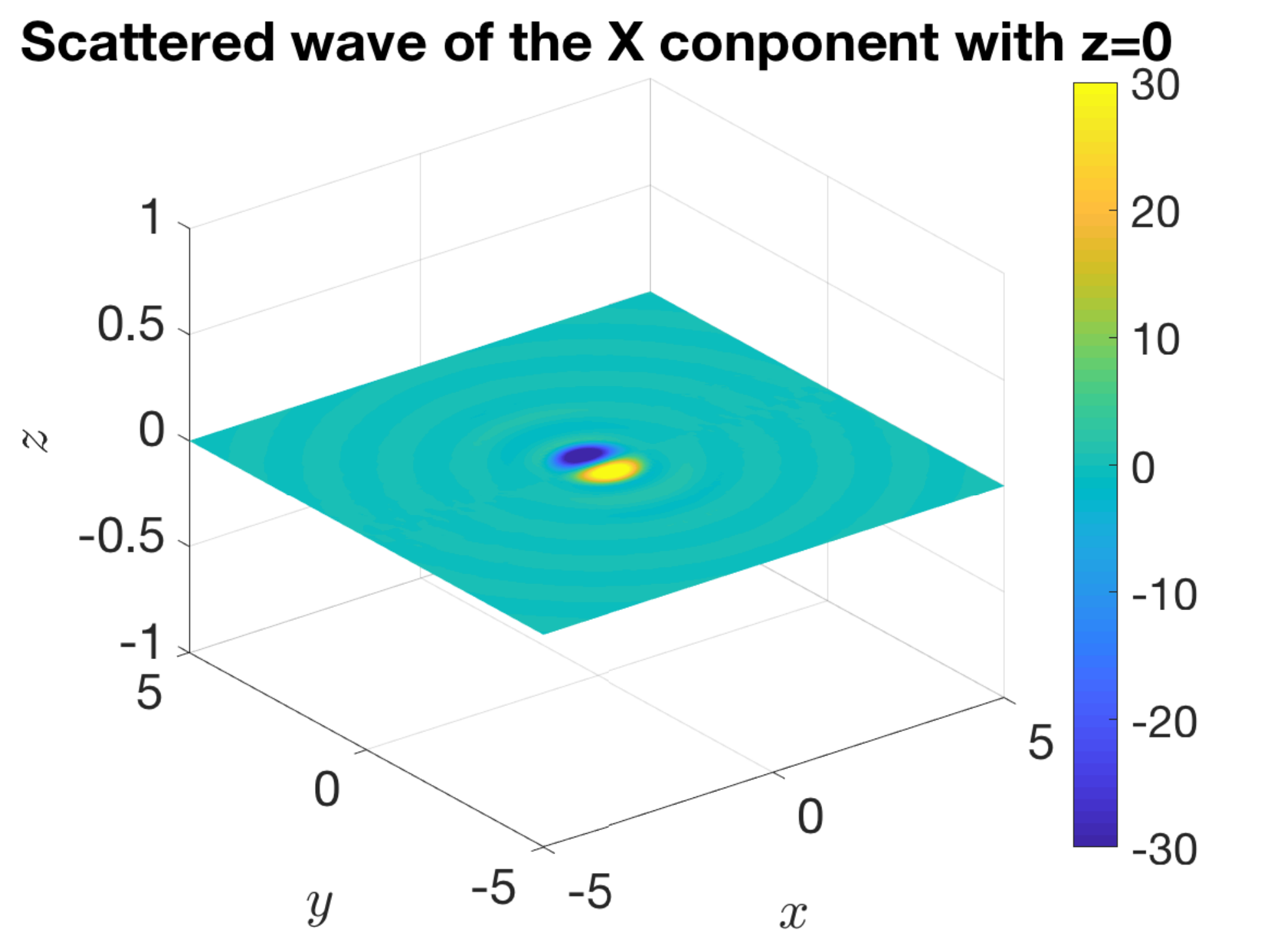}\\}
  \caption{Slice plotting of the $x$-components of the incident, total and scattered electric fields associated to the electromagnetic configuration described in \eqref{eq:gg1}, \eqref{eq:gg2} and \eqref{eq:cloaking_source}. }
\end{figure}

\begin{figure}\label{fig:reson_drude_y}
  \centering
  {\includegraphics[width=4.4cm]{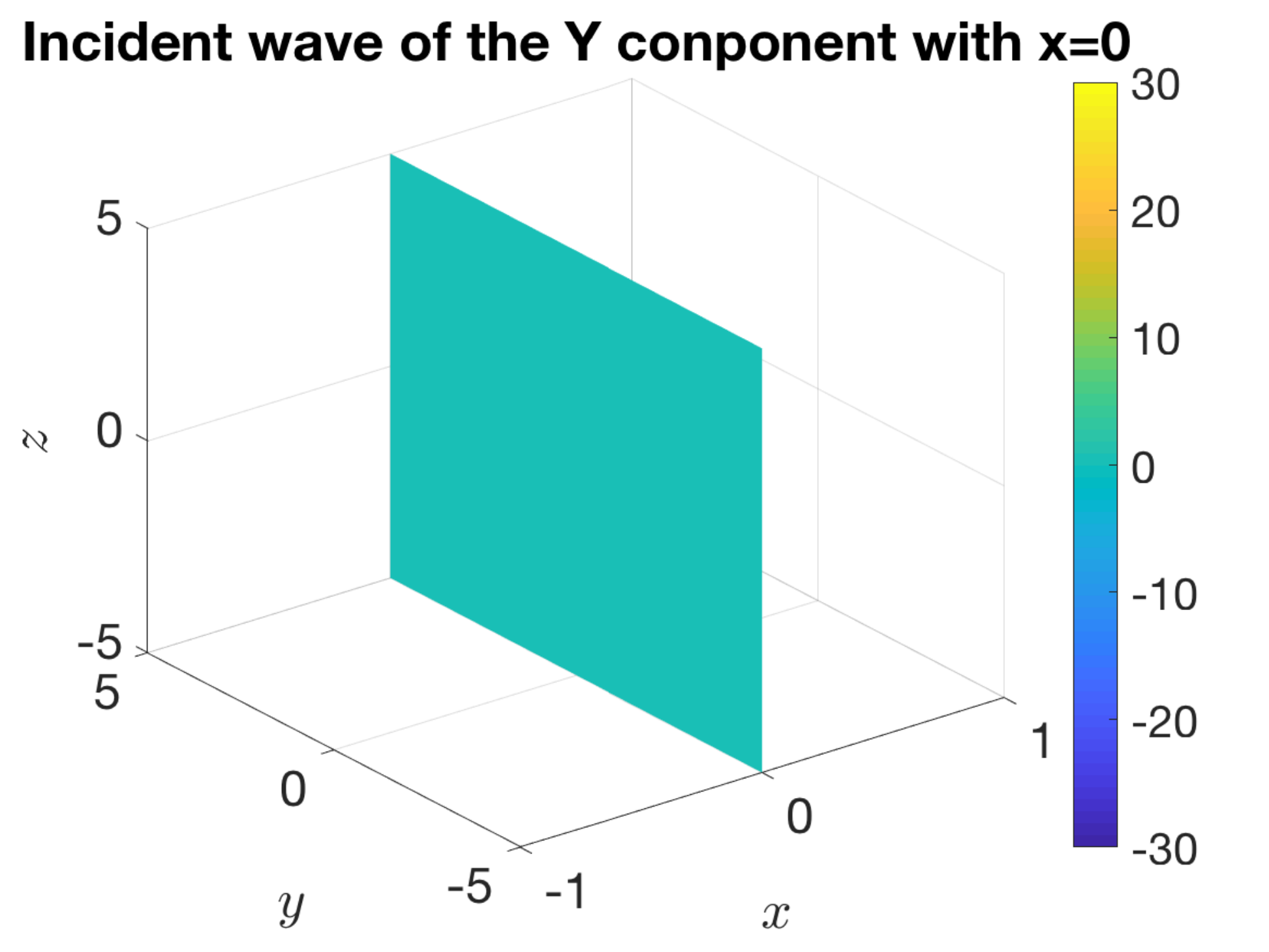}}
 {\includegraphics[width=4.4cm]{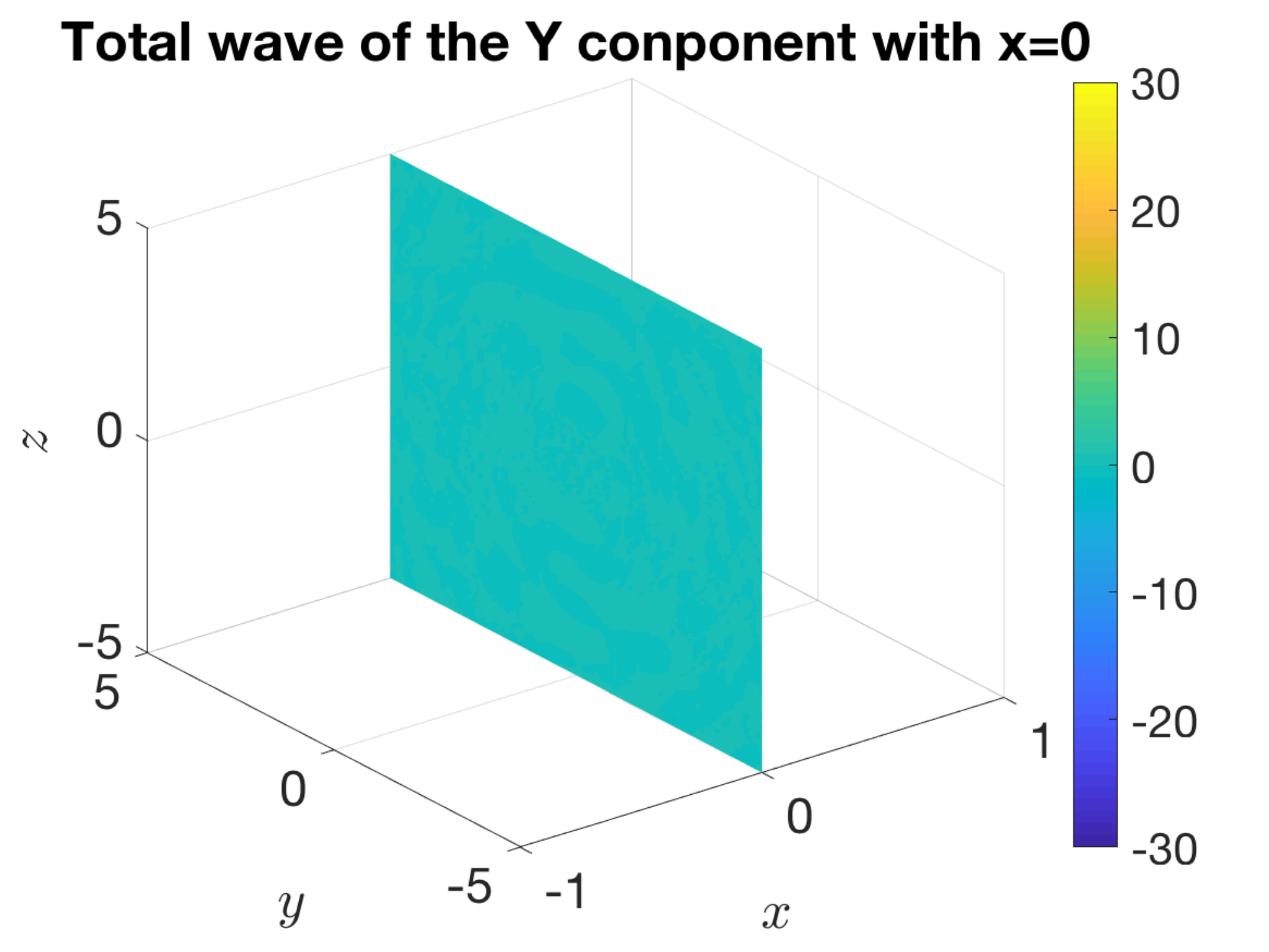}}
 {\includegraphics[width=4.4cm]{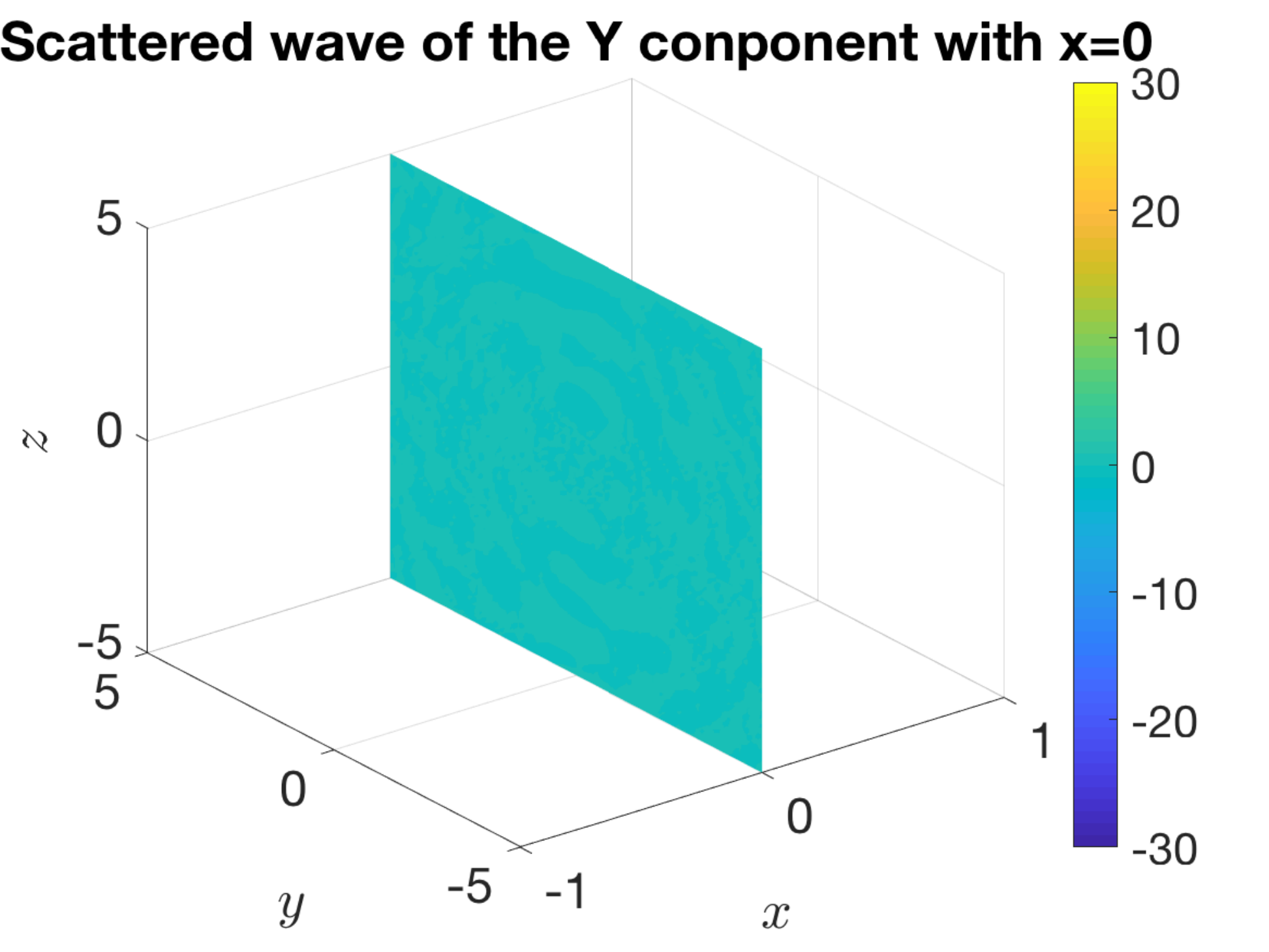}\\}
 {\includegraphics[width=4.4cm]{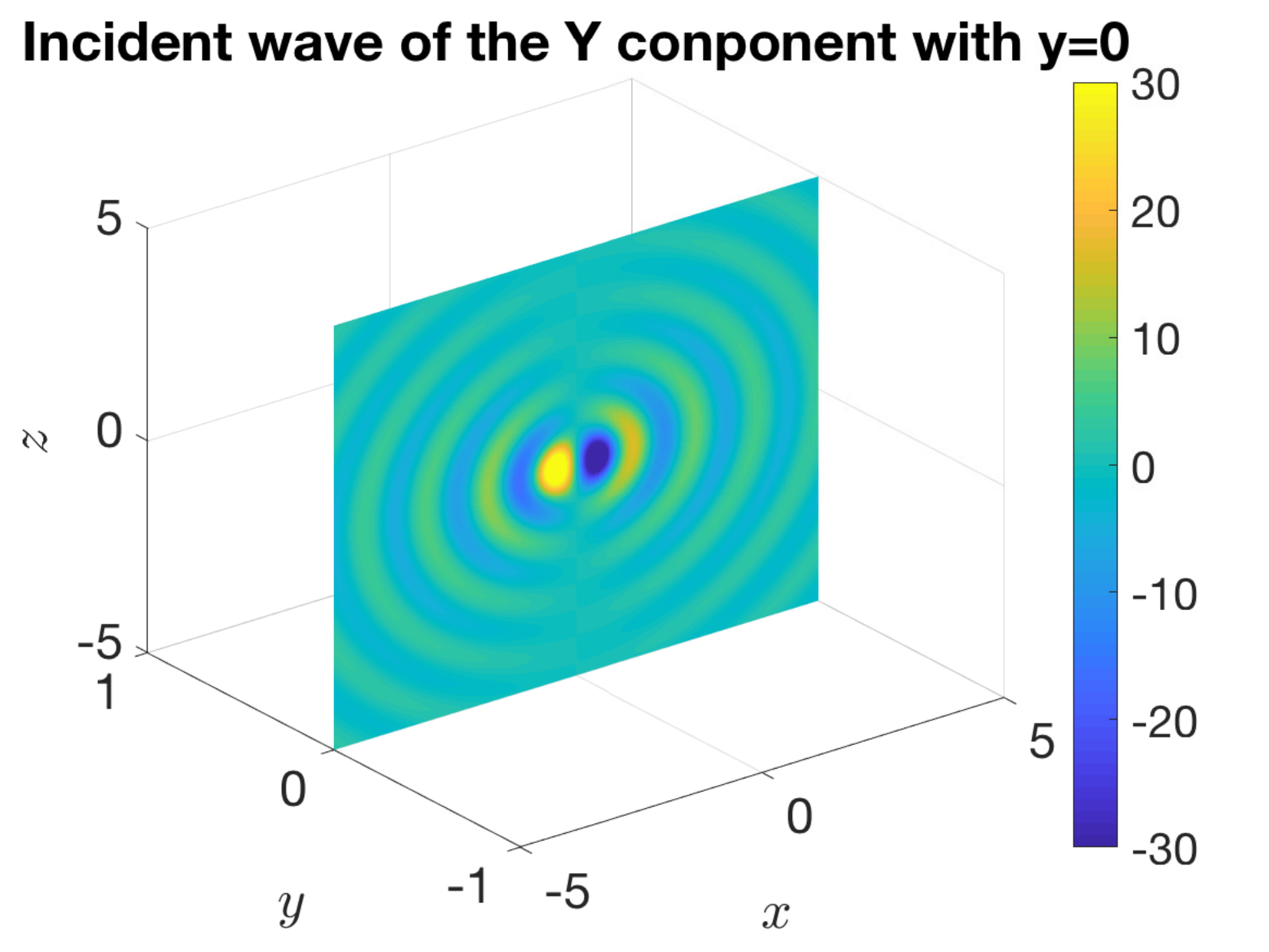}}
 {\includegraphics[width=4.4cm]{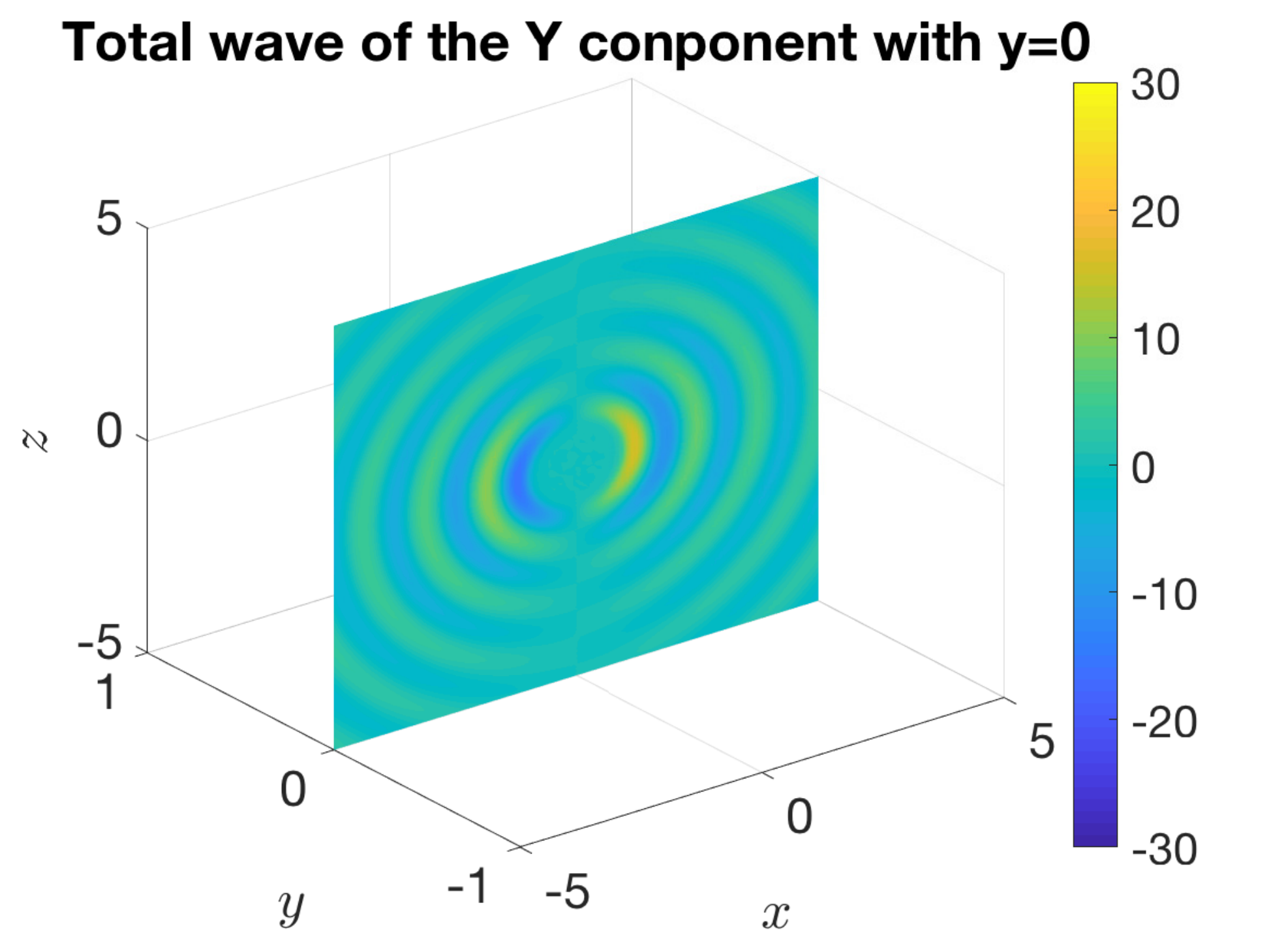}}
 {\includegraphics[width=4.4cm]{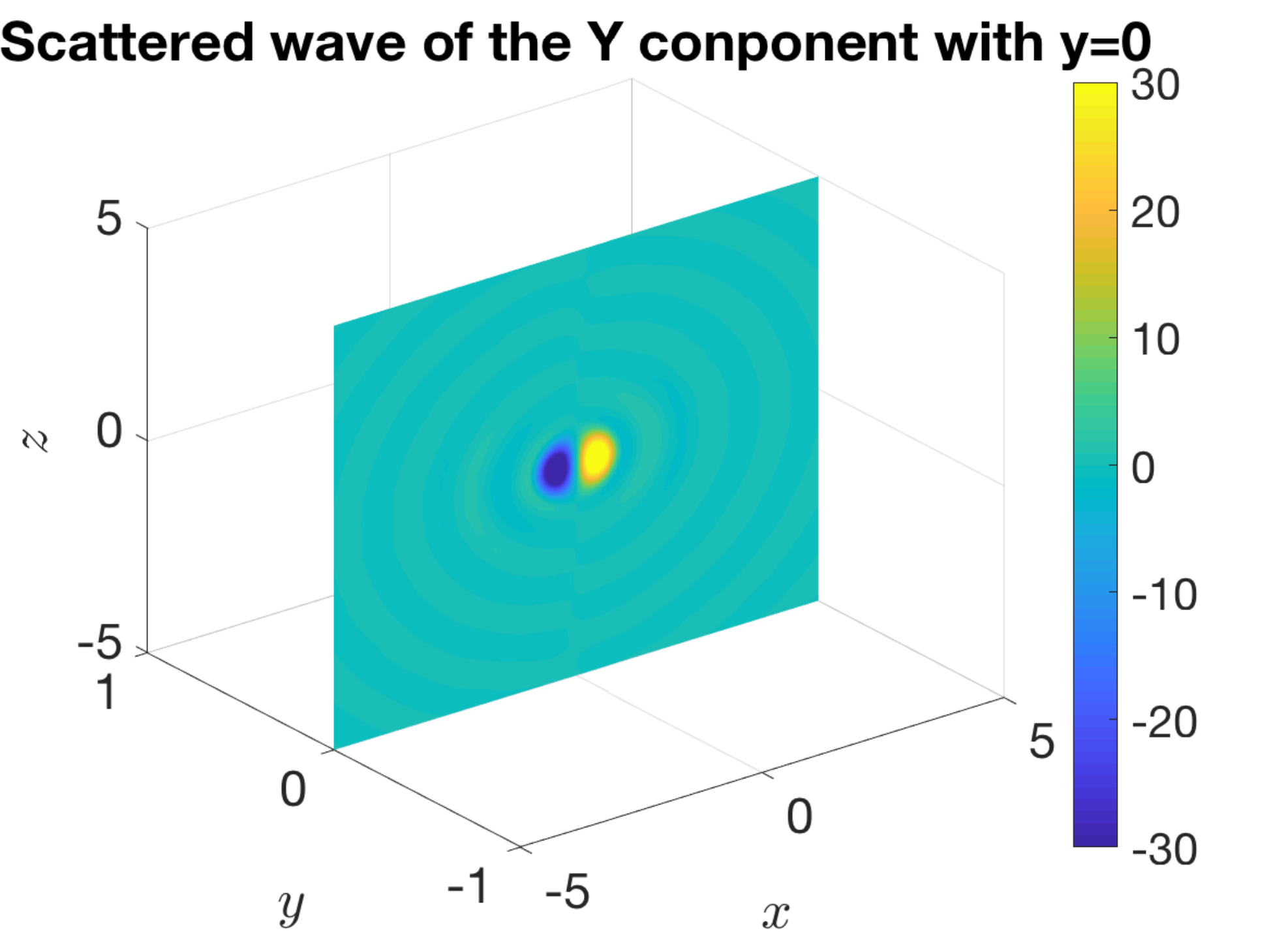}\\}
 {\includegraphics[width=4.4cm]{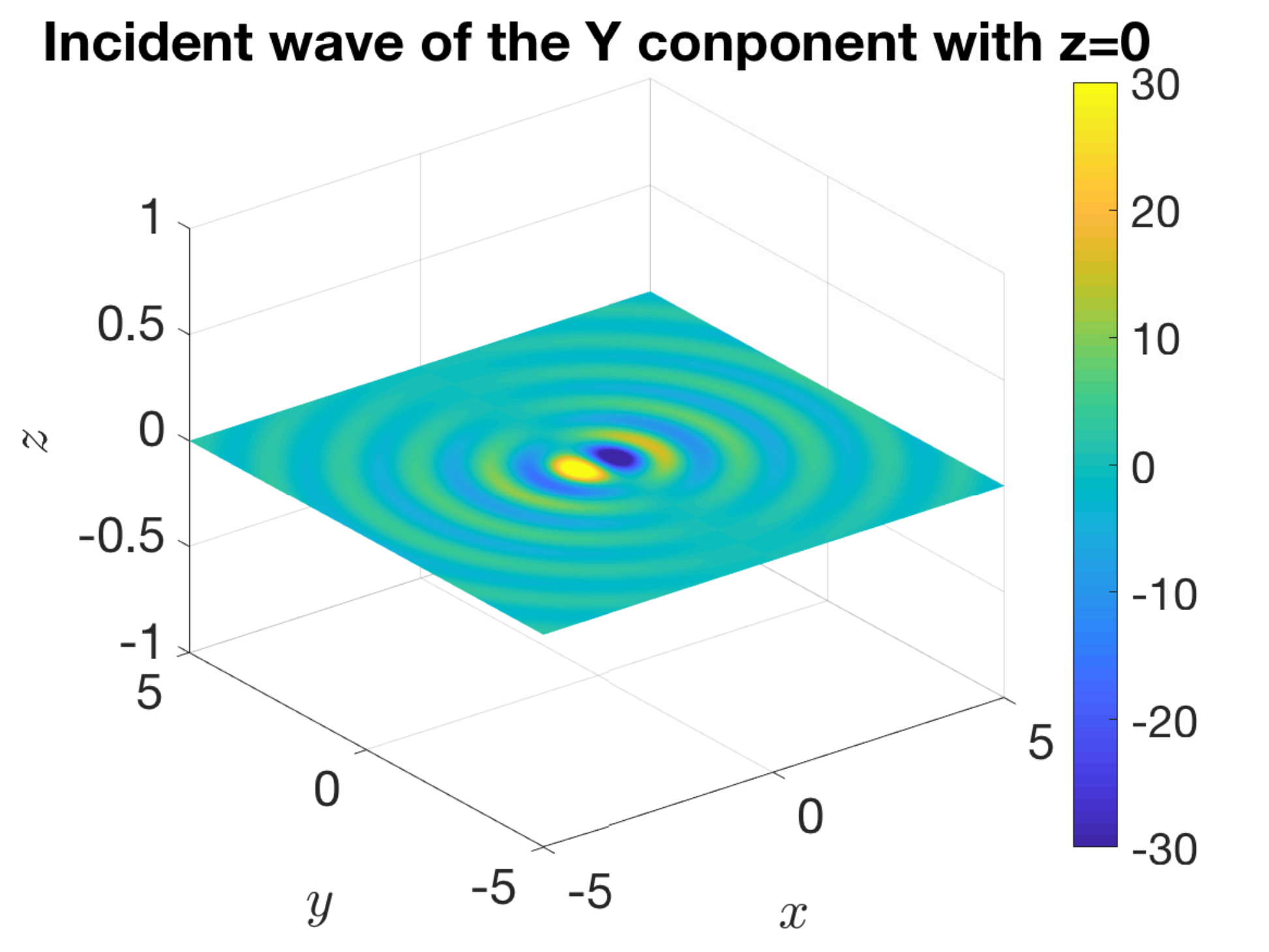}}
 {\includegraphics[width=4.4cm]{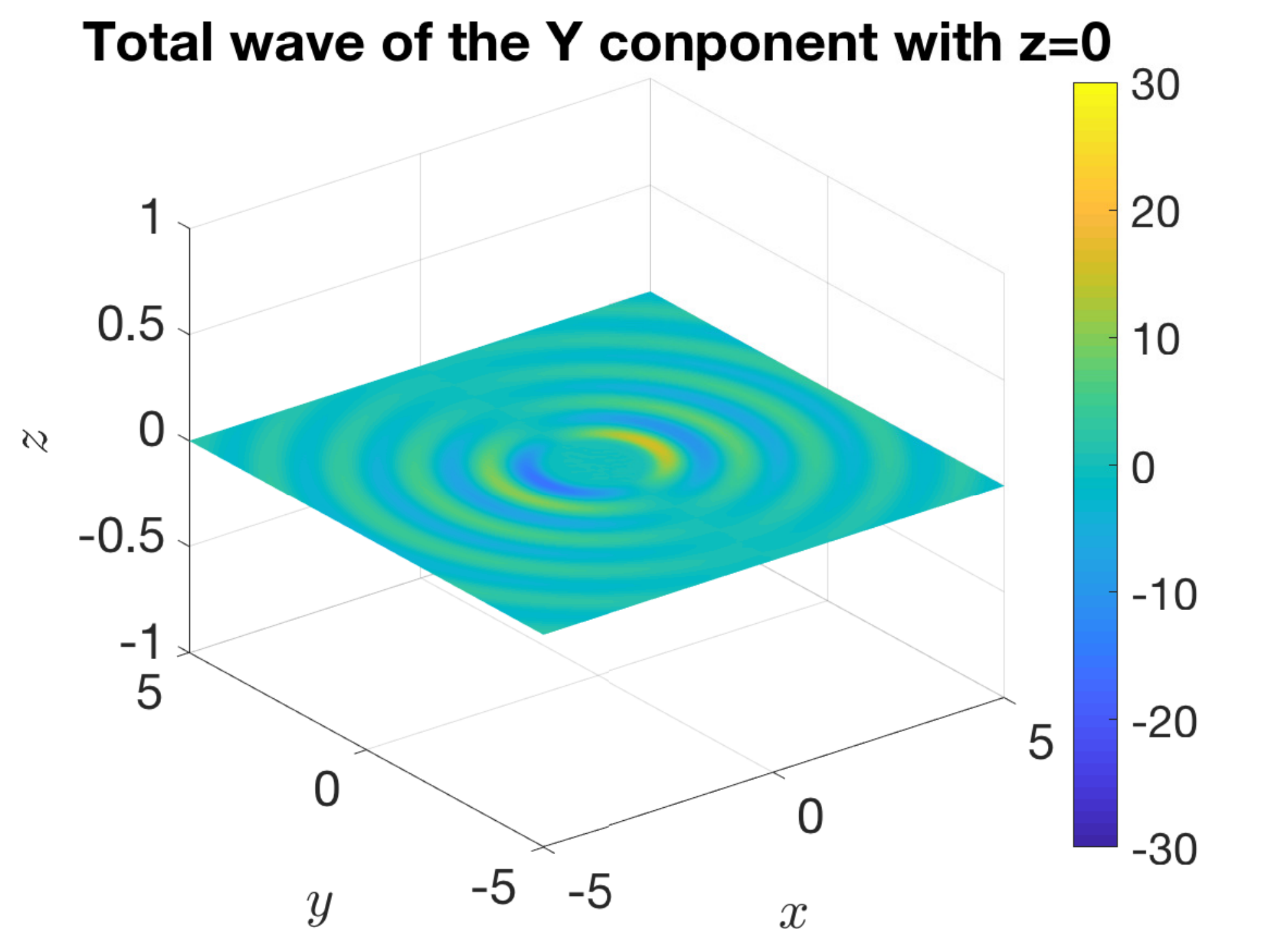}}
 {\includegraphics[width=4.4cm]{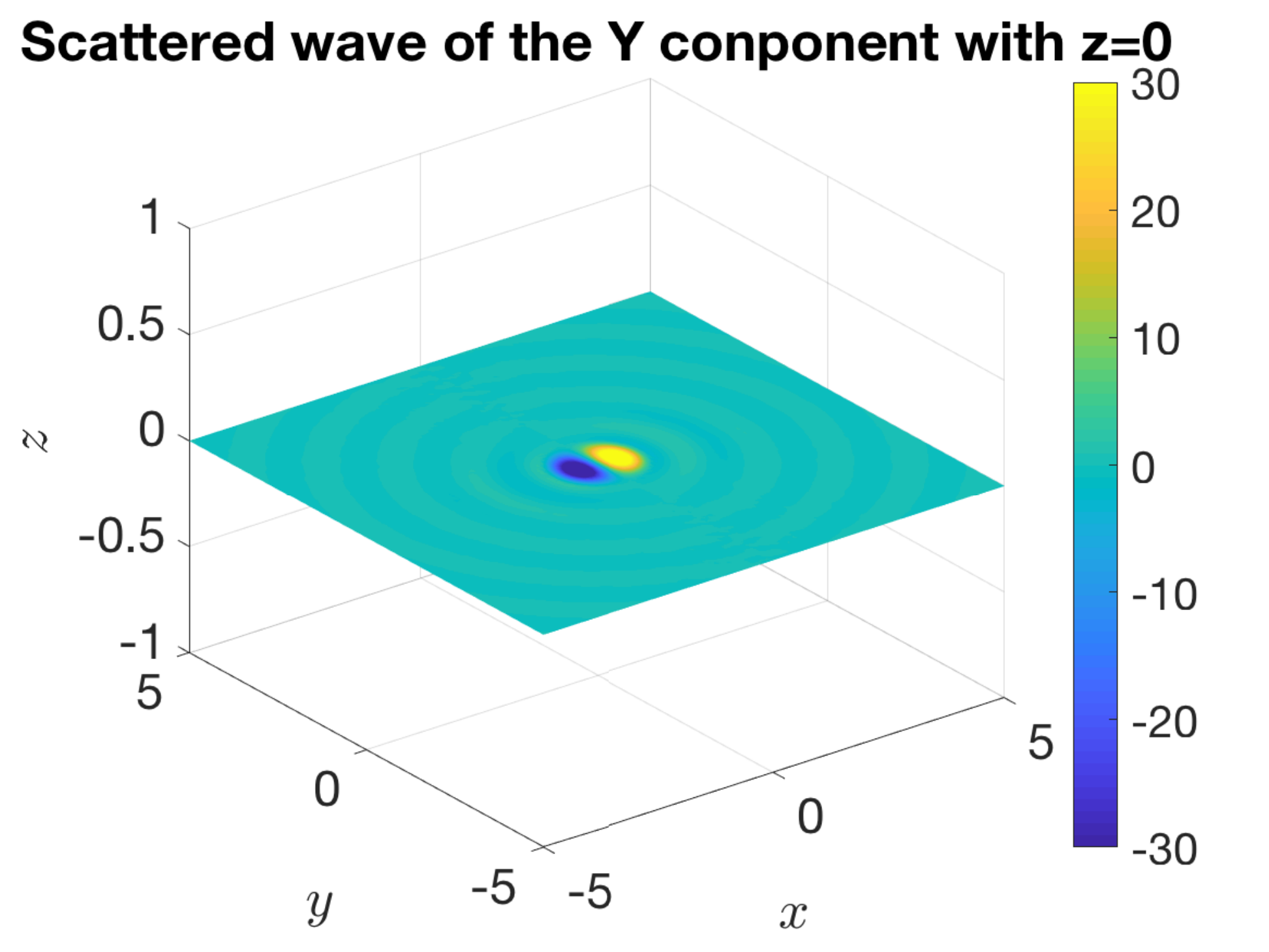}\\}
  \caption{Slice plotting of the $y$-components of the incident, total and scattered electric fields associated to the electromagnetic configuration described in \eqref{eq:gg1}, \eqref{eq:gg2} and \eqref{eq:cloaking_source}. }
\end{figure}

\section*{Acknowledgment}

The work of H. Liu was supported by the FRG and startup grants from Hong Kong Baptist University, and Hong Kong RGC General Research Funds, 12302017 and 12301218.

\end{document}